\documentclass[review]{elsarticle}

\journal{Journal of \LaTeX\ Templates}
\usepackage{amsmath, amsthm, amsfonts}
\usepackage[width=17cm,height=22cm]{geometry}
\usepackage{float}
\usepackage{graphicx,color,epsfig,epstopdf}
\newtheorem{thm}{Theorem}[section]

\newtheorem{lem}[thm]{Lemma}
\newtheorem{exm}[thm]{Example}

\theoremstyle{definition}
\newtheorem{dfn}[thm]{Definition}
\theoremstyle{remark}

\begin{document}
\begin{frontmatter}
\title{Stability and synchronization of a fractional BAM neural network system of high-order   type}
\author[mymainaddress]{Sakina Othmani \corref{mycorrespondingauthor}}
\ead{sothmani@usthb.dz}

\author[mysecondaryaddress]{Nasser-eddine Tatar}
\cortext[mycorrespondingauthor]{Corresponding author}
\ead{tatarn@kfupm.edu.sa}

\address[mymainaddress]{Laboratory of SDG, Faculty of Mathematics, University of Science and Technology Houari Boumedienne, PB 32, El-Alia 16111, Bab Ezzouar, Algiers, Algeria}
\address[mysecondaryaddress]{King Fahd University of Petroleum and Minerals,
Department of Mathematics and Statistics,
Dhahran, 31261 Saudi Arabia}

\begin{abstract}
In this paper, stability and synchronization of a Caputo fractional BAM neural network system of high-order type and neutral delays are examined.  A mixture of properties of fractional calculus, Laplace transform, and analytical techniques is used to derive Mittag-Leffler stability and synchronization for two classes of activation functions. A fractional version of Halanay inequality is utilized to deal with the fractional character of the system and some suitable evaluations and handling  to cope with the higher order feature. Another feature is the treatment of unbounded activation functions.   Explicit  examples to validate the theoretical outcomes are shown at the end.
\end{abstract}


\end{frontmatter}

\section{Introduction}

The synchronization process  involves the coherence between coupled systems over time.
 It can be in different forms: full synchronization, anti-synchronization, delayed synchronization, generalized synchronization, phase synchronization, projective synchronization or  finite time synchronization.  This  process is very effective in many  fields of technology such as  computer science \cite{A15,A16,A17,A18}.

%


One of the most efficient neural networks system is bidirectional associative memory (BAM) neural networks system  introduced by Kosko in 1987.  They take the form of recurrent neural networks and widen the single-layer self-associated Hebbian correlator.
These networks are effectively used in  signal and image processing, pattern recognition, and optimization problems. With the aim of  describing  and modelling  the dynamics of complex neural reactions, the incorporation of information about the past state derivative in systems is necessary. This kind of delay is referred to as neutral delay \cite{A12,A13,A14}. One of the negative effects of delays on systems are  oscillations, divergences, chaos and bifurcations \cite{A11}.
On the other hand, owing to the higher approximation property, faster convergence rate,  greater
storage capacity and greater fault tolerance,  higher-order neural networks are more advantageous compared to the lower-order neural networks \cite{A8,A9,A10}.
 Moreover, fractional derivatives  reflect the reliance of states on their past history, and therefore fractional models are capable of portraying adequately many complex  phenomena and processes \cite{AA3,AA2,AA1}.

In this paper, we consider the following fractional higher-order  BAM neural network with  distributed delays as well as   delays of neutral type
\begin{equation}\label{hfn1}
\left\{
\begin{array}{l}
D_{C}^{\delta}\Big[ x_{p}(t)-cx_{p}(t-\mu )\Big] =  -a_{p}x_{p}(t)
+\sum\limits_{q,s=1}^{n_{2}} d_{qps}\int_{0}^{\infty } k_{qps}(s)g_{q}( y_{q}(t-s)) ds \\
\times \int_{0}^{\infty
}h_{qps}(s) g_{s}( y_{s}(t-s)) ds + I_{p} ,\; t>0,\; p=1,...,n_{1}, \\
D_{C}^{\delta}\Big[ y_{q}(t)- \bar{c} y_{q}(t-\mu)\Big] =   -\bar{a}_{q} y_{q}(t)
+\sum\limits_{p, r =1}^{n_{1}} \bar{d}_{pqr}\int\nolimits_{0}^{%
\infty } \bar{k}_{pqr}(s) \bar{g}_{p}( x_{p}(t-s)) ds  \\
 \times \int\nolimits_{0}^{\infty
}\bar{h}_{pqr}(s)\bar{g}_{r}( x_{r}(t-s)) ds+ J_{q},\; t>0,\; q=1,...,n_{2}, \\
x_{p}(t)=\phi_{p}(t),\; t\leq 0,\;p=1,...,n_{1},  \\
y_{q}(t)=\varphi _{q}(t),\; t\leq 0,\;q=1,...,n_{2},
\end{array}
\right.
\end{equation}%
where $n_{1}+n_{2}$ is the number of neurons, $x_{p}$ and $y_{q}$
correspond to the state of the $i$-th neuron and the $j$-th neuron  at time $t$;  $a_{p}>0$ and $\bar{a}_{q}>0$ denote the dissipation
coefficients; $c>0$ and $\bar{c}>0$ represent the coefficients of the temporal derivation of the lagged states;  $d_{qps}$ and $\bar{d}_{pqr}$ account for  the second order parameters; $g_{q}$ and $\bar{g}_{p}$ are   the activation functions;
 $\mu$
 corresponds to the neutral delays; $k_{qps}, h_{qps}, \bar{k}_{pqr}$ and  $\bar{h}_{pqr}$
represent the distributed delay kernels; $I_{p}$ and $J_{q}$ refer to
external inputs; finally, $\phi_{p}$ and $\varphi _{q}$ are   the history functions of
the $p$-th and the $q$-th  state.

Numerous investigations have been carried out on fractional neural network systems.
In \cite{A21,A19,A20}, stability and synchronization of  the following fractional Hopfield NN system were examined  
\begin{equation*}
\left\{
\begin{array}{l}
_{0}D^{\gamma}_{t} w_{k}(t)= -\alpha_{k} w_{k}(t)+ \sum\limits_{l=1}^{p} \beta_{kl}h_{l}( w_{l}(t))+\chi_{k}, \; k=1, 2, ..., p,\\
w_{k}(0)=w_{k0}.
\end{array}
\right.
\end{equation*}
Both Mittag-Leffler stability of the equilibrium  and   synchronization  via linear feedback controls were proved, utilizing an extended second method of Lyapunov for Lipschitz continuous activation functions in \cite{A19}. In contrast, for discontinuous activation functions,   asymptotic stability results were investigated  in \cite{A20}, whilst Mittag-Leffler synchronization was discussed in \cite{A21} through linear feedback  controls. The authors in \cite{A21,A20} used  Fillipov theory, Laplace transform technique and fractional differential inequalities.
Furthermore,
fractional-order bidirectional associative memory (BAM) neural networks  with delays were treated in \cite{A25,A22,AA30,A30,A24,A27,A26,A28,A23,A29}. 
In \cite{A22}, sufficient conditions were established ensuring  the  asymptotic stability for  a fractional BAM NN with leakage delays. Yang et al.  \cite{A23} discussed the asymptotic stability of a fractional BAM NN with discrete delays via a fractional inequality for Lipschitz continuous activation functions.  The following fractional BAM NN with delays depending of time  was studied in \cite{A24}
\begin{equation*}
\left\{
\begin{array}{l}
{}^{c}_{0}D^{\gamma}_{t} z_{k}(t)= -\alpha_{k} z_{k}(t)+ \sum\limits_{l=1}^{q} \beta_{kl} h_{l}( w_{l}(t))+ \sum\limits_{l=1}^{q} \delta_{kl} h_{l}(w_{l}(t- \mu_{kl}(t)))+\chi_{k}, \; k=1, 2, ..., p,\\
{}^{c}_{0}D^{\gamma}_{t} w_{l}(t)= -\bar{\alpha}_{l} w_{l}(t)+ \sum\limits_{k=1}^{p} \bar{\beta}_{lk}\bar{h}_{k}(z_{k}(t))+\sum\limits_{k=1}^{p} \bar{\beta}_{lk} \bar{h}_{k}(z_{k}(t- \nu_{lk}(t))+\bar{\chi}_{l}, \; l=1, 2, ..., q,
\end{array}
\right.
\end{equation*}
where the asymptotic stability of the stationary state was achieved  through a fractional inequality and Lyapunov functionals.  In \cite{A25}, Cao and Bai proved the stability for a fractional BAM NN with distributed delays using the  M-matrix theory, Laplace transform and some inequalities such as   Gronwall inequality. Besides, the stability for fractional BAM NN with delays and impulses were studied in \cite{A27,A26}, considering Lipschitz continuous activation functions and linear impulsive operators.  Moreover, in \cite{A28}, a new generalized Gronwall inequality was proved to examine the stability of Mittag-Leffler kind for  a fractional BAM NN model, whilst,  Mittag-Leffler synchronization results through delayed controls were shown in \cite{A29}.

In this paper, we examine the  Mittag-Leffler stability and synchronization for both bounded and unbounded activation functions.  Apart from H\"{o}lder continuous functions, no works have been reported on unbounded activation functions so far.
For this aim we shall utilize, in addition to these techniques, a new Halanay inequality of fractional order with neutral delay and distributed delay for a wide family of delay kernels defined by an integral condition. The unbounded case is more problematic as it requires more skills. In particular, it requires appropriate manipulations and adequate estimations.




This paper is arranged as follows: In Section 2, we present some notation, definitions, and lemmas. Mittag-Leffler stability of the equilibrium is  discussed for the bounded and unbounded  cases in Section 3 and  Section 4, respectively. The synchronization result  is examined in Section 5. Finally, numerical illustrations are provided  to confirm the findings in Section 6.

\section{Preliminaries}
This section contains  some specific  assumptions, definitions of  fractional derivatives and lemmas.
 For the sake of brevity, we shall omit the ranges of indexes. In all our assumptions and statements, it is understood that the indexes
 $p, r$ and $q, s$ range from $1$ to $n_{1}$ and from $1$ to $n_{2}$, respectively.

\textbf{(A1)} The delay kernel functions $k_{qps}, h_{qps}, \bar{k}_{pqr}$ and $\bar{h}_{pqr}$ are
piecewise continuous and  non-negative such that
\begin{equation*}
\begin{array}{c}
  \hat{k}_{qps}=\int\nolimits_{0}^{\infty } k_{qps} ( s ) ds < \infty ,\;
\hat{h}_{qps}=\int\nolimits_{0}^{\infty }h_{qps}( s ) ds < \infty ,\\
  \hat{\bar{k}}_{pqr}=\int\nolimits_{0}^{\infty } k_{pqr} ( s ) ds < \infty ,\;
\hat{\bar{h}}_{pqr}=\int\nolimits_{0}^{\infty }\bar{h}_{pqr}( s ) ds < \infty.
\end{array}
\end{equation*}

\textbf{(A2)} The functions $ g_{q}$  and $\bar{g}_{p}$  satisfy for some constants $G, \bar{G}>0$
\begin{equation*}
  | g_{q}(x) | \leq G , \;
 | \bar{g}_{p}(x) | \leq \bar{G} ,\; x \in \mathbb{R}.
\end{equation*}%

\textbf{(A3) }The functions $ g_{q}$ and $\bar{g}_{p}$ are Lipschitz continuous
on $\mathbb{R}$ with Lipschitz constants $L_{q}$ and $M_{p}$ such that
\begin{equation*}
\begin{array}{c}
|g_{q}(x)-g_{q}(y)|\leq L_{q}|x-y|, \quad
|\bar{g}_{p}(x)-\bar{g}_{p}(y)|\leq M_{p}|x-y|, \;\forall x,y\in \mathbb{R}.
\end{array}
\end{equation*}

\begin{dfn}
The point  $(x_{p}^{*}, y_{q}^{*})$ is an equilibrium of the system (\ref{hfn1}), if
\begin{equation*}
\begin{array}{c}
0=-a_{p}x_{p}^{*}
+\sum\limits_{q,s=1}^{n_{2}} d_{qps} \hat{k}_{qps}g_{q}(
y_{q}^{* }) \hat{h}_{qps} g_{s}( y_{s}^{*})
+I_{p}, \\
0=-\bar{a}_{q} y_{q}^{*}
+\sum\limits_{p, r =1}^{n_{1}}\bar{d}_{pqr} \hat{\bar{k}}_{pqr} \bar{g}_{p}(
x_{p}^{* }) \hat{\bar{h}}_{pqr} \bar{g}_{r}( x_{r}^{*})
+J_{q}.
\end{array}
\end{equation*}%
\end{dfn}
According to the uniform boundedness \textbf{(A2) }and Lipschitz continuity \textbf{(A3)}, there exists  a unique equilibrium $(x_{p}^{*}, y_{q}^{*})$ \cite{AA30}.


\begin{dfn}
For any measurable function $f$, the Riemann-Liouville fractional integral of order $%
\delta>0$ is equal to
\begin{equation*}
I^{\delta}f(t)=\frac{1}{\Gamma (\delta )}\int\limits_{0}^{t}(t-s)^{\delta
-1}f(s)ds,\; \delta >0
\end{equation*}%
 if the integral term can be found.
Notice that  $\Gamma (\delta )$\ is the  Gamma function.
\end{dfn}
\begin{dfn}  The Caputo fractional derivative of order $\delta $ is given  by
\begin{equation*}
D_{C}^{\delta}f(z)=\frac{1}{\Gamma (1-\delta)}\int\limits_{0}^{z}(z-w)^{-\delta }f^{\prime }(w)dw,\;0<\delta<1
\end{equation*}%
 if the integral term can be found.\\
The one-parametric and two-parametric Mittag-Leffler functions are expressed  by%
\begin{equation*}
E_{\delta }(w):=\sum\limits_{\kappa=0}^{\infty }\frac{w^{\kappa}}{\Gamma (\delta
\kappa+1)},\; Re(\delta )>0,
\end{equation*}%
and%
\begin{equation*}
E_{\delta ,\rho }(w):=\sum\limits_{\kappa=0}^{\infty }\frac{w^{\kappa}}{\Gamma
(\delta \kappa+\rho )},\; Re(\delta )>0,\; Re(\rho)>0,
\end{equation*}%
respectively, with the remark that  $E_{\delta ,1}(w)\equiv E_{\delta }(w)$.
\end{dfn}
\begin{lem}\cite{A31}
For $\sigma, \gamma, \beta >0$, we have
\begin{equation}\label{formula1}
I^{\sigma} t^{\gamma-1} E_{\alpha, \beta} (c t^{\beta})(x)= x^{\sigma + \gamma-1} E_{\beta, \sigma+\gamma}(c x^{\beta}).
\end{equation}
\end{lem}
\textbf{Mainardi's conjecture:} \cite{A32} For any $t>0$ and
fixed $\delta,$ $0<\delta <1,$ we get
\begin{equation} \label{mc1}
\frac{1}{1+c\Gamma (1-\delta )t^{\delta}}\leq E_{\delta }(-ct^{\delta
})\leq \frac{1}{1+c\Gamma (1+\delta)^{-1}t^{\delta}},\;t\geq 0.
\end{equation}%
This conjecture has been proved in  \cite{A33,A34}.

\begin{dfn} The  solution $u(t)$\
is  globally $\delta $-Mittag-Leffler stable ($0<\delta<1$) if for some  constants $\Lambda, \lambda >0$
\begin{equation*}
\left\Vert u(t)\right\Vert \leq \Lambda E_{\delta}(-\lambda t^{\delta}),\;t>0
\end{equation*}%
for a prescribed  norm $\left\Vert .\right\Vert.$  A local $\delta$-Mittag-Leffler
stability is provided for small data.
%
\end{dfn}

The result below has been proved in \cite{A35}. We report it here with its proof for self-containedness.
\begin{lem}\cite{A35}
Assuming  that  $y(t)$ is a solution of
\begin{equation}\label{st1}
\left\{
\begin{array}{l}
D_{C}^{\gamma} \Big[y(t)-c y(t-\mu)\Big] \leq - r y(t) +\int_{0}^{\infty} h(s) y(t-s) ds, \; 0< \gamma < 1, \; t, c, \mu>0, \\
y(t)= \phi(t)\geq 0, \; t \leq 0,
\end{array}
\right.
\end{equation}
with $r>0$ and $h$ is a nonnegative summable function. If $c>0$ and $h$ are such as
$$\int_{0}^{t} (t-s)^{\gamma-1} E_{\gamma, \gamma}(-r(t-s)^{\gamma}) \Big(  \int_{-\infty}^{s} E_{\gamma}(- r \lambda^{\gamma}) h(s-\lambda) d \lambda \Big) ds \leq M E_{\gamma}(-r t^{\gamma}), \; t>0$$
hold for some $M>0$ with
$$M < 1 - \Big[  1+ \Gamma(1+\gamma) \Gamma(1-\gamma) \Big]Vc, \;  \Big[  1+ \Gamma(1+\gamma) \Gamma(1-\gamma) \Big]Vc< 1, \; V:= \frac{1}{r \mu^{\gamma}} +2^{\gamma} \Gamma(1-\gamma).$$
Then, for some  constant $\Lambda>0$
$$y(t) \leq \Lambda E_{\gamma}(-rt^{\gamma}), \; t>0. $$
\end{lem}
\begin{proof} Let $|\phi(s) | < y_{0} E_{\gamma}(-r (s+\mu)^{\gamma}), \; s\in [- \mu, 0], \; y_{0}>0.$ It is obvious that for $0< c <1$, we obtain the expression
\begin{equation*}
\begin{array}{c}
y(t) - c y(t-\mu) = E_{\gamma}(-r t^{\gamma}) \Big[  \phi(0)- c \phi(-\mu) \Big]\\ +\int_{0}^{t} (t-s)^{\gamma-1} E_{\gamma, \gamma}(-r (t-s)^{\gamma})\Big(  -rc y(s -\mu)+\int_{0}^{\infty} h(\sigma) y(s-\sigma) d\sigma \Big) ds, \; t>0.
\end{array}
\end{equation*}
Then
\begin{equation}\label{sm1}
\begin{array}{c}
|y(t)| \leq 2 y_{0} E_{\gamma}(-r t^{\gamma}) +c |y(t-\mu) |+ r c  \int_{0}^{t} (t-s)^{\gamma-1} E_{\gamma, \gamma}(-r (t-s)^{\gamma}) |y(s-\mu) | ds \\+\int_{0}^{t} (t-s)^{\gamma-1} E_{\gamma, \gamma}(-r (t-s)^{\gamma})\Big( \int_{0}^{\infty} h(\sigma)|y(s-\sigma) | d \sigma\Big) ds, \ t>0.
\end{array}
\end{equation}
For $t\in [0, \mu],$
\begin{equation}\label{smm1}
\begin{array}{c}
\frac{|y(t)|}{E_{\gamma}(-rt^{\gamma})} \leq 3 y_{0}  + \frac{r c y_{0}}{E_{\gamma}(-rt^{\gamma})} \int_{0}^{t} (t-s)^{\gamma-1} E_{\gamma, \gamma}(-r (t-s)^{\gamma}) E_{\gamma}(-r s^{\gamma}) ds +M \sup\limits_{-\infty < \sigma\leq t } \frac{|y(\sigma) |}{E_{\gamma}(-r \sigma^{\gamma})}.
\end{array}
\end{equation}
Again, as
\begin{equation}\label{sm2}
\begin{array}{c}
 \int_{0}^{t} (t-s)^{\gamma-1} E_{\gamma, \gamma}(-r (t-s)^{\gamma}) E_{\gamma}(-r s^{\gamma}) ds\\
 \leq \frac{\Gamma(1+\gamma)}{r}  \int_{0}^{t} (t-s)^{\gamma-1} E_{\gamma, \gamma}(-r (t-s)^{\gamma}) s^{-\gamma}ds\\
 \leq \frac{\Gamma(1+\gamma)\Gamma(1-\gamma)}{r}E_{\gamma,1}(-r t^{\gamma}),
\end{array}
\end{equation}
we may write
$$\frac{|y(t)|}{E_{\gamma}(-rt^{\gamma})} \leq 3 y_{0}  + c y_{0} \Gamma(1+\gamma)\Gamma(1-\gamma) +M \sup\limits_{-\infty < \sigma\leq t } \frac{|y(\sigma) |}{E_{\gamma}(-r \sigma^{\gamma})}$$
or
\begin{equation}\label{sm3}
(1-M)\frac{|y(t)|}{E_{\gamma}(-rt^{\gamma})} \leq \Big[ 3 + c  \Gamma(1+\gamma)\Gamma(1-\gamma)\Big]y_{0}.
\end{equation}
In case $t\in[\mu, 2 \mu ]$, we first notice that
\begin{equation}\label{sm6}
\begin{array}{c}
|y(t-\mu)| \leq \frac{ 3 + c  \Gamma(1+\gamma)\Gamma(1-\gamma)}{1-M}\frac{E_{\gamma}(-r (t-\mu)^{\gamma})}{E_{\gamma}(-r t^{\gamma})} y_{0} E_{\gamma}(-rt^{\gamma})\\
\leq \frac{ 3 + c  \Gamma(1+\gamma)\Gamma(1-\gamma)}{1-M} B y_{0} E_{\gamma}(-r t^{\gamma})
\end{array}
\end{equation}
where
\begin{equation}\label{st8}
\begin{array}{c}
\frac{E_{\gamma}(-r (t-\mu)^{\gamma})}{E_{\gamma}(-r t^{\gamma})} \leq \frac{1}{E_{\gamma}(-r t^{\gamma})}\leq \frac{1}{E_{\gamma}(-r (2\mu)^{\gamma})}\leq 1+ r \Gamma(1-\gamma) (2\mu)^{\gamma}:= B.
\end{array}
\end{equation}
Using the relations (\ref{sm1}) and (\ref{sm6}), we entail
\begin{equation*}
\begin{array}{c}
|y(t)| \leq 2 y_{0} E_{\gamma}(-r t^{\gamma}) + c B y_{0} \frac{ 3 + c \Gamma(1+\gamma)\Gamma(1-\gamma)}{1-M} E_{\gamma}(-r t^{\gamma}) \\+ r c B y_{0} \frac{ 3 + c  \Gamma(1+\gamma)\Gamma(1-\gamma)}{1-M}  \int_{0}^{t} (t-s)^{\gamma-1} E_{\gamma, \gamma}(-r (t-s)^{\gamma}) E_{\gamma}(-r s^{\gamma}) ds\\
 +\int_{0}^{t} (t-s)^{\gamma-1} E_{\gamma, \gamma}(-r (t-s)^{\gamma})\Big( \int_{0}^{\infty} h(\sigma)|y(s-\sigma) | d \sigma\Big) ds
\end{array}
\end{equation*}
Next, we apply (\ref{sm2}), to obtain
\begin{equation*}
\begin{array}{c}
|y(t)| \leq 2 y_{0} E_{\gamma}(-r t^{\gamma}) +c B y_{0}\frac{3+c \Gamma(1+\gamma)\Gamma(1-\gamma)}{1-M} E_{\gamma}(-r t^{\gamma}) + r c B y_{0} \frac{ 3 + c  \Gamma(1+\gamma)\Gamma(1-\gamma)}{1-M} \frac{  \Gamma(1+\gamma)\Gamma(1-\gamma)}{r} E_{\gamma}(-r t^{\gamma})\\
 +\int_{0}^{t} (t-s)^{\gamma-1} E_{\gamma, \gamma}(-r (t-s)^{\gamma})\Big( \int_{0}^{\infty} h(\sigma)|y(s-\sigma) | d \sigma\Big) ds
\end{array}
\end{equation*}
or
\begin{equation}\label{sm4}
\begin{array}{c}
(1-M)\frac{|y(t)|}{E_{\gamma}(-r t^{\gamma})} \leq 2 y_{0}+ c B y_{0}\Big[1+\Gamma(1+\gamma)\Gamma(1-\gamma)\Big]
 \frac{3 +c\Gamma(1+\gamma)\Gamma(1-\gamma)}{1-M}\\
 \leq 2y_{0}+3 B y_{0}\frac{[1+\Gamma(1+\gamma)\Gamma(1-\gamma)]}{1-M}c+3B y_{0} \frac{[1+\Gamma(1+\gamma)\Gamma(1-\gamma)]^{2}}{1-M}c^{2}.
\end{array}
\end{equation}
For $t\in [2 \mu, 3\mu]$, in view of the estimation $\frac{t^{\gamma}}{(t-\mu)^{\gamma}}\leq 2^{\gamma}$ and the relations (\ref{mc1}), we obtain
\begin{equation}\label{st110}
\begin{array}{c}
\frac{E_{\gamma}(-r (t-\mu)^{\gamma})}{ E_{\gamma}(-r t^{\gamma})} \leq \frac{1+ r \Gamma(1-\gamma) t^{\gamma}}{1+ r \Gamma(1+\gamma)^{-1}(t-\mu)^{\gamma}}\leq \frac{\Gamma(1+\gamma)}{r \mu^{\gamma}} + \frac{\Gamma(1+\gamma)\Gamma(1-\gamma)t^{\gamma}}{(t-\mu)^{\gamma}}\\
\leq \Gamma(1+\gamma) \Big[\frac{1}{r \mu^{\gamma}}  +2^{\gamma} \Gamma(1-\gamma) \Big]\\
\leq \frac{1}{r \mu^{\gamma}}+2^{\gamma}\Gamma(1-\gamma)=:V>1.
\end{array}
\end{equation}
Notice that (\ref{st110}) holds for all $t \geq 2 \mu$ and, as $\Gamma(1+ \gamma)$ is very close to (and below) than $1$, we may ignore it.

Therefore (\ref{sm1}) implies
\begin{equation*}
\begin{array}{c}
|y(t)| \leq 2 y_{0} E_{\gamma}(-r t^{\gamma}) + \frac{c V}{1-M}\Big[  2 y_{0}  + 3 B y_{0}Wc + 3 B y_{0} W^{2}(1-M)c^{2}\Big]
E_{\gamma}(-r t^{\gamma})\\ + cV \frac{  \Gamma(1+\gamma)\Gamma(1-\gamma)}{1-M} \Big[ 2 y_{0}  + 3 B y_{0}W c + 3 B y_{0} W^{2}(1-M)c^{2} \Big]E_{\gamma}(-r t^{\gamma})\\
 +\int_{0}^{t} (t-s)^{\gamma-1} E_{\gamma, \gamma}(-r (t-s)^{\gamma})\Big( \int_{0}^{\infty} h(\sigma)|y(s-\sigma) | d \sigma\Big) ds
\end{array}
\end{equation*}
where
\begin{equation}\label{w1}
W:=\frac{1+\Gamma(1+\gamma)\Gamma(1-\gamma)}{1-M}.
\end{equation}
So
\begin{equation*}
\begin{array}{c}
(1-M)\frac{|y(t)|}{E_{\gamma}(-r t^{\gamma})}
 \leq 2 y_{0}  +\frac{cV}{1-M}\Big[ 2y_{0}+
 3 B y_{0} Wc+ 3 B y_{0} W ^{2}(1-M)c^{2}\Big]\\
 \times [1+\Gamma(1+\gamma)\Gamma(1-\gamma)\Big]
\end{array}
\end{equation*}
or
\begin{equation*}
\begin{array}{c}
(1-M)\frac{|y(t)|}{E_{\gamma}(-r t^{\gamma})}
 \leq 2 y_{0}  +c V W\Big[ 2y_{0}+
 3 B y_{0} Wc+ 3 B y_{0} W ^{2}(1-M)c^{2}\Big].
\end{array}
\end{equation*}
That is
\begin{equation}\label{sm7}
\begin{array}{c}
(1-M)\frac{|y(t)|}{E_{\gamma}(-r t^{\gamma})}
 \leq 3 B  y_{0} \Big[ 1+
 V Wc+ (V W c)^{2}+(V W c)^{3}\Big].
\end{array}
\end{equation}
\textbf{Claim}: We have
\begin{equation}\label{sm8}
\begin{array}{c}
(1-M)\frac{|y(t)|}{E_{\gamma}(-r t^{\gamma})}
 \leq 3 B  y_{0} \sum\limits_{l=0}^{k} (V W c)^{l}, \; t\in [(k-1)\mu, k\mu].
\end{array}
\end{equation}
By (\ref{sm3}), (\ref{sm4}) and (\ref{sm7}), the claim is valid  for $n=1, 2$ and $3$, resp. Assume that it is true  on $[(k-1)\mu, k\mu]$. Let $t\in[k\mu, (k+1)\mu]$, then by the relations  (\ref{sm1}), (\ref{sm2}) and (\ref{sm8}), it is obvious that
\begin{equation*}
\begin{array}{c}
|y(t)| \leq 2 y_{0} E_{\gamma}(-r t^{\gamma}) + 3 B c V \frac{y_{0}}{1-M} \sum\limits_{l=0}^{k} (V W c)^{k} E_{\gamma}(-r t^{\gamma})\\
+3 B c V \frac{y_{0}}{1-M}  \Gamma(1+\gamma)\Gamma(1-\gamma) \sum\limits_{l=0}^{k} (V W c)^{l} E_{\gamma}(-r t^{\gamma})\\+\int_{0}^{t} (t-s)^{\gamma-1} E_{\gamma, \gamma}(-r (t-s)^{\gamma})\Big( \int_{0}^{\infty} h(\sigma)|y(s-\sigma) | d \sigma\Big) ds
\end{array}
\end{equation*}
or
$$(1-M)\frac{|y(t)|}{E_{\gamma}(-r t^{\gamma})}
 \leq 2y_{0}+ 3 B c V  \frac{y_{0}}{1-M} \Big[ 1+ \Gamma(1+\gamma)\Gamma(1-\gamma)\Big] \sum\limits_{l=0}^{k} (V W c)^{l}.$$
 In view of (\ref{w1}), we deduce that
 $$(1-M)\frac{|y(t)|}{E_{\gamma}(-r t^{\gamma})}
 \leq 3 B  y_{0} \Big\{ 1+  c V W \sum\limits_{l=0}^{k} (V W c)^{l}\Big\} = 3 B y_{0} \sum\limits_{l=0}^{k+1} (V W c)^{l}$$
 and (\ref{sm8}) is fulfilled. If $VW c < 1,$ then $\sum\limits_{l=0}^{+\infty} (V W c)^{l}$ is convergent.
\end{proof}
Two classes of kernels satisfying the assumptions in this lemma are provided in \cite{A35}. Namely, we shall endorse the following notation in the remainder  of this paper $\sum\limits_{p,r;q,s=1}^{n_{1};n_{2}}= \sum\limits_{p,r=1}^{n_{1}}\sum\limits_{q,s=1}^{n_{2}}.$
\section{Bounded activation functions}
In this section,  the  Mittag-Leffler stability of system (\ref{hfn1}) will be discussed using  the Laplace transform and fractional  calculus. The boundedness of the activation functions  converts the non-linearities  resulting from  the higher-order terms to linear expressions. We shall study the stability of the origin point of the following  transformed system
\begin{equation}\label{t1}
\begin{array}{c}
 D_{C}^{\delta}\Big[ u_{p}(t)-cu_{p}(t-\mu )\Big] =
 -a_{p} u_{p}(t)
+\sum\limits_{q,s=1}^{n_{2}} d_{qps} \int\nolimits_{0}^{\infty
} k_{qps}(\omega) g_{q}( y_{q}(t-\omega)) d\omega \\ \times  \int\nolimits_{0}^{\infty
} h_{qps}(\omega) g_{s}( y_{s}(t-\omega)) d\omega
-  \sum\limits_{q,s=1}^{n_{2}} d_{qps} \hat{k}_{qps} g_{q}( y_{q}^{\ast
}) \hat{h}_{qps} g_{s}(y_{s}^{\ast }),\; t>0,
  \end{array}
  \end{equation}
  \begin{equation}\label{t2}
  \begin{array}{c}
D_{C}^{\delta}\Big[ v_{q}(t)-\bar{c}v_{q}(t- \mu )\Big] = -\bar{a}_{q} v_{q}(t)
 +\sum\limits_{p,r=1}^{n_{1}} \bar{d}_{pqr} \int\nolimits_{0}^{\infty
} \bar{k}_{pqr}(\omega) \bar{g}_{p}( x_{p}(t-\omega)) d\omega\\ \times
 \int\nolimits_{0}^{\infty
} \bar{h}_{pqr}(\omega) \bar{g}_{r}( x_{r}(t-\omega)) d\omega
 -  \sum\limits_{p,r=1}^{n_{1}} \bar{d}_{pqr} \hat{\bar{k}}_{pqr} \bar{g}_{p}( x_{p}^{\ast
}) \hat{\bar{h}}_{pqr} \bar{g}_{r}(x_{r}^{\ast }), \; t>0,
\end{array}
\end{equation}
where
\begin{equation*}
\begin{array}{c}
x_{p}(t)= u_{p}(t)+ x_{p}^{*}, \; y_{q}(t)= v_{q}(t) +y_{q}^{*}, \\
u_{p}(t)  = \tilde{\phi}_{p}(t) = \phi_{p}(t) - x_{p}^{*}, \;
v_{q}(t) = \tilde{\varphi}_{q}(t) = \varphi_{q}(t) - y_{q}^{*}, \; t \leq 0.
\end{array}
\end{equation*}
Notice  that this  is equivalent to the stability of the equilibrium point of the system (\ref{hfn1}).\\

 We shall adopt the notation below
\begin{equation*}
\begin{array}{c}
  F := \frac{1}{\xi \mu^{\delta}}+2^{\delta} \Gamma(1-\delta),\;
  a= \min\limits_{1\leq p \leq n_{1}}\{   a_{p} \}, \; \bar{a}= \min\limits_{1\leq q \leq n_{2}}\{   \bar{a}_{q} \}, \\ \xi= \min\{ a, \bar{a} \}, \; c^{*}= \max\{ c, \bar{c}\}, \; a^{*}=\max\{ a, \bar{a} \},\\
  K(t)=\max\Bigg\{\sum\limits_{p;q,s=1}^{n_{1};n_{2}}  G L_{q} \Big[ d_{qps}%
\hat{h}_{qps} k_{qps}(t)+ d_{spq} \hat{k}_{spq} h_{spq}(t )\Big], \\  \sum\limits_{p,r;q=1}^{n_{1};n_{2}}  \bar{G} M_{p}\Big[ \bar{d}_{pqr}
\hat{\bar{h}}_{pqr} \bar{k}_{pqr}(t )+ \bar{d}_{rqp} \hat{\bar{k}}_{rqp} \bar{h}_{rqp}(t )\Big] \Bigg\}.
\end{array}
\end{equation*}
\textbf{(A4)} Let $\Omega>0, \; 0<c^{*}<1$ be constants such that  $ \Omega < 1- \big[1+ \Gamma(1+\delta) \Gamma(1-\delta)   \big] F \Big(\frac{a^{*}c^{*}}{\xi}\Big), \; \big[1+ \Gamma(1+\delta) \Gamma(1-\delta)   \big] F \Big(\frac{a^{*}c^{*}}{\xi}\Big) < 1,$
\begin{equation}\label{k1}
  \int_{0}^{t} (t-\omega)^{\delta-1} E_{\delta,\delta} (-\xi(t-\omega)^{\delta}) \Big( \int_{-\infty}^{\omega} E_{\delta} (- \xi \lambda^{\delta}) K(\omega-\lambda)  d\lambda \Big) d\omega
   \leq  \Omega E_{\delta} (-  \xi t^{\delta}), \; t> 0.
  \end{equation}
\begin{thm}\label{th1}
Assume that \textbf{(A1)-(A4)}
hold. Then, for some constants $C, \xi> 0$
$$ u(t) \leq C E_{\delta} (-\xi t^{\delta}), \quad  v(t) \leq C E_{\delta} (-\xi t^{\delta}),    \quad t\geq 0. $$
where $u(t)=\sum\limits_{p=1}^{n_{1}} | u_{p}(t) | $ and $v(t)=\sum\limits_{q=1}^{n_{2}} | v_{q}(t) |$.
\end{thm}
\begin{proof}

The non-linear terms may be expressed as
\begin{equation}\label{b1}
\begin{array}{c}
\sum\limits_{q,s=1}^{n_{2}} d_{qps}\int\nolimits_{0}^{\infty
}k_{qps}(\omega) g_{q} ( y_{q}(t-\omega)) d\omega \int\nolimits_{0}^{\infty
}h_{qps}(\omega) g_{s}( y_{s}(t-\omega)) d\omega
-\sum\limits_{q,s=1}^{n_{2}} d_{qps} \hat{k}_{qps} g_{q}( y_{q}^{\ast
}) \hat{h}_{qps} g_{s}( y_{s}^{\ast }) \\
 =\sum\limits_{q,s=1}^{n_{2}}\int\nolimits_{0}^{\infty }\Bigg\{ \Big[
d_{qps}\Big( \int\nolimits_{0}^{\infty } h_{qps}(\omega) g_{s}(
y_{s}(t-\omega)) d\omega \Big) k_{qps}(\omega)+ d_{spq} \hat{k}_{spq}g_{s}(
y_{s}^{\ast }) h_{spq}(\omega)\Big]   \tilde{g}_{q}(v_{q}(t-\omega))
 \Bigg\} d\omega,
\end{array}
\end{equation}
and
\begin{equation}\label{b2}
\begin{array}{c}
 \sum\limits_{p,r=1}^{n_{1}} \bar{d}_{pqr}\int\nolimits_{0}^{\infty
}\bar{k}_{pqr}(\omega) \bar{g}_{p} ( x_{p}(t-\omega)) d\omega \int\nolimits_{0}^{\infty
}\bar{h}_{pqr}(\omega) \bar{g}_{r}( x_{r}(t-\omega)) d\omega
-\sum\limits_{p,r=1}^{n_{1}} \bar{d}_{pqr} \hat{\bar{k}}_{pqr} \bar{g}_{p}( x_{p}^{\ast
}) \hat{\bar{h}}_{pqr} \bar{g}_{r}( x_{r}^{\ast }) \\
 =\sum\limits_{p,r=1}^{n_{1}}\int\nolimits_{0}^{\infty }\Bigg\{ \Big[
\bar{d}_{pqr}\Big( \int\nolimits_{0}^{\infty } \bar{h}_{pqr}(\omega) \bar{g}_{r}(
x_{r}(t-\omega)) d\omega \Big) \bar{k}_{pqr}(\omega)+ \bar{d}_{rqp} \hat{\bar{k}}_{rqp}\bar{g}_{r}(
x_{r}^{\ast }) \bar{h}_{rqp}(\omega)\Big]  \tilde{\bar{g}}_{p}(u_{p}(t-\omega)) \Bigg\} d\omega,
\end{array}
\end{equation}
where
\begin{equation*}
\begin{array}{c}
\tilde{g}_{q}\left( v_{q}(t)\right) :=g_{q}( y_{q}(t))
-g_{q}( y_{q}^{\ast }),\;
\tilde{\bar{g}}_{p}( u_{p}(t)) :=\bar{g}_{p}( x_{p}(t))
- \bar{g}_{p}( x_{p}^{\ast }).
\end{array}
\end{equation*}
Next, we add and substract the expressions $ca_{p} u_{p}(t-\mu)$ and $\bar{c}\bar{a}_{q} v_{q}(t-\mu)$ in the equations (\ref{t1}) and (\ref{t2}), resp.
\begin{equation}\label{s1}
\begin{array}{c}
 D_{C}^{\delta}\Big[ u_{p}(t)-cu_{p}(t-\mu )\Big] =
 -a_{p} \Big[ u_{p}(t) -c u_{p}(t-\mu)\Big] - c a_{p}u_{p}(t-\mu)
  +\sum\limits_{q,s=1}^{n_{2}} d_{qps} \int\nolimits_{0}^{\infty
} k_{qps}(\omega) \\ \times g_{q}( y_{q}(t-\omega)) d\omega
 \int\nolimits_{0}^{\infty
} h_{qps}(\omega) g_{s}( y_{s}(t-\omega)) d\omega
-  \sum\limits_{q,s=1}^{n_{2}} d_{qps}  \hat{k}_{qps} g_{q}( y_{q}^{\ast
}) \hat{h}_{qps} g_{s}(y_{s}^{\ast }),
\end{array}
\end{equation}
\begin{equation} \label{s2}
\begin{array}{c}
 D_{C}^{\delta}\Big[ v_{q}(t)-\bar{c}v_{q}(t-\mu)\Big] =  -\bar{a}_{q}\Big[  v_{q}(t) -\bar{c} v_{q}(t-\mu)\Big] - \bar{c}\bar{a}_{q} v_{q}(t-\mu)
  +\sum\limits_{p,r=1}^{n_{1}} \bar{d}_{pqr} \int\nolimits_{0}^{\infty
} \bar{k}_{pqr}(\omega) \\ \times \bar{g}_{p}( x_{p}(t-\omega)) d\omega
 \int\nolimits_{0}^{\infty
} \bar{h}_{pqr}(\omega) \bar{g}_{r}( x_{r}(t-\omega)) d\omega
  -  \sum\limits_{p,r=1}^{n_{1}} \bar{d}_{pqr} \hat{\bar{k}}_{pqr} \bar{g}_{p}( x_{p}^{\ast
}) \hat{\bar{h}}_{pqr} \bar{g}_{r}(x_{r}^{\ast }).
\end{array}
\end{equation}
An application of Laplace transform to (\ref{s1}) and (\ref{s2}) yields
\begin{equation}\label{l1}
\begin{array}{c}
 u_{p}(t)-cu_{p}(t- \mu )= E_{\delta}( - a_{p} t^{\delta})\Big[ \tilde{\phi}_{p}(0)-c \tilde{\phi} _{p}(-\mu )\Big]
- c a_{p} \int_{0}^{t}(t-\omega)^{\delta-1}E_{\delta,\delta}(-a_{p} (t-\omega)^{\delta})\\
\times u_{p}(\omega-\mu )d\omega  +
 \sum\limits_{q,s=1}^{n_{2}}d_{qps} \int_{0}^{t}(t-\omega)^{\delta-1} E_{\delta,\delta}(-a_{p} (t-\omega)^{\delta}) \int\nolimits_{0}^{\infty
} k_{qps}(\lambda) g_{q}( y_{q}(\omega-\lambda)) d\lambda \\
\times  \int\nolimits_{0}^{\infty
} h_{qps}(\lambda) g_{s}( y_{s}(\omega-\lambda)) d\lambda d\omega
 - \sum\limits_{q,s=1}^{n_{2}} d_{qps} \hat{k}_{qps} g_{q}( y_{q}^{\ast
}) \hat{h}_{qps} g_{s}(y_{s}^{\ast }) \int_{0}^{t}(t-\omega)^{\delta-1} E_{\delta,\delta}(-a_{p} (t-\omega)^{\delta}) d\omega,
\end{array}
\end{equation}
\begin{equation}\label{l2}
\begin{array}{c}
v_{q}(t)-\bar{c}v_{q}(t- \mu ) =  E_{\delta}( - \bar{a}_{q} t^{\delta})\Big[ \tilde{\varphi}_{q}(0)-\bar{c}
\tilde{\varphi} _{q}(-\mu )\Big]
- \bar{c}  \bar{a}_{q} \int_{0}^{t}(t-\omega)^{\delta-1}E_{\delta,\delta}(-\bar{a}_{q} (t-\omega)^{\delta}) \\
  \times v_{q}(\omega- \mu )d\omega
 +  \sum\limits_{p,r=1}^{n_{1}}\bar{d}_{pqr} \int_{0}^{t}(t-\omega)^{\delta-1} E_{\delta
,\delta}(- \bar{a}_{q} (t-\omega)^{\delta}) \int\nolimits_{0}^{\infty
} \bar{k}_{pqr}(\lambda) \bar{g}_{p}( x_{p}(\omega-\lambda)) d\lambda\\
\times  \int\nolimits_{0}^{\infty
} \bar{h}_{pqr}(\lambda) \bar{g}_{r}( x_{r}(\omega-\lambda)) d\lambda d\omega
 - \sum\limits_{p,r=1}^{n_{1}} \bar{d}_{pqr} \hat{\bar{k}}_{pqr} \bar{g}_{p}( x_{p}^{\ast
}) \hat{\bar{h}}_{pqr} \bar{g}_{r}(x_{r}^{\ast })\int_{0}^{t}(t-\omega)^{\delta-1} E_{\delta,\delta}(- \bar{a}_{q} (t-\omega)^{\delta}) d\omega.
\end{array}
\end{equation}
In light of the assumptions \textbf{(A2)} and \textbf{(A3)}, (\ref{b1}) and (\ref{b2}), we obtain after evaluating the last terms in (\ref{l1}) and (\ref{l2})
\begin{equation*}
\begin{array}{c}
 u_{p}(t)-cu_{p}(t- \mu ) \leq  E_{\delta}( - a_{p} t^{\delta})\Big[ \tilde{\phi}_{p}(0)-c \tilde{\phi} _{p}(-\mu )\Big]
- c a_{p} \int_{0}^{t}(t-\omega)^{\delta-1}\\
\times E_{\delta,\delta}(-a_{p} (t-\omega)^{\delta}) u_{p}(\omega-\mu )d\omega
  + G  \sum\limits_{q,s=1}^{n_{2}} L_{q}\int_{0}^{t}(t-\omega)^{\delta-1} E_{\delta,\delta}(-a_{p} (t-\omega)^{\delta}) \\ \times \int\nolimits_{0}^{\infty }\Big[ d_{qps}
\hat{h}_{qps}k_{qps}(\lambda )+ d_{spq}\hat{k}_{spq} h_{spq}(\lambda )\Big]| v_{q}(\omega-\lambda ) | d\lambda d\omega,\;t>0,
\end{array}
\end{equation*}
\begin{equation*}
\begin{array}{c}
 v_{q}(t)-\bar{c}v_{q}(t- \mu ) \leq  E_{\delta}( -  \bar{a}_{q} t^{\delta})\Big[ \tilde{\varphi}_{q}(0)-\bar{c}
\tilde{\varphi} _{q}(-\mu )\Big]- \bar{c}  \bar{a}_{q} \int_{0}^{t}(t-\omega)^{\delta-1}\\ \times E_{\delta,\delta}(-\bar{a}_{q} (t-\omega)^{\delta}) v_{q}(\omega-\mu )d\omega
 + \bar{G}  \sum\limits_{p,r=1}^{n_{1}} M_{p}\int_{0}^{t}(t-\omega)^{\delta-1} E_{\delta,\delta}(- \bar{a}_{q} (t-\omega)^{\delta})
  \\ \times \int\nolimits_{0}^{\infty }\Big[ \bar{d}_{pqr}
\hat{\bar{h}}_{pqr}\bar{k}_{pqr}(\lambda )+ \bar{d}_{rqp}\hat{\bar{k}}_{rqp} \bar{h}_{rqp}(\lambda )\Big]| u_{p}(\omega-\lambda ) | d\lambda d\omega,\;t>0.
\end{array}
\end{equation*}%
Therefore
\begin{equation*}
\begin{array}{c}
| u_{p}(t) | \leq c | u_{p}(t-\mu) | + E_{\delta}(-a t^{\delta}) | \tilde{\phi}
_{p}(0)- c \tilde{\phi}_{p}(-\mu ) |
+ c a \int_{0}^{t}(t-\omega)^{\delta-1}\\ \times E_{\delta,\delta}(-a (t-\omega)^{\delta}) | u_{p}(\omega-\mu ) | d\omega
 + G \sum\limits_{q,s=1}^{n_{2}} L_{q}\int_{0}^{t}(t-\omega)^{\delta-1}E_{\delta,\delta}(-a (t-\omega)^{\delta})\\ \times \int\nolimits_{0}^{\infty }\Big[ d_{qps}
\hat{h}_{qps} k_{qps}(\lambda ) + d_{spq} \hat{k}_{spq} h_{spq}(\lambda )\Big]| v_{q}(\omega-\lambda ) | d\lambda d\omega,
\end{array}
\end{equation*}
and
\begin{equation*}
\begin{array}{c}
|v_{q}(t)| \leq \bar{c} | v_{q}(t- \mu) | + E_{\delta}(-\bar{a} t^{\delta})| \tilde{\varphi}_{q}(0) -  \bar{c}
\tilde{ \varphi}_{q}(-\mu )|
+ \bar{c} \bar{a} \int_{0}^{t}(t-\omega)^{\delta
-1} \\ \times E_{\delta,\delta}(- \bar{a}  (t-\omega)^{\delta}) | v_{q}(\omega-\mu) | d\omega
+ \bar{G} \sum\limits_{p,r =1}^{n_{1}} M_{p}\int_{0}^{t}(t-\omega)^{\delta-1}E_{\delta
,\delta}(-\bar{a} (t-\omega)^{\delta})\\ \times \int\nolimits_{0}^{\infty }\Big[ \bar{d}_{pqr}
\hat{\bar{h}}_{pqr} \bar{k}_{pqr}(\lambda ) + \bar{d}_{rqp} \hat{\bar{k}}_{rqp} \bar{h}_{rqp}(\lambda )\Big]
| u_{p}(\omega-\lambda ) | d\lambda d\omega.
\end{array}
\end{equation*}
After summing up, we arrive at
\begin{equation}\label{f3}
\begin{array}{c}
u(t) \leq c u(t-\mu )
+E_{\delta}(-a  t^{\delta}) \sum\limits_{p=1}^{n_{1}} | \tilde{\phi} _{p}(0)- c\tilde{\phi} _{p}(-\mu
)|
+ c a  \int_{0}^{t}(t-\omega)^{\delta-1} \\ \times E_{\delta,\delta
}(-a (t-\omega)^{\delta}) u(\omega-\mu )  d\omega
 + G   \sum\limits_{p;q,s=1}^{n_{1};n_{2}} L_{q}\int_{0}^{t}(t-\omega)^{\delta-1}E_{\delta
,\delta}(-a (t-\omega)^{\delta})\\ \times
 \int\nolimits_{0}^{\infty }\Big[ d_{qps}%
\hat{h}_{qps} k_{qps}(\lambda )+ d_{spq} \hat{k}_{spq} h_{spq}(\lambda )\Big]
v(\omega-\lambda )  d\lambda d\omega,\;t>0,
\end{array}
\end{equation}
\begin{equation}\label{f4}
\begin{array}{c}
v(t) \leq \bar{c} v(t- \mu)
+E_{\delta}(- \bar{a} t^{\delta}) \sum\limits_{q=1}^{n_{2}}| \tilde{\varphi} _{q}(0)- \bar{c}\tilde{\varphi} _{q}(-\mu
)|
 + \bar{c} \bar{a} \int_{0}^{t}(t-\omega)^{\delta-1} \\ \times E_{\delta ,\delta
}(- \bar{a} (t-\omega)^{\delta}) v(\omega-\mu ) d\omega
+ \bar{G}  \sum\limits_{p,r;q=1}^{n_{1};n_{2}} M_{p}\int_{0}^{t}(t-\omega)^{\delta-1}E_{\delta
,\delta}(- \bar{a} (t-\omega)^{\delta}) \\ \times  \int\nolimits_{0}^{\infty }\Big[ \bar{d}_{pqr}
\hat{\bar{h}}_{pqr} \bar{k}_{pqr}(\lambda )+ \bar{d}_{rqp} \hat{\bar{k}}_{rqp} \bar{h}_{rqp}(\lambda )\Big]
u(\omega-\lambda )  d\lambda d\omega,\;t>0.
\end{array}
\end{equation}
Notice that we can assume that
$$| \tilde{\phi}(\omega)  | \leq u_{0} E_{\delta}(- a (\omega+\mu)^{\delta}) \quad  \text{for} \; \omega\in [-\mu, 0], \; u_{0}> 0,$$
$$| \tilde{\varphi}(\omega)  | \leq v_{0} E_{\delta}(- \bar{a}  (\omega+\mu)^{\delta}) \quad  \text{for} \; \omega\in [-\mu, 0], \; v_{0}> 0.$$
This is always possible because, if $\tilde{\phi}(\omega) $ and $\tilde{\varphi}(\omega) $ are bounded by $\vartheta$ and $\bar{\vartheta}$, resp, then
\begin{eqnarray*}
  \vartheta &\leq& u_{0} E_{\delta}(- a \mu^{\delta}) \leq u_{0} E_{\delta}(- a (\omega+\mu)^{\delta}), \\
  \bar{\vartheta} &\leq& v_{0} E_{\delta}(- \bar{a}  \mu^{\delta}) \leq v_{0} E_{\delta}(- \bar{a} (\omega+\mu)^{\delta}).
\end{eqnarray*}
We can choose  $u_{0}= \frac{\vartheta}{E_{\delta}(-a \mu^{\delta})}$ and $v_{0}= \frac{\bar{\vartheta}}{E_{\delta}(-\bar{a} \mu^{\delta})}$.
 Consequently, (\ref{f3}) and (\ref{f4}) become
\begin{equation}\label{ff3}
\begin{array}{c}
u(t) \leq  2 u_{0}E_{\delta}(-a  t^{\delta}) +  c u(t-\mu )
 +  c a \int_{0}^{t}(t-\omega)^{\delta-1} E_{\delta,\delta
}(-a (t-\omega)^{\delta}) u(\omega-\mu )  d\omega \\
 + G   \sum\limits_{p;q,s=1}^{n_{1};n_{2}} L_{q}\int_{0}^{t}(t-\omega)^{\delta-1}E_{\delta,\delta}(-a (t-\omega)^{\delta}) \int\nolimits_{0}^{\infty }\Big[ d_{qps}%
\hat{h}_{qps} k_{qps}(\lambda )+ d_{spq} \hat{k}_{spq} h_{spq}(\lambda )\Big]
v(\omega-\lambda )  d\lambda d\omega,
\end{array}
\end{equation}
\begin{equation}\label{ff4}
\begin{array}{c}
v(t) \leq 2 v_{0} E_{\delta}(- \bar{a}  t^{\delta})+ \bar{c} v(t-\mu)  + \bar{c} \bar{a}  \int_{0}^{t}(t-\omega)^{\delta-1} E_{\delta,\delta}(- \bar{a} (t-\omega)^{\delta}) v(\omega-\mu ) d\omega \\
 +  \bar{G}  \sum\limits_{p,r;q=1}^{n_{1};n_{2}} M_{p}\int_{0}^{t}(t-\omega)^{\delta-1}E_{\delta,\delta}(- \bar{a} (t-\omega)^{\delta})
 \int\nolimits_{0}^{\infty }\Big[ \bar{d}_{pqr}
\hat{\bar{h}}_{pqr} \bar{k}_{pqr}(\lambda )+ \bar{d}_{rqp} \hat{\bar{k}}_{rqp} \bar{h}_{rqp}(\lambda )\Big]
u(\omega-\lambda )  d\lambda d\omega.
\end{array}
\end{equation}
Let $V(t) = \max\{ u(t), v(t) \}$ and $V_{0}= \max\{u_{0}, v_{0}   \}$. From the relations (\ref{ff3}) and (\ref{ff4}), we entail the single inequality
\begin{equation}\label{ess5}
\begin{array}{c}
V(t) \leq 2 V_{0}E_{\delta}(-\xi t^{\delta})  +
  c^{*} V(t-\mu )
+ c^{*} a^{*}  \int_{0}^{t}(t-\omega)^{\delta-1} E_{\delta,\delta}(-\xi (t-\omega)^{\delta}) V(\omega-\mu )  d\omega \\
+  \int_{0}^{t}(t-\omega)^{\delta-1}E_{\delta,\delta}(-\xi (t-\omega)^{\delta})
 \int\nolimits_{0}^{\infty } K(\lambda)
V(\omega-\lambda )  d\lambda d\omega,\;t>0.
\end{array}
\end{equation}
 For $t\in [0, \mu]$ and from (\ref{ess5}), we see that
 \begin{equation}\label{f5}
 \begin{array}{c}
V(t) \leq 3 V_{0}E_{\delta}(-\xi t^{\delta})  + c^{*} a^{*} \int_{0}^{t}(t-\omega)^{\delta-1} E_{\delta,\delta
}(- \xi (t-\omega)^{\delta}) V(\omega-\mu )  d\omega  \\
 +  \int_{0}^{t}(t-\omega)^{\delta-1}E_{\delta
,\delta}(-\xi (t-\omega)^{\delta})\int\nolimits_{0}^{\infty }
K(\lambda)V(\omega-\lambda )  d\lambda d\omega.
\end{array}
\end{equation}
Dividing by $E_{\delta}(-\xi t^{\delta})$, we get
\begin{equation}\label{s5}
\begin{array}{c}
\frac{V(t)}{E_{\delta}(-\xi t^{\delta})} \leq 3 V_{0}  +\frac{ c^{*} a^{*} V_{0}}{E_{\delta}( - \xi t^{\delta})} \int_{0}^{t}(t-\omega)^{\delta-1} E_{\delta,\delta}(-\xi (t-\omega)^{\delta})  E_{\delta}(- \xi \omega^{\delta}) d\omega\\
 +  \frac{1}{E_{\delta}(-\xi t^{\delta})}  \int_{0}^{t}(t-\omega)^{\delta-1}
 E_{\delta
,\delta}(-\xi (t-\omega)^{\delta})
\int\nolimits_{-\infty}^{\omega} K(\omega-\lambda) E_{\delta}(- \xi \lambda^{\delta}) d \lambda d\omega  \sup\limits_{- \infty < \lambda \leq t } \frac{V(\lambda )}{E_{\delta}(-\xi \lambda^{\delta})}.
\end{array}
\end{equation}
By virtue  of  the  estimation  (\ref{mc1}) and formula (\ref{formula1}), we obtain
\begin{equation}\label{g1}
\begin{array}{c}
  \int_{0}^{t}(t-\omega)^{\delta-1} E_{\delta,\delta
}(- \xi(t-\omega)^{\delta}) E_{\delta}(- \xi \omega^{\delta})  d\omega  \leq  \frac{\Gamma(1+\delta)}{\xi} \int_{0}^{t}(t-\omega)^{\delta-1} E_{\delta,\delta
}(-\xi (t-\omega)^{\delta})  \omega^{ - \delta}  d\omega    \\
   \leq \frac{\Gamma(1+\delta) \Gamma(1-\delta)}{\xi}  E_{\delta}(-\xi t^{\delta}),
\end{array}
\end{equation}
and from the condition on the kernels, the relation  (\ref{s5}) becomes
\begin{equation*}
\begin{array}{c}
\frac{V(t)}{E_{\delta}(-\xi t^{\delta})} \leq 3 V_{0}  + \frac{ c^{*} a^{*}}{\xi} V_{0} \Gamma(1+ \delta) \Gamma(1- \delta)
+ \Omega \sup\limits_{- \infty < \lambda \leq t}\frac{V(\lambda )}{E_{\delta}(-\xi  \lambda^{\delta})},
\end{array}
\end{equation*}
and thus
\begin{equation}\label{f7}
\begin{array}{c}
\sup\limits_{- \infty < \lambda \leq t} \frac{V(\lambda)}{E_{\delta}(-\xi \lambda^{\delta})} \leq 3 V_{0}  + \frac{ c^{*} a^{*}}{\xi} V_{0} \Gamma(1+ \delta) \Gamma(1- \delta)
+ \Omega \sup\limits_{- \infty < \lambda \leq t}\frac{V(\lambda )}{E_{\delta}(-\xi  \lambda^{\delta})}.
\end{array}
\end{equation}
Therefore,
\begin{equation}\label{eq8}
\begin{array}{c}
(1-\Omega)\frac{V(t)}{E_{\delta}(-\xi t^{\delta})} \leq \Big[ 3   + \frac{c^{*} a^{*}}{\xi}   \Gamma(1+ \delta) \Gamma(1-\delta) \Big] V_{0}.
\end{array}
\end{equation}
In case of $t\in [\mu, 2 \mu]$, we note that
 \begin{equation}\label{eq9}
 \begin{array}{c}
   V(t-\mu)  \leq W^{*} \frac{E_{\delta}(-\xi  (t-\mu)^{\delta})}{E_{\delta}(- \xi  t^{\delta})} V_{0} E_{\delta}(- \xi  t^{\delta})
   \leq W^{*} B V_{0} E_{\delta}(- \xi t^{\delta}),
 \end{array}
 \end{equation}
where
\begin{equation*}
\begin{array}{c}
W^{*} = \frac{3+ \frac{c^{*} a^{*}}{\xi}  \Gamma(1+\delta) \Gamma(1-\delta) }{1-\Omega},\\
\frac{E_{\delta}(-\xi  (t-\mu)^{\delta})}{E_{\delta}(- \xi t^{\delta})} \leq
   \frac{1}{E_{\delta}(- \xi t^{\delta})} \leq \frac{1}{E_{\delta}(- \xi (2 \mu)^{\delta})} \leq 1+ \xi  \Gamma(1- \delta) (2 \mu)^{\delta}  =:B.
\end{array}
\end{equation*}
In view of the relations  (\ref{ess5}), (\ref{g1}) and (\ref{eq9}), we find
\begin{equation}\label{f6}
\begin{array}{c}
V(t) \leq 2 V_{0} E_{\delta}(- \xi  t^{\delta}) + c^{*} B V_{0} W^{*}  E_{\delta}(- \xi t^{\delta})
  +  a^{*} c^{*}   B V_{0} W^{*}
  \frac{ \Gamma(1+ \delta) \Gamma(1- \delta) }{\xi} E_{\delta}(- \xi t^{\delta})\\
 +  \int_{0}^{t}(t-\omega)^{\delta-1}E_{\delta
,\delta}(- \xi (t-\omega)^{\delta})    \int\nolimits_{0}^{\infty }K(\lambda)
V(\omega-\lambda )  d\lambda  d\omega.
\end{array}
\end{equation}
Therefore, as $\xi \leq a^{*}$
\begin{equation}\label{eq10}
\begin{array}{c}
(1-\Omega) \frac{V(t)}{E_{\delta}(-\xi t^{\delta})} \leq 2 V_{0}  + \frac{c^{*} a^{*}}{\xi}  B V_{0} \Big[1+  \Gamma(1+\delta) \Gamma(1-\delta)  \Big] W^{*} \\
\leq  2 V_{0}  + 3  B V_{0} W \Big(\frac{a^{*}c^{*}}{\xi}\Big) + 3 B V_{0} W^{2} \Big(\frac{a^{*}c^{*}}{\xi}\Big)^{2},
\end{array}
\end{equation}
where $W=\frac{1+\Gamma(1+\delta) \Gamma(1-\delta) }{1-\Omega}.$\\
For $t\in [2 \mu, 3 \mu]$, from the estimation $\frac{t^{\delta}}{(t-\mu)^{\delta}} \leq 2^{\delta}$ and the relation (\ref{mc1}), we obtain
a new estimation of $E_{\delta}(- \xi(t- \mu)^{\delta} )/ E_{\delta}(- \xi t^{\delta})$
\begin{equation}\label{eqq11}
\begin{array}{c}
  \frac{E_{\delta}(- \xi (t- \mu)^{\delta})}{E_{\delta}(- \xi t^{\delta})} \leq \frac{1+ \xi  \Gamma(1-\delta) t^{\delta}}{1+ \xi  \Gamma(1+\delta)^{-1} (t-\mu)^{\delta}} \leq \frac{\Gamma(1+\delta)}{ \xi \mu^{\delta}}
   + \frac{ \Gamma(1+\delta) \Gamma(1-\delta) t^{\delta}}{(t- \mu)^{\delta}}
   \leq \Gamma(1+\delta) \Big[\frac{1}{\xi \mu^{\delta}} + 2^{\delta} \Gamma(1-\delta)    \Big]
   \end{array}
   \end{equation}
or
\begin{equation}\label{eq12}
\begin{array}{c}
\frac{E_{\delta}(- \xi (t- \mu)^{\delta})}{E_{\delta}(- \xi t^{\delta})}   \leq  \frac{1}{\xi \mu^{\delta}} + 2^{\delta} \Gamma(1-\delta)   = :F, \; t \in [ 2 \mu, 3 \mu].
\end{array}
\end{equation}
Therefore, (\ref{ess5}) implies
\begin{equation}\label{eq13}
\begin{array}{c}
V(t)  \leq  2 V_{0}  E_{\delta}(- \xi t^{\delta}) + \frac{c^{*} F}{1-\Omega} \Big[  2 V_{0} + 3 B V_{0}  W  \Big(\frac{a^{*}c^{*}}{\xi}\Big) + 3 B V_{0} W^{2} (1-\Omega) \Big(\frac{a^{*}c^{*}}{\xi}\Big)^{2}  \Big]  E_{\delta}(- \xi t^{\delta})\\
+ c^{*} a^{*}F \frac{\Gamma(1+\delta) \Gamma(1-\delta)}{\xi(1- \Omega)}\Big [ 2 V_{0} +3 B V_{0} W \Big(\frac{a^{*}c^{*}}{\xi}\Big) + 3 B V_{0} W^{2} (1-\Omega) \Big(\frac{a^{*}c^{*}}{\xi}\Big)^{2}\Big]E_{\delta}(- \xi t^{\delta}) \\
 + \int_{0}^{t}(t-\omega)^{\delta-1}E_{\delta
,\delta}(- \xi (t-\omega)^{\delta})    \int\nolimits_{0}^{\infty }
K(\lambda) V(\omega-\lambda )  d\lambda  d\omega.
\end{array}
\end{equation}
Hence
\begin{equation*}
\begin{array}{c}
(1- \Omega)\frac{V(t)}{ E_{\delta}(- \xi t^{\delta})}  \leq  2 V_{0}  + \Big(\frac{a^{*}c^{*}}{\xi}\Big) F W \Big[  2 V_{0} + 3 B V_{0}  W  \Big(\frac{a^{*}c^{*}}{\xi}\Big) + 3 B V_{0} W^{2} (1-\Omega) \Big(\frac{a^{*}c^{*}}{\xi}\Big)^{2}  \Big].
\end{array}
\end{equation*}
Thus
\begin{equation}\label{eq15}
\begin{array}{c}
(1- \Omega)\frac{V(t)}{ E_{\delta}(- \xi t^{\delta})} \leq 3 B V_{0} \Big[ 1+  F W \Big(\frac{a^{*}c^{*}}{\xi}\Big) + \Big(F W \Big(\frac{a^{*}c^{*}}{\xi}\Big)\Big)^{2} + \Big(F W \Big(\frac{a^{*}c^{*}}{\xi}\Big)\Big)^{3}\Big].
\end{array}
\end{equation}
 We claim that
\begin{equation}\label{cl1}
\begin{array}{c}
  (1-\Omega) \frac{V(t)}{E_{\delta}(- \xi t^{\delta})} \leq 3 B V_{0} \sum\limits_{l=0}^{k} \Big(F W \Big(\frac{a^{*}c^{*}}{\xi}\Big) \Big)^{l}, \;  t \in [(k-1) \mu, k\mu].
  \end{array}
\end{equation}
By virtue of the estimations   (\ref{eq8}), (\ref{eq10}) and (\ref{eq15}), the relation (\ref{cl1}) is valid  for $k=1,2$ and $3$ respectively.
Assuming  that it is valid for $t\in [(k-1) \mu, k\mu]$ and we will prove it
for $t\in [k \mu, (k+1) \mu]$.
 From  the relations (\ref{ess5}), (\ref{g1}) and (\ref{cl1}), we get
\begin{equation*}
\begin{array}{c}
V(t) \leq 2 V_{0} E_{\delta}(-\xi t^{\delta}) +  3B F c^{*} \frac{V_{0}}{1-\Omega} \sum\limits_{l=0}^{k} \Big(F W \Big(\frac{a^{*}c^{*}}{\xi}\Big)  \Big)^{l}  E_{\delta}(-\xi t^{\delta})\\
+ 3B F \Big(\frac{a^{*}c^{*}}{\xi}\Big)   \frac{V_{0}}{1-\Omega} \Gamma(1+\delta) \Gamma(1-\delta) \sum\limits_{l=0}^{k} \Big(F W \Big(\frac{a^{*}c^{*}}{\xi}\Big) \Big)^{l}  E_{\delta}(- \xi t^{\delta}) \\
 + \int_{0}^{t}(t-\omega)^{\delta-1}E_{\delta
,\delta}(- \xi (t-\omega)^{\delta})    \int\nolimits_{0}^{\infty } K(\lambda)
V(\omega-\lambda )  d\lambda d\omega.
\end{array}
\end{equation*}
Consequently,
\begin{equation*}
\begin{array}{c}
(1-\Omega) \frac{V(t) }{E_{\delta}(-\xi t^{\delta})} \leq 2 V_{0} + 3 B  F \Big(\frac{a^{*}c^{*}}{\xi}\Big) \frac{V_{0}}{1- \Omega}\Big[ 1+ \Gamma(1+\delta) \Gamma(1-\delta) \Big] \sum\limits_{l=0}^{k} \Big(F W \Big(\frac{a^{*}c^{*}}{\xi}\Big)   \Big)^{l}.
\end{array}
\end{equation*}
Then, by the definition of $W$
\begin{equation*}
\begin{array}{c}
(1-\Omega) \frac{V(t)}{E_{\delta}(- \xi  t^{\delta})} \leq 3 B  V_{0} \Big\{ 1+ \Big(\frac{a^{*}c^{*}}{\xi}\Big)  F W    \sum\limits_{l=0}^{k} \Big(F W \Big(\frac{a^{*}c^{*}}{\xi}\Big) \Big)^{l}  \Big\}
= 3  B V_{0}  \sum\limits_{l=0}^{k+1} \Big(F W \Big(\frac{a^{*}c^{*}}{\xi}\Big)  \Big)^{l}.
\end{array}
\end{equation*}
In light of the conditions indicated in  Theorem \ref{th1}, the series is convergent.

 \end{proof}

\section{Unbounded activation functions}
Unbounded activation functions in case of higher-order NNs are not easy to deal with because of the nonlinear terms. This is in contrast to the lower-order case.
 To address this issue, we shall use some analytical techniques based on suitable evaluations and properties of the Mittag-Leffler functions. The notation below will be utilized
\begin{equation*}
\begin{array}{c}
  U :=A\frac{\Gamma(1+\delta) \Gamma(1-\delta)}{\xi }  , \; A=\max\Big\{ \sum\limits_{p=1}^{n_{1}}  a_{p}, \sum\limits_{q=1}^{n_{2}}  \bar{a}_{q} \Big\},\;
  \Lambda := \sum\limits_{k=0}^{+\infty} \Big[   2 B^{*} c^{*} (1+ U) \Big]^{k},  \;
   B^{*} := \max \{ B, F   \}, \\
\theta= \sum\limits_{p;q,s=1}^{n_{1};n_{2}}  d_{qps}  \Big[ L_{q} \hat{h}_{qps} g_{s}(y_{s}^{*})+ L_{s} \hat{k}_{qps} g_{q}(y_{q}^{*}) \Big], \; \nu=  \sum\limits_{p;q,s=1}^{n_{1};n_{2}}   d_{qps} L_{q} L_{s}, \\
\bar{\theta} =  \sum\limits_{p,r; q=1}^{n_{1};n_{2}}  \bar{d}_{pqr}  \Big[ M_{p} \hat{\bar{h}}_{pqr} \bar{g}_{r}(x_{r}^{*})+ M_{r} \hat{\bar{h}}_{pqr} \bar{g}_{p}(x_{p}^{*}) \Big], \;  \bar{\nu}= \sum\limits_{p,r;q=1}^{n_{1};n_{2}}   \bar{d}_{pqr} M_{p} M_{r}, \\
k(t)= \max\limits_{1\leq  p\leq n_{1}; 1\leq q, s\leq n_{2}}\{ k_{qps}(t)\},  h(t)= \max\limits_{1\leq  p\leq n_{1}; 1\leq q,s\leq n_{2}}\{ h_{qps}(t)\}, \; K(t)= \max\{k(t), h(t) \}, \\
\bar{k}(t)=\max\limits_{1 \leq p,r \leq n_{1}; 1 \leq q \leq n_{2}}\{  \bar{k}_{pqr}(t)\},  \bar{h}(t)=\max\limits_{1 \leq p, r \leq n_{1}; 1\leq q\leq n_{2}}\{ \bar{h}_{rqp}(t) \}, \; H(t)=\max\{\bar{k}(t), \bar{h}(t)\},\\
 \pi =\max\{\theta, \bar{\theta}\}, \kappa=\max\{ \nu, \bar{\nu}\}, \;
K^{*}(t)=\max\{K(t), H(t)\}, \; k^{*}(t)=\max\{k(t), \bar{k}(t)\}, \\ h^{*}(t)=\max\{h(t), \bar{h}(t)\}, \; \hat{h}^{*}= \int_{0}^{\infty} h^{*}(t) dt.
\end{array}
\end{equation*}
\textbf{(A5)} Let $\Omega$ and $c^{*}$ be positive constants such that $ \Omega \pi  <\frac{1}{4},$ $c^{*}  < \min \Big\{ 1, \frac{1}{2 B^{*}(1+ U)}\Big\},$
\begin{equation}\label{k1}
  \int_{0}^{t} (t-\omega)^{\delta-1} E_{\delta,\delta} (-\xi(t-\omega)^{\delta}) \Big( \int_{-\infty}^{\omega} E_{\delta} (- \xi \lambda^{\delta}) K^{*}(\omega-\lambda)  d\lambda \Big) d\omega
   \leq  \Omega E_{\delta} (-  \xi t^{\delta}), \; t> 0.
  \end{equation}
\begin{thm}\label{th2}
Assuming that  \textbf{(A1)}, \textbf{(A3)} and \textbf{(A5)} hold.
 Then, the solutions of system  (\ref{hfn1}) are locally $\delta-$Mittag-Leffler stable, which means that for  some constant $C>0$
\begin{eqnarray*}
  u(t) &\leq& C V_{0}E_{\delta} (-\xi t^{\delta}), \;  v(t) \leq C V_{0} E_{\delta} (-\xi t^{\delta}), \quad t>0,
\end{eqnarray*}
for small $V_{0}$   and  positive constant $\xi$.
\end{thm}
\begin{proof}
It is obvious that
\begin{equation}\label{nl1}
\begin{array}{c}
 \sum\limits_{q,s=1}^{n_{2}}d_{qps}\int\nolimits_{0}^{\infty
} k_{qps}(\omega) g_{q}( y_{q}(t-\omega)) d\omega \int\nolimits_{0}^{\infty} h_{qps}(\omega) g_{s}( y_{s}(t-\omega)) d\omega
-\sum\limits_{q,s=1}^{n_{2}} d_{qps} \hat{k}_{qps} g_{q}( y_{q}^{\ast
}) \hat{h}_{qps} g_{s}( y_{s}^{\ast })\\
 = \sum\limits_{q,s=1}^{n_{2}} d_{qps}
\Bigg\{ \Big[ \int\nolimits_{0}^{\infty }k_{qps}(\omega)  \Big[  g_{q}(y_{q}(t-\omega)) - g_{q}(y_{q}^{*}) \Big] d\omega
\int\nolimits_{0}^{\infty }h_{qps}(\omega) \Big[  g_{s}(y_{s}(t-\omega)) - g_{s}(y_{s}^{*}) \Big] d\omega \\
 +
\hat{h}_{qps} g_{s}(y_{s}^{*}) \int\nolimits_{0}^{\infty }  k_{qps}(\omega) \Big[  g_{q}(y_{q}(t-\omega)) - g_{q}(y_{q}^{*}) \Big]d\omega   \\
+ \hat{k}_{qps} g_{q}(y_{q}^{*}) \int\nolimits_{0}^{\infty }  h_{qps}(\omega) \Big[  g_{s}(y_{s}(t-\omega)) - g_{s}(y_{s}^{*}) \Big]d\omega  \Bigg\}
\end{array}
\end{equation}
and
\begin{equation}\label{nl2}
\begin{array}{c}
 \sum\limits_{p,r=1}^{n_{1}}\bar{d}_{pqr}\int\nolimits_{0}^{\infty
} \bar{k}_{pqr}(\omega) \bar{g}_{p}( x_{p}(t-\omega)) d\omega \int\nolimits_{0}^{\infty} \bar{h}_{pqr}(\omega) \bar{g}_{r}( x_{r}(t-\omega)) d\omega
-\sum\limits_{p,r=1}^{n_{1}} \bar{d}_{pqr} \hat{\bar{k}}_{pqr} \bar{g}_{p}( x_{p}^{\ast
}) \hat{\bar{h}}_{pqr} \bar{g}_{r}( x_{r}^{\ast })\\
 = \sum\limits_{p,r=1}^{n_{1}} \bar{d}_{pqr}
\Bigg\{ \Big[ \int\nolimits_{0}^{\infty }\bar{k}_{pqr}(\omega)  \Big[  \bar{g}_{p}(x_{p}(t-\omega)) - \bar{g}_{p}(x_{p}^{*}) \Big] d\omega
\int\nolimits_{0}^{\infty } \bar{h}_{pqr}(\omega) \Big[  \bar{g}_{r}(x_{r}(t-\omega)) - \bar{g}_{r}(x_{r}^{*}) \Big] d\omega \\
 + \hat{\bar{h}}_{pqr}(\omega) \bar{g}_{r}(x_{r}^{*}) \int\nolimits_{0}^{\infty }  \bar{k}_{pqr}(\omega) \Big[  \bar{g}_{p}(x_{p}(t-\omega)) - \bar{g}_{p}(x_{p}^{*}) \Big]d\omega \\
+  \hat{\bar{k}}_{pqr}(\omega) \bar{g}_{p}(x_{p}^{*}) \int\nolimits_{0}^{\infty }  \bar{h}_{pqr}(\omega) \Big[  \bar{g}_{r}(x_{r}(t-\omega)) - \bar{g}_{r}(x_{r}^{*}) \Big]d\omega  \Bigg\}.
\end{array}
\end{equation}
These identities are very useful. They will facilitate some evaluations below.

According to (\ref{l1}), (\ref{l2}), (\ref{nl1}) and (\ref{nl2}), we obtain
\begin{equation*}
\begin{array}{c}
 u_{p}(t)-c u_{p}(t-\mu) \leq  E_{\delta}(- a_{p} t^{\delta}) |\tilde{\phi}_{p}(0) - c \tilde{\phi}_{p}(\mu) |
   + c  a_{p} \int_{0}^{t} (t-\omega)^{\delta-1} E_{\delta, \delta}(-  a_{p} (t-\omega)^{\delta})\\ \times  |u_{p}(\omega-\mu) | d\omega
   +   \sum\limits_{q,s=1}^{n_{2}} d_{qps} L_{q} L_{s} \int_{0}^{t} (t-\omega)^{\delta-1} E_{\delta, \delta} (- a_{p} (t-\omega)^{\delta} )
   \Big( \int_{0}^{\infty} k_{qps}(\lambda) |v_{q}(\omega-\lambda)  | d \lambda\\ \times
      \int_{0}^{\infty} h_{qps}(\lambda) |v_{s}(\omega- \lambda)  | d \lambda d\omega   \Big)
    +   \sum\limits_{q,s=1}^{n_{2}} d_{qps} L_{q} \hat{h}_{qps} g_{s}(y_{s}^{*})
  \Big( \int_{0}^{t} (t-\omega)^{\delta-1} E_{\delta, \delta}(- a_{p} (t-\omega)^{\delta})\\ \times
     \int_{0}^{\infty} k_{qps}(\lambda) |v_{q}(\omega- \lambda) | d \lambda d\omega  \Big)
    +   \sum\limits_{q,s=1}^{n_{2}} d_{qps} L_{s} \hat{k}_{qps} g_{q}(y_{q}^{*})
  \Big(  \int_{0}^{t} (t-\omega)^{\delta-1} E_{\delta, \delta}(- a_{p} (t-\omega)^{\delta})\\ \times
   \int_{0}^{\infty} h_{qps}(\lambda) |v_{s}(\omega- \lambda) | d \lambda d\omega \Big)
\end{array}
\end{equation*}
and
\begin{equation*}
\begin{array}{c}
 v_{q}(t)-\bar{c} v_{q}(t- \mu) \leq  E_{\delta}(- \bar{a}_{q} t^{\delta}) |\tilde{\varphi}_{q}(0) -
 \bar{c} \tilde{\varphi}_{q}(\mu) |
  + \bar{c}  \bar{a}_{q} \int_{0}^{t} (t-\omega)^{\delta-1} E_{\delta, \delta}(-  \bar{a}_{q} (t-\omega)^{\delta})\\ \times
   |v_{q}(\omega-\mu) | d\omega
  +   \sum\limits_{p,r=1}^{n_{1}} \bar{d}_{pqr} M_{p} M_{r} \int_{0}^{t} (t-\omega)^{\delta-1} E_{\delta, \delta} (- \bar{a}_{q} (t-\omega)^{\delta} )
   \Big( \int_{0}^{\infty} \bar{k}_{pqr}(\lambda) |u_{p}(\omega-\lambda)  | d \lambda\\ \times
      \int_{0}^{\infty} \bar{h}_{pqr}(\lambda) |u_{r}(\omega- \lambda)  | d \lambda d\omega \Big)
  +  \sum\limits_{p,r=1}^{n_{1}} \bar{d}_{pqr} M_{p} \hat{\bar{h}}_{pqr} \bar{g}_{r}(x_{r}^{*})
  \Big(  \int_{0}^{t} (t-\omega)^{\delta-1} E_{\delta, \delta}(- \bar{a}_{q} (t-\omega)^{\delta})\\ \times
    \int_{0}^{\infty} \bar{k}_{pqr}(\lambda) |u_{p}(\omega- \lambda) | d \lambda d\omega  \Big)
    +   \sum\limits_{p,r=1}^{n_{1}} \bar{d}_{pqr} M_{r} \hat{\bar{k}}_{pqr} \bar{g}_{p}(x_{p}^{*})
  \Big(  \int_{0}^{t} (t-\omega)^{\delta-1} E_{\delta, \delta}(- \bar{a}_{q} (t-\omega)^{\delta})\\ \times
     \int_{0}^{\infty} \bar{h}_{pqr}(\lambda) |u_{r}(\omega- \lambda) | d \lambda d\omega \Big).
\end{array}
\end{equation*}
Using the   notation in the above, we get
\begin{equation*}
\begin{array}{c}
 u(t) \leq  (1+c) u_{0} E_{\delta}(- a t^{\delta})  +  c u(t-\mu) + c \sum\limits_{p=1}^{n_{1}}  a_{p} \int_{0}^{t} (t-\omega)^{\delta-1} E_{\delta, \delta}(-  a  (t-\omega)^{\delta}) u(\omega-\mu)  d\omega \\
   +   \sum\limits_{p;q,s=1}^{n_{1};n_{2}}  d_{qps}  \Big[ L_{q} \hat{h}_{qps} g_{s}(y_{s}^{*})+ L_{s} \hat{k}_{qps} g_{q}(y_{q}^{*}) \Big]  \Big( \int_{0}^{t} (t-\omega)^{\delta-1} E_{\delta, \delta}(- a  (t-\omega)^{\delta})
   \\ \times  \int_{-\infty}^{\omega}
   \max \Big\{ k_{qps}(\omega-\lambda ), h_{qps} (\omega-\lambda) \Big \}
    v(\lambda)  d \lambda d\omega \Big)
   +   \sum\limits_{p;q,s=1}^{n_{1};n_{2}}   d_{qps} L_{q} L_{s} \int_{0}^{t} (t-\omega)^{\delta-1} E_{\delta, \delta} (- a (t-\omega)^{\delta-1} )\\
   \times
  \Big( \int_{0}^{\infty} k_{qps}(\lambda) v(\omega-\lambda)   d \lambda
    \int_{0}^{\infty} h_{qps}(\lambda) v( \omega-\lambda)   d \lambda d\omega\Big)
\end{array}
\end{equation*}
and
\begin{equation*}
\begin{array}{c}
 v(t) \leq (1+\bar{c}) v_{0} E_{\delta}(-\bar{a} t^{\delta})  +  \bar{c} v(t-\mu)
 + \bar{c} \sum\limits_{q=1}^{n_{2}}  \bar{a}_{q} \int_{0}^{t} (t-\omega)^{\delta-1} E_{\delta, \delta}(-  \bar{a}  (t-\omega)^{\delta}) v(\omega-\mu)  d\omega   \\
   +  \sum\limits_{p,r; q=1}^{n_{1};n_{2}}  \bar{d}_{pqr}  \Big[ M_{p} \hat{\bar{h}}_{pqr} \bar{g}_{r}(x_{r}^{*})+ M_{r} \hat{\bar{h}}_{pqr} \bar{g}_{p}(x_{p}^{*}) \Big]  \Big( \int_{0}^{t} (t-\omega)^{\delta-1} E_{\delta, \delta}(- \bar{a} (t-\omega)^{\delta})  \\
   \times  \int_{-\infty}^{\omega}
   \max \Big\{ \bar{k}_{pqr}(\omega-\lambda), \bar{h}_{pqr}  (\omega-\lambda)\Big \} u(\lambda)  d \lambda d\omega \Big)
  +   \sum\limits_{p,r;q=1}^{n_{1};n_{2}}   \bar{d}_{pqr} M_{p} M_{r} \int_{0}^{t} (t-\omega)^{\delta-1} E_{\delta, \delta} (- \bar{a} (t-\omega)^{\delta})
  \\ \times \Big( \int_{0}^{\infty} \bar{k}_{pqr}(\lambda) u(\omega-\lambda)   d \lambda    \int_{0}^{\infty} \bar{h}_{pqr}(\lambda) u(\omega- \lambda)   d \lambda d\omega \Big).
\end{array}
\end{equation*}
Let $V(t)= \max \{ u(t), v(t) \}, \; t\geq 0$ and $V_{0}= \max\{ u_{0}, v_{0}\}$, then
\begin{equation*}
\begin{array}{c}
 u(t) \leq  (1+c) V_{0} E_{\delta}(- a t^{\delta})  +  c V(t-\mu) + c \sum\limits_{p=1}^{n_{1}}  a_{p} \int_{0}^{t} (t-\omega)^{\delta-1} E_{\delta, \delta}(-  a  (t-\omega)^{\delta}) V(\omega-\mu)  d\omega \\
   +   \theta  \Big( \int_{0}^{t} (t-\omega)^{\delta-1} E_{\delta, \delta}(- a  (t-\omega)^{\delta})
    \int_{-\infty}^{\omega}
   K (\omega-\lambda)
    V(\lambda)  d \lambda d\omega \Big)
   +   \nu \int_{0}^{t} (t-\omega)^{\delta-1} E_{\delta, \delta} (- a (t-\omega)^{\delta} )\\
  \times \Big( \int_{0}^{\infty} k(\lambda) V(\omega-\lambda)   d \lambda
    \int_{0}^{\infty} h(\lambda) V( \omega-\lambda)   d \lambda d\omega\Big),
\end{array}
\end{equation*}
\begin{equation*}
\begin{array}{c}
 v(t) \leq (1+\bar{c}) V_{0} E_{\delta}(-\bar{a} t^{\delta})  +  \bar{c} V(t-\mu)
 + \bar{c} \sum\limits_{q=1}^{n_{2}}  \bar{a}_{q} \int_{0}^{t} (t-\omega)^{\delta-1} E_{\delta, \delta}(-  \bar{a}  (t-\omega)^{\delta}) V(\omega-\mu)  d\omega   \\
   +  \bar{\theta} \Big( \int_{0}^{t} (t-\omega)^{\delta-1} E_{\delta, \delta}(- \bar{a} (t-\omega)^{\delta})
   \int_{-\infty}^{\omega}
     H(\omega-\lambda) V(\lambda)  d \lambda d\omega \Big)
  +   \bar{\nu} \int_{0}^{t} (t-\omega)^{\delta-1} E_{\delta, \delta} (- \bar{a} (t-\omega)^{\delta})
  \\ \times \Big( \int_{0}^{\infty} \bar{k}(\lambda) V(\omega-\lambda)   d \lambda    \int_{0}^{\infty} \bar{h}(\lambda) V(s- \lambda)   d \lambda d\omega \Big).
\end{array}
\end{equation*}
Furthermore,
\begin{equation}\label{fnl5}
\begin{array}{c}
V(t) \leq  (1+c^{*}) V_{0} E_{\delta}(- \xi t^{\delta})  +  c^{*} V(t-\mu) + c^{*} A \int_{0}^{t} (t-\omega)^{\delta-1} E_{\delta, \delta}(-  \xi (t-\omega)^{\delta}) V(\omega-\mu)  d\omega \\
   +   \pi   \int_{0}^{t} (t-\omega)^{\delta-1} E_{\delta, \delta}(- \xi  (t-\omega)^{\delta})
    \int_{-\infty}^{\omega}
   K^{*} (\omega-\lambda)
    V(\lambda)  d \lambda d\omega
   +   \kappa \int_{0}^{t} (t-\omega)^{\delta-1} E_{\delta, \delta} (- \xi (t-\omega)^{\delta} )\\
  \times  \int_{0}^{\infty} k^{*}(\lambda) V(\omega-\lambda)   d \lambda
    \int_{0}^{\infty} h^{*}(\lambda) V( \omega-\lambda)   d \lambda d\omega.
\end{array}
\end{equation}
For $\omega\in [0, \mu]$, we have $-\mu\leq \omega- \mu \leq 0$ and assume 
\begin{equation*}
\begin{array}{c}
  \tilde{\phi}(\omega) := \sum\limits_{p=1}^{n_{1}} | \tilde{\phi}_{p}(\omega)  | \leq u_{0} E_{\delta} (- a (\omega+ \mu)^{\delta}), \quad \omega\in [-\mu, 0],
  \end{array}
  \end{equation*}
\begin{equation*}
\begin{array}{c}
\tilde{\varphi}(\omega) :=  \sum\limits_{q=1}^{n_{2}} | \tilde{\varphi}_{q}(\omega)   | \leq v_{0} E_{\delta} (- \bar{a} (\omega+ \mu)^{\delta}), \quad \omega\in [-\mu, 0].
\end{array}
\end{equation*}
Choosing  $\eta > 0$    such that
  $ \Omega \eta \kappa  \hat{h}^{*}  < \frac{1}{4}$
and $V_{0} \Lambda < \frac{\eta}{4}$, furthermore,  as $V(t)$ is continuous on $[0, t_{*}]$, we have $V(t)\leq \eta $  on $[0, t_{*}]$ with  $t_{*} > 0.$
If $t_{*} \leq \mu$, and $0 < t \leq t_{*}$, then
\begin{equation}\label{esnl1}
\begin{array}{c}
    \int_{0}^{t} (t-\omega)^{\delta-1} E_{\delta, \delta} (- \xi (t-\omega)^{\delta} )
   \int_{0}^{\infty} k^{*}(\lambda) V(\omega-\lambda)   d \lambda    \int_{0}^{\infty} h^{*}(\lambda) V(\omega- \lambda)   d \lambda d\omega \\
    \leq  \eta \hat{h}^{*}   \int_{0}^{t} (t-\omega)^{\delta-1} E_{\delta, \delta} (- \xi(t-\omega)^{\delta} )
   \int_{0}^{\infty} k^{*}(\lambda) V(\omega-\lambda)   d \lambda.
\end{array}
\end{equation}
In view of the estimations (\ref{fnl5}) and  (\ref{esnl1}), we obtain
\begin{equation*}
\begin{array}{c}
 V(t) \leq  (1+c^{*}) V_{0} E_{\delta}(- \xi t^{\delta})  +  c^{*} V(t-\mu)
 + c^{*} V_{0} A \int_{0}^{t} (t-\omega)^{\delta-1} E_{\delta, \delta}(-  \xi (t-\omega)^{\delta}) E_{\delta}(- \xi \omega^{\delta})  d\omega \\
+   \Big( \pi + \eta \kappa  \hat{h}^{*} \Big)  \Big( \int_{0}^{t} (t-\omega)^{\delta-1} E_{\delta, \delta}(- \xi (t-\omega)^{\delta})   \int_{-\infty}^{\omega} K^{*}(\omega-\lambda)  E_{\delta}(- \xi \lambda^{\delta})  \frac{V(\lambda)}{E_{\delta}(- \xi \lambda^{\delta})}  d \lambda d\omega\Big).
   \end{array}
\end{equation*}
Therefore
\begin{equation*}
\begin{array}{c}
 V(t) \leq  \Big[
  1+ 2c^{*}  + c^{*} A \frac{\Gamma(1+\delta) \Gamma(1-\delta) }{ \xi} \Big]V_{0} E_{\delta}(- \xi t^{\delta})
   +    \Big( \pi + \eta \kappa  \hat{h}^{*} \Big)   \sup\limits_{- \infty < \lambda \leq t} \frac{V(\lambda)}{E_{\delta}(-   \xi \lambda^{\delta}) }\\
   \times \Big(  \int_{0}^{t} (t-\omega)^{\delta-1}
   E_{\delta, \delta}(- \xi (t-\omega)^{\delta})   \int_{-\infty}^{\omega}
    K^{*}(\omega-\lambda )E_{\delta}(-  \xi  \lambda^{\delta})   d \lambda d\omega\Big).
\end{array}
\end{equation*}
In light of the relation (\ref{k1}), we end up with
\begin{equation}\label{fnl7}
\begin{array}{c}
 V(t) \leq  \Big[
  1+ 2c^{*}  + c^{*} A \frac{\Gamma(1+\delta) \Gamma(1-\delta) }{ \xi}  \Big]V_{0} E_{\delta}(- \xi t^{\delta})
  +  \Omega  \Big( \pi + \eta \kappa  \hat{h}^{*} \Big)
    \sup\limits_{- \infty < \lambda \leq t} \frac{V(\lambda)}{E_{\delta}(-  \xi \lambda^{\delta}) }E_{\delta}(- \xi t^{\delta}).
\end{array}
\end{equation}
Dividing by $E_{\delta}(- \xi t^{\delta})$, (\ref{fnl7}) yields
\begin{equation*}
\begin{array}{c}
 \sup\limits_{- \infty < \lambda \leq t} \frac{ V(\lambda)}{E_{\delta}(- \xi \lambda^{\delta}) }
 \leq  \Big[
  1+ 2c^{*}  + A c^{*} \frac{\Gamma(1+\delta) \Gamma(1-\delta) }{ \xi } \Big]V_{0} + \Omega  \Big( \pi + \eta \kappa  \hat{h}^{*} \Big) \sup\limits_{- \infty < \lambda \leq t} \frac{ V(\lambda)}{E_{\delta}(- \xi \lambda^{\delta}) },
\end{array}
\end{equation*}
or
\begin{equation}\label{fnl11}
\begin{array}{c}
  \Big[1-  \Omega  \Big( \pi + \eta \kappa  \hat{h}^{*} \Big)   \Big] \sup\limits_{- \infty < \lambda \leq t} \frac{ V(\lambda)}{E_{\delta}(- \xi \lambda^{\delta}) }
 \leq  \Big[
  1+ 2c^{*}  + A c^{*} \frac{\Gamma(1+\delta) \Gamma(1-\delta) }{ \xi } \Big]V_{0}: = \Lambda_{1} V_{0},
\end{array}
\end{equation}
with
\begin{equation*}
\begin{array}{c}
  \Lambda_{1} :=  1+ 2 c^{*} + A  c^{*}  \frac{\Gamma(1+\delta)  \Gamma(1-\delta)}{\xi}. 
\end{array}
\end{equation*}
In view of the previous  assumptions  $\Omega  \Big( \pi + \eta \kappa  \hat{h}^{*} \Big)   < \frac{1}{2},$
the term $\Big[1- \Omega \Big( \pi+ \eta \kappa  \hat{h}^{*} \Big) \Big]$ is positive, and
\begin{equation}\label{sa1}
\begin{array}{c}
 V(t) \leq 2 V_{0} \Lambda_{1} E_{\delta} (-\xi t^{\delta}), \; t\in[0, \mu].
\end{array}
\end{equation}
As $V_{0} \Lambda  < \frac{\eta}{4}$  implies $V(t_{*}) < \frac{\eta}{2}$, the process can  be continued.

For $t_{*} \in ( \mu, 2 \mu]$ and $ \mu \leq t \leq t_{*}, \; 0 \leq t- \mu \leq  \mu $, notice  that (\ref{mc1}) gives
\begin{equation*}
\begin{array}{c}
  \frac{E_{\delta}(- \xi  (t- \mu)^{\delta})}{E_{\delta}(- \xi t^{\delta})} \leq
   \frac{1}{E_{\delta}(- \xi t^{\delta})} \leq \frac{1}{E_{\delta}(- \xi  (2 \mu)^{\delta})} \leq 1+ \xi \Gamma(1-\delta) (2\mu)^{\delta} = : B
\end{array}
\end{equation*}
and
\begin{equation*}
\begin{array}{c}
  V(t-\mu)
   \leq  2 V_{0} \Lambda_{1} \frac{E_{\delta}(- \xi (t- \mu)^{\delta})}{E_{\delta}(- \xi t^{\delta})} E_{\delta}(- \xi t^{\delta}) \leq 2 V_{0} \Lambda_{1} B E_{\delta}(- \xi t^{\delta}).
\end{array}
\end{equation*}
Returning  to (\ref{fnl5}), we infer that
\begin{equation*}
\begin{array}{c}
  V(t) \leq (1+c^{*}) V_{0} E_{\delta}(- \xi t^{\delta}) + 2 c^{*} V_{0} \Lambda_{1} B E_{\delta}(- \xi  t^{\delta})
   + 2 c^{*} V_{0}  \Lambda_{1} B A \int_{0}^{t} (t-\omega)^{\delta-1}
  E_{\delta, \delta}(- \xi (t-\omega)^{\delta})\\ \times  E_{\delta}(- \xi \omega^{\delta}) d\omega
+   \pi  \int_{0}^{t} (t-\omega)^{\delta-1} E_{\delta, \delta}(- \xi (t-\omega)^{\delta}) \int_{-\infty}^{\omega}  K^{*}(\omega-\lambda) V(\lambda)  d \lambda d\omega  \\
  +  \hat{h}^{*} \eta \kappa \int_{0}^{t} (t-\omega)^{\delta-1} E_{\delta, \delta} (- \xi (t-\omega)^{\delta} )
   \int_{-\infty}^{\omega} k^{*}(\omega-\lambda) V(\lambda)   d \lambda,
   \end{array}
\end{equation*}
and therefore
\begin{equation}\label{fnll13}
\begin{array}{c}
  V(t) \leq   V_{0} \Big[ (1+  c^{*} ) + 2 c^{*} \Lambda_{1}  B    \Big] E_{\delta}(- \xi t^{\delta})+ 2 V_{0} \Lambda_{1} B A c^{*}  \frac{\Gamma(1+\delta) \Gamma(1-\delta)}{\xi }  E_{\delta}(- \xi t^{\delta})     \\
  +  \Big( \pi + \eta \kappa  \hat{h}^{*} \Big)   \int_{0}^{t} (t-\omega)^{\delta-1} E_{\delta, \delta}(- \xi (t-\omega)^{\delta})
     \int_{-\infty}^{\omega}
   K^{*}(\omega-\lambda) E_{\delta}(- \xi \lambda^{\delta}) \frac{V(\lambda)}{E_{\delta}(- \xi \lambda^{\delta})}  d \lambda d\omega.
   \end{array}
   \end{equation}
The estimation (\ref{fnll13}) implies
\begin{equation}\label{fnll15}
\begin{array}{c}
  \frac{V(t)}{E_{\delta}(- \xi  t^{\delta})} \leq   V_{0} \Big[ (1+ c^{*} ) + 2 c^{*}  \Lambda_{1} B    \Big] +  2 V_{0} \Lambda_{1} B A  c^{*} \frac{\Gamma(1+\delta) \Gamma(1-\delta)}{\xi}
   +  \Omega \Big( \pi+ \eta \kappa  \hat{h}^{*} \Big)  \sup\limits_{- \infty < \lambda \leq t} \frac{V(\lambda)}{E_{\delta}(- \xi  \lambda^{\delta}) }.
\end{array}
\end{equation}
Hence
\begin{equation*}
\begin{array}{c}
  \Big[1- \Omega \Big( \pi + \eta \kappa  \hat{h}^{*} \Big) \Big] \frac{V(t)}{E_{\delta}(- \xi t^{\delta})} \leq   V_{0} \Big[ (1+ c^{*} ) + 2 c^{*} \Lambda_{1} B    \Big] + 2 V_{0} \Lambda_{1} B A c^{*} \frac{\Gamma(1+\delta) \Gamma(1-\delta)}{\xi}.
\end{array}
\end{equation*}
As a consequence
\begin{equation}\label{fnl16}
\begin{array}{c}
  V(t) \leq 2 V_{0} \Lambda_{2} E_{\delta}(- \xi t^{\delta}),
  \end{array}
\end{equation}
where
\begin{equation*}
\begin{array}{c}
  \Lambda_{2} := (1+ c^{*} ) +2 c^{*}  \Lambda_{1} B + 2 c^{*}  \Lambda_{1} B A \frac{\Gamma(1+\delta) \Gamma(1-\delta)}{ \xi }.
\end{array}
\end{equation*}
In case $t_{*} \in (2 \mu, 3 \mu]$, and $2 \mu< t \leq  t_{*}, \;  \mu \leq t- \mu \leq 2 \mu $, from (\ref{fnl5}), we have
\begin{equation}\label{fnll19}
\begin{array}{c}
  V(t) \leq   (1+c^{*}) V_{0} E_{\delta}(- \xi  t^{\delta})  + 2 c^{*} V_{0} \Lambda_{2}  E_{\delta}(- \xi t^{\delta})
  + c^{*} A \int_{0}^{t} (t-\omega)^{\delta-1}
  E_{\delta, \delta}(- \xi  (t-\omega)^{\delta})\\ \times  V(\omega-\mu) d\omega
  +  \sup\limits_{-\infty < \lambda \leq t} \frac{V(\lambda)}{ E_{\delta} (- \xi \lambda^{\delta})}
  \Big( \pi + \eta \kappa  \hat{h}^{*} \Big) \int_{0}^{t} (t-\omega)^{\delta-1} E_{\delta, \delta}(- \xi  (t-\omega)^{\delta})   \\ \times  \int_{-\infty}^{\omega}
   K^{*}(\omega-\lambda) E_{\delta}(- \xi \lambda^{\delta})  d \lambda d\omega.
   \end{array}
   \end{equation}
In accordance with  the estimations (\ref{sa1}) and (\ref{fnl16}), we conclude that
\begin{equation*}
\begin{array}{c}
  \int_{0}^{t} (t-\omega)^{\delta-1} E_{\delta, \delta}(- \xi  (t-\omega)^{\delta}) V(\omega-\mu) d\omega
  \leq V_{0} \int_{0}^{\mu} (t-\omega)^{\delta-1}
    E_{\delta, \delta}(- \xi (t-\omega)^{\delta}) E_{\delta}(- \xi \omega^{\delta})  d\omega   \\
   + 2 V_{0} \Lambda_{1} \int_{\mu}^{2 \mu} (t-\omega)^{\delta-1} E_{\delta, \delta}(- \xi  (t-\omega)^{\delta}) E_{\delta}(- \xi (\omega- \mu)^{\delta}) d\omega\\
   +  2 V_{0} \Lambda_{2}
  \int_{2 \mu}^{t} (t-\omega)^{\delta-1} E_{\delta, \delta}(- \xi  (t-\omega)^{\delta}) E_{\delta}(- \xi (\omega-\mu)^{\delta})d\omega.
\end{array}
\end{equation*}
Therefore
\begin{equation*}
\begin{array}{c}
    \int_{0}^{t} (t-\omega)^{\delta-1} E_{\delta, \delta}(- \xi (t-\omega)^{\delta}) V(\omega-\mu) d\omega
   \leq V_{0} \int_{0}^{\mu} (t-\omega)^{\delta-1}
    E_{\delta, \delta}(- \xi (t-\omega)^{\delta}) E_{\delta}(- \xi \omega^{\delta})  d\omega   \\
    + 2 V_{0} B \Lambda_{1} \int_{\mu}^{2 \mu} (t-\omega)^{\delta-1} E_{\delta, \delta}(- \xi  (t-\omega)^{\delta}) E_{\delta}(- \xi \omega^{\delta}) d\omega\\
  +  2 V_{0} F \Lambda_{2}
  \int_{2 \mu}^{t} (t-\omega)^{\delta-1} E_{\delta, \delta}(- \xi (t-\omega)^{\delta}) E_{\delta}(- \xi  \omega^{\delta}) d\omega,
\end{array}
\end{equation*}
and
\begin{equation*}
\begin{array}{c}
   \int_{0}^{t} (t-\omega)^{\delta-1} E_{\delta, \delta}(- \xi  (t-\omega)^{\delta}) V(\omega-\mu) d\omega
   \leq V_{0} \int_{0}^{\mu} (t-\omega)^{\delta-1}
    E_{\delta, \delta}(- \xi (t-\omega)^{\delta}) E_{\delta}(- \xi  \omega^{\delta})  d\omega   \\
   +2 V_{0}  \Lambda_{2} \max\{B, F \}   \int_{\mu}^{t} (t-\omega)^{\delta-1} E_{\delta, \delta}(- \xi  (t-\omega)^{\delta}) E_{\delta}(- \xi \omega^{\delta}) d\omega.
\end{array}
\end{equation*}
Moreover
\begin{equation}\label{fnl21}
\begin{array}{c}
  \int_{0}^{t} (t-\omega)^{\delta-1} E_{\delta, \delta}(-  \xi (t-\omega)^{\delta}) V(\omega-\mu) d\omega\\
  \leq
  2 V_{0} \Lambda_{2} \max\{B, F  \} \int_{0}^{t} (t-\omega)^{\delta-1}
    E_{\delta, \delta}(- \xi  (t-\omega)^{\delta}) E_{\delta}(- \xi \omega^{\delta})  d\omega   \\
   \leq  2 V_{0} \Lambda_{2} \max\{B, F  \} \frac{\Gamma(1+\delta) \Gamma(1-\delta)}{\xi} E_{\delta}(- \xi  t^{\delta})
  \leq 2 V_{0} \Lambda_{2} B^{*} \frac{\Gamma(1+\delta) \Gamma(1-\delta)}{\xi} E_{\delta}(- \xi t^{\delta}),
\end{array}
\end{equation}
with  $ F : = \frac{B}{ \xi \mu^{\delta}}$ and $ \max\{ B, F\}= B \max\{1,  \frac{1}{ \xi  \mu^{\delta}} \} = :B^{*}.$\\
In view of the relations (\ref{fnll19}) and (\ref{fnl21}), we find
\begin{equation}\label{fnl23}
\begin{array}{c}
  \frac{V(t)}{E_{\delta}(- \xi t^{\delta})} \leq   V_{0} \Big[ (1+ c^{*} ) + 2 c^{*} F \Lambda_{2}    \Big] + 2 V_{0} \Lambda_{2} B^{*} A c^{*}\frac{\Gamma(1+\delta) \Gamma(1-\delta)}{\xi }   + \Omega \Big( \pi + \eta \kappa  \hat{h}^{*} \Big)  \sup\limits_{- \infty < \lambda \leq t} \frac{V(\lambda)}{E_{\delta}(- \xi  \lambda^{\delta}) }.
\end{array}
\end{equation}
Then,
\begin{equation*}
\begin{array}{c}
  \Big[1-\Omega \Big(\pi+ \eta \kappa  \hat{h}^{*} \Big) \Big] \sup\limits_{- \infty < \lambda \leq t} \frac{V(\lambda)}{E_{\delta}(-  \xi \lambda^{\delta}) }   \leq V_{0} \Big[(1+ c^{*} ) + 2 c^{*}  F \Lambda_{2}   \Big]
  + 2 V_{0} \Lambda_{2} B^{*} A c^{*}  \frac{\Gamma(1+\delta) \Gamma(1-\delta)}{\xi }.
\end{array}
\end{equation*}
Consequently, we obtain
\begin{equation}\label{fnl25}
\begin{array}{c}
  V(t) \leq 2 \Lambda_{3} V_{0} E_{\delta} (- \xi t^{\delta}),
  \end{array}
\end{equation}
where $ \Lambda_{3} := (1+ c^{*} ) + 2 c^{*} B^{*} \Lambda_{2} + 2 \Lambda_{2} B^{*} A c^{*}  \frac{\Gamma(1+\delta) \Gamma(1-\delta)}{\xi }.$\\
Recalling that  $U := A \frac{\Gamma(1+\delta) \Gamma(1-\delta)}{\xi},
$
\begin{equation*}
\begin{array}{c}
  \Lambda_{1} = (1+ c^{*} ) + c^{*}(1+U),  \\
  \Lambda_{2} = (1+ c^{*} ) +2 c^{*}  B \Lambda_{1} (1+ U)
   = (1+ c^{*} ) + 2(1+ c^{*}) B \Big[ c^{*} (1+U)   \Big] + 2 B\Big[ c^{*} (1+U)    \Big]^{2}
\end{array}
\end{equation*}
and
\begin{equation*}
\begin{array}{c}
  \Lambda_{3} =  (1+c^{*})+ 2 c^{*} B^{*} \Lambda_{2} (1+ U)  \\
   \leq  (1+ c^{*} ) \Bigg\{  1+ \Big[ 2B^{*} c^{*}  (1+U)   \Big]  + \Big[ 2 B^{*}  c^{*} (1+U)   \Big]^{2} + \Big[ 2 B^{*} c^{*}  (1+U)   \Big]^{3}         \Bigg\}.
\end{array}
\end{equation*}
We claim that
\begin{equation}\label{fnr26}
  V(t) \leq  2 \Lambda_{k} V_{0} E_{\delta} (-  \xi  t^{ \delta}),\; (k-1) \mu< t \leq t_{*}, \; t_{*} \in ( (k-1) \mu, k \mu],  \; k \geq 1 ,
\end{equation}
with
\begin{equation*}
  \Lambda_{k}  \leq (1+c^{*}) \sum\limits_{l=0}^{k} \Big[ 2  B^{*} c^{*} (1+ U)  \Big]^{l}.
\end{equation*}
We proceed by induction to prove (\ref{fnr26}).
It is obvious  that (\ref{fnr26}) is valid for $k=1, 2, 3$.
Assume that the claim (\ref{fnr26})  is valid  for $k$. We want to prove it for $k+1$. Clearly, we obtain
\begin{equation*}
\begin{array}{c}
  \Lambda_{k+1} := (1+ c^{*} ) +2 B^{*} c^{*} \Lambda_{k} (1+U)
   \leq(1+ c^{*} ) + 2 B^{*} c^{*}  (1+ U) (1+ c^{*} ) \sum\limits_{l=0}^{k} \Big[  2 B^{*} c^{*}  (1+U) \Big]^{l} \\
  \leq  (1+ c^{*} )  \Bigg\{ 1 +  2 B^{*} c^{*} (1+ U ) \sum\limits_{l=0}^{k}  \Big[  2 B^{*} c^{*}  (1+U) \Big]^{l}    \Bigg\}
  \leq  (1+c^{*}  )  \Bigg\{ 1 +   \sum\limits_{l=0}^{k}  \Big[  2 B^{*} c^{*}   (1+U) \Big]^{l+1}    \Bigg\}\\
   \leq   (1+c^{*} )   \sum\limits_{l=0}^{k+1}  \Big[  2 B^{*}c^{*}   (1+U) \Big]^{l}.
\end{array}
\end{equation*}
Hence
$$  V(t)  \leq    2 \Lambda V_{0}     E_{\delta} (- \xi  t^{\delta}), \quad   t> 0, $$
with
\begin{equation*}\label{ser1}
\Lambda:= \sum\limits_{l=0}^{+\infty}  \Big[  2 B^{*} c^{*} (1+U) \Big]^{l}.
\end{equation*}
According to the conditions stated  in Theorem \ref{th2}, $\Lambda$  is convergent.
\end{proof}

\section{Synchronization}
From the above stability results, the synchronization of  coupled systems can be derived. We refer to  system (\ref{hfn1}) as the uncontrolled system, and  the controlled system is given by
\begin{equation}\label{sl1}
\begin{array}{c}
D_{C}^{\delta}\Big[ z_{p}(t)-cz_{p}(t-\mu )\Big] =  -a_{p}z_{p}(t)
+\sum\limits_{q,s=1}^{n_{2}} d_{qps} \Big( \int_{0}^{\infty } k_{qps}(\omega)g_{q}( w_{q}(t-\omega)) d\omega \\
 \times \int_{0}^{\infty
}h_{qps}(\omega) g_{s}( w_{s}(t-\omega)) d\omega \Big) + I_{p} + \chi_{p}(t),\; t>0, \\
D_{C}^{\delta}\Big[ w_{q}(t)- \bar{c} w_{q}(t-\mu )\Big] =   -\bar{a}_{q} w_{q}(t)
+\sum\limits_{p, r=1}^{n_{1}} \bar{d}_{pqr} \Big(\int\nolimits_{0}^{%
\infty } \bar{k}_{pqr}(\omega) \bar{g}_{p}( z_{p}(t-\omega)) d\omega  \\
 \times  \int\nolimits_{0}^{\infty
}\bar{h}_{pqr}(\omega)\bar{g}_{r}( z_{r}(t-\omega)) d\omega \Big)+ J_{q} + \bar{\chi}_{q}(t),\; t>0,
\end{array}
\end{equation}
 where $\chi_{p}(t)$ and $\bar{\chi}_{q}(t)$ are the feedback controls.\\
Besides, the error system is defined  by
\begin{equation}\label{er1}
\begin{array}{c}
D_{C}^{\delta}\Big[ e_{p}(t)-c e_{p}(t-\mu )\Big]= -a_{p} e_{p}(t)
+\sum\limits_{q,s=1}^{n_{2}} d_{qps}\int_{0}^{\infty } k_{qps}(\omega)g_{q}( w_{q}(t-\omega)) d\omega \\
 \times \int_{0}^{\infty}h_{qps}(\omega) g_{s}( w_{s}(t-\omega)) d\omega - \sum\limits_{q,s=1}^{n_{2}} d_{qps}\int_{0}^{\infty } k_{qps}(\omega)g_{q}( y_{q}(t-\omega)) d\omega \\
 \times \int_{0}^{\infty}h_{qps}(\omega) g_{s}( y_{s}(t-\omega)) d\omega   + \chi_{p}(t),\; t>0, \\
 D_{C}^{\delta}\Big[ \bar{e}_{q}(t)- \bar{c} \bar{e}_{q}(t-\mu )\Big] = -\bar{a}_{q} \bar{e}_{q}(t)
+\sum\limits_{p, r =1}^{n_{1}}  \bar{d}_{pqr}\int\nolimits_{0}^{%
\infty } \bar{k}_{pqr}(\omega) \bar{g}_{p}( z_{p}(t-\omega)) d\omega  \\
  \times \int\nolimits_{0}^{\infty
}\bar{h}_{pqr}(\omega)\bar{g}_{r}( z_{r}(t-\omega)) d\omega
- \sum\limits_{p,r=1}^{n_{1}} \bar{d}_{pqr}\int\nolimits_{0}^{%
\infty } \bar{k}_{pqr}(\omega) \bar{g}_{p}( x_{p}(t-\omega)) d\omega  \\
  \times \int\nolimits_{0}^{\infty
}\bar{h}_{pqr}(\omega)\bar{g}_{r}( x_{r}(t-\omega)) d\omega  + \bar{\chi}_{q}(t),\; t>0,
\end{array}
\end{equation}
where $e_{p}(t):= z_{p}(t)- x_{p}(t)$ and  $\bar{e}_{q}(t):= w_{q}(t)- y_{q}(t)$ are the synchronization errors.\\
Let the controls be
\begin{eqnarray*}
  \chi_{p}(t) &:=&  -\beta (z_{p}(t)- x_{p}(t)), \quad \beta > 0, \\
  \bar{\chi}_{q}(t) &:=& -\bar{\beta} (w_{q}(t)- y_{q}(t)), \quad \bar{\beta} > 0.
\end{eqnarray*}
 Then, the synchronization of  systems (\ref{hfn1}) and (\ref{sl1}) boils down to the  stability shown in the preceding sections. Moreover, the convergence rate is enhanced  by adding these  negative feedbacks to  the dissipation coefficients $a_{p}$ and $\bar{a}_{q}$.
\section{Numerical illustration }
Two examples of higher-order fractional BAM NNs will be given to  validate the previous  theoretical findings.
\begin{exm}
We consider  the following  fractional higher-order   BAM NN system
\begin{equation}\label{examp1}
\begin{array}{c}
  D^{b}_{c}[x_{1}(t)-cx_{1}(t-\mu)] = -a_{1} x_{1}(t)  +  \sum\limits_{q,s =1}^{2} d_{q1s} \int_{0}^{t} k_{q1s}(\omega) g_{q}(y_{q}(t-\omega))d\omega\\
  \int_{0}^{t} h_{q1s}(\omega) g_{s}(y_{s}(t-\omega))d\omega + 1,\\
 D^{b}_{c}[x_{2}(t)-cx_{2}(t-\mu)] = -a_{2} x_{2}(t)  +  \sum\limits_{q,s=1}^{2} d_{q2s} \int_{0}^{t} k_{q2s}(\omega) g_{q}(y_{q}(t-\omega))d\omega
 \\
  \int_{0}^{t} h_{q2s}(\omega) g_{s}(y_{s}(t-\omega))d\omega   +0.75, \\
  D^{b}_{c}[y_{1}(t)-\bar{c}y_{1}(t-\mu)] = -\bar{a}_{1} y_{1}(t)  +  \sum\limits_{p, r=1}^{2} \bar{d}_{p1r} \int_{0}^{t} \bar{k}_{p1r}(\omega) \bar{g}_{p}(x_{p}(t-\omega))d\omega \\
  \int_{0}^{t} \bar{h}_{p1r}(\omega) \bar{g}_{r}(x_{r}(t-\omega))d\omega   + 0.5, \\
  D^{b}_{c}[y_{2}(t)-\bar{c}y_{2}(t-\mu)] =  -\bar{a}_{2} y_{2}(t)  +  \sum\limits_{p,r =1}^{2}\bar{d}_{p2r} \int_{0}^{t} \bar{k}_{p2r}(\omega) \bar{g}_{p}(x_{p}(t-\omega))d\omega \\
     \int_{0}^{t} \bar{h}_{p2r}(\omega) \bar{g}_{r}(x_{r}(t-\omega))d\omega  + 1, \\
    x_{1}(t) = -0.5, \quad x_{2}(t)=-1, \; t\in[-10, 0] \\
y_{1}(t) = -0.75, \quad y_{2}(t)=  -1.5, \; t\in [-10, 0],
\end{array}
\end{equation}
where the coefficients and  functions for $p, q=1, 2, \; t\in[0, 10]$ are taken as
\begin{equation*}
\begin{array}{c}
a_{1}= 5, \; a_{2}= 7, \;  \bar{a}_{1}= 6, \; \bar{a}_{2}= 8, \;  g_{q}(x) =tanh(x), \; \bar{g}_{p}(x) = tanh(x),\\
 k_{qps}(t)= h_{qps}(t)= e^{-5t}, \;  \bar{k}_{pqr}(t)= \bar{h}_{pqr}(t)= e^{-6t}, \; r, s =1, 2,\\
 c = \bar{c}= 0.0001, \; \mu=1, \; b=0.9,\\
d_{111} = 1.3, \; d_{112}= 0.5, \; d_{211}=1, \; d_{212}=  0.25, \\
d_{121} = 0.75, \; d_{122}= 1,\;  d_{221}=   0.5, \; d_{222}=   0.4, \\
\bar{d}_{111} = 0.6, \; \bar{d}_{112}= 1, \; \bar{d}_{211}=0.5, \; \bar{d}_{212}=  0.25, \\
\bar{d}_{121} = 1, \; \bar{d}_{122}= 1.4, \;  \bar{d}_{221}=   0.75, \; \bar{d}_{222}=   1.25.
\end{array}
\end{equation*}
The controlled  system is described by
\begin{equation}\label{examp2}
\begin{array}{c}
  D^{b}_{c}[\bar{x}_{1}(t)-c\bar{x}_{1}(t-\mu)] = -a_{1} \bar{x}_{1}(t)  +  \sum\limits_{q,s=1}^{2} d_{q1s} \int_{0}^{t} k_{q1s}(\omega) g_{q}(\bar{y}_{q}(t-\omega))d\omega  \\
  \int_{0}^{t} h_{q1s}(\omega) g_{s}(\bar{y}_{s}(t-\omega))d\omega + 1  + \chi_{1}(\bar{x}_{1}(t)), \\
 D^{b}_{c}[\bar{x}_{2}(t)-c\bar{x}_{2}(t-\mu)] =  -a_{2} \bar{x}_{2}(t)  +  \sum\limits_{q,s=1}^{2} d_{q2s} \int_{0}^{t} k_{q2s}(\omega) g_{q}(\bar{y}_{q}(t-\omega))d\omega\\
   \int_{0}^{t} h_{q2s}(\omega) g_{s}(\bar{y}_{s}(t-\omega))d\omega   + 0.75 + \chi_{2}(\bar{x}_{2}(t)),  \\
  D^{b}_{c}[\bar{y}_{1}(t)-\bar{c}\bar{y}_{1}(t-\mu)] =  -\bar{a}_{1} \bar{y}_{1}(t)  +  \sum\limits_{p, r=1}^{2} \bar{d}_{p1r} \int_{0}^{t} \bar{k}_{p1r}(\omega) \bar{g}_{p}(\bar{x}_{p}(t-\omega))d\omega \\
   \int_{0}^{t} \bar{h}_{p1r}(\omega) \bar{g}_{r}(\bar{x}_{r}(t-\omega))d\omega   + 0.5 + \bar{\chi}_{1}(\bar{y}_{1}(t)), \\
  D^{b}_{c}[\bar{y}_{2}(t)-\bar{c}\bar{y}_{2}(t-\mu)] =  -\bar{a}_{2}\bar{y}_{2}(t)  +  \sum\limits_{p, r=1}^{2}  \bar{d}_{p2r} \int_{0}^{t} \bar{k}_{p2r}(\omega) \bar{g}_{p}(\bar{x}_{p}(t-\omega))d\omega\\
     \int_{0}^{t} \bar{h}_{p2r}(\omega) \bar{g}_{r}(\bar{x}_{r}(t-\omega))d\omega  + 1 +\bar{ \chi}_{2}(\bar{y}_{2}(t)),\\
     \bar{x}_{1}(t)= -1, \quad \bar{x}_{2}(t)= -1.75, \; t \in[-10, 0]\\
\bar{y}_{1}(t)= -1, \quad \bar{y}_{2}(t)=  -2, \; t \in[-10, 0],
\end{array}
\end{equation}
where
\begin{eqnarray*}
 \chi_{p}(t)&=& -\beta(\bar{x}_{p}(t)-x_{p}(t)), \quad  \bar{\chi}_{q}(t)= -\bar{\beta} (\bar{y}_{q}(t)-y_{q}(t)), \quad
 \beta = \bar{\beta}= 2.
\end{eqnarray*}
\begin{equation}\label{examp3}
\begin{array}{c}
  D^{b}_{c}[e_{1}(t)-ce_{1}(t-\mu)] =  -a_{1} e_{1}(t)  +  \sum\limits_{q,s=1}^{2} d_{q1s} \int_{0}^{t} k_{q1s}(\omega) g_{q}(\bar{y}_{q}(t-\omega))d\omega\\
     \int_{0}^{t} h_{q1s}(\omega) g_{s}(\bar{y}_{s}(t-\omega))d\omega - \sum\limits_{q,s=1}^{2} d_{q1s} \int_{0}^{t} k_{q1s}(\omega) g_{q}(y_{q}(t-\omega))d\omega\\
      \int_{0}^{t} h_{q1s}(\omega) g_{s}(y_{s}(t-\omega))d\omega +\chi_{1}(t)   , \\
 D^{b}_{c}[e_{2}(t)-ce_{2}(t-\mu)] = -a_{2} e_{2}(t)  +  \sum\limits_{q,s=1}^{2}  d_{q2s} \int_{0}^{t} k_{q2s}(\omega) g_{q}(\bar{y}_{q}(t-\omega))d\omega\\
     \int_{0}^{t} h_{q2s}(\omega) g_{s}(\bar{y}_{s}(t-\omega))d\omega -\sum\limits_{q,s=1}^{2} d_{q2s} \int_{0}^{t} k_{q2s}(\omega) g_{q}(y_{q}(t-\omega))d\omega\\
    \int_{0}^{t} h_{q2s}(\omega) g_{s}(y_{s}(t-\omega))d\omega  +\chi_{2}(t),  \\
  D^{b}_{c}[\bar{e}_{1}(t)-\bar{c}\bar{e}_{1}(t-\mu)] = -\bar{a}_{1} \bar{e}_{1}(t)  +  \sum\limits_{p,r=1}^{2} \bar{d}_{p1r} \int_{0}^{t} \bar{k}_{p1r}(\omega) \bar{g}_{p}(\bar{x}_{p}(t-\omega))d\omega \\
     \int_{0}^{t} \bar{h}_{p1r}(\omega) \bar{g}_{r}(\bar{x}_{r}(t-\omega))d\omega -\sum\limits_{p, r=1}^{2} \bar{d}_{p1r} \int_{0}^{t} \bar{k}_{p1r}(\omega) \bar{g}_{p}(x_{p}(t-\omega))d\omega \\
     \int_{0}^{t} \bar{h}_{p1r}(\omega) \bar{g}_{r}(x_{r}(t-\omega))d\omega    + \bar{\chi}_{1}(t) , \\
  D^{b}_{c}[\bar{e}_{2}(t)-\bar{c}\bar{e}_{2}(t-\mu)] =  -\bar{a}_{2} \bar{e}_{2}(t)  +  \sum\limits_{p, r =1}^{2} \bar{d}_{p2r} \int_{0}^{t} \bar{k}_{p2r}(\omega) \bar{g}_{p}(\bar{x}_{p}(t-\omega))d\omega\\
    \int_{0}^{t} \bar{h}_{p2r}(\omega) \bar{g}_{r}(\bar{x}_{r}(t-\omega))d\omega -\sum\limits_{p, r =1}^{2} \bar{d}_{p2r} \int_{0}^{t} \bar{k}_{p2r}(\omega) \bar{g}_{p}(x_{p}(t-\omega))d\omega\\
   \int_{0}^{t} \bar{h}_{p2r}(\omega) \bar{g}_{r}(x_{r}(t-\omega))d\omega  + \bar{\chi}_{2}(t),
\end{array}
\end{equation}
the  system (\ref{examp2}) is the error system, where
$$e_{p}(t)=\bar{x}_{p}(t)-x_{p}(t), \quad \bar{e}_{q}(t)=\bar{y}_{q}(t)-y_{q}(t).$$
Furthermore, we have
\begin{equation*}
\begin{array}{c}
 \bigg[1+ \Gamma(1+b) \Gamma(1-b)   \bigg] F \frac{a^{*} c^{*}}{\xi} = 0.02<  1,\\
    \Omega < 1- \bigg[1+ \Gamma(1+b) \Gamma(1-b)   \bigg] F \frac{a^{*} c^{*}}{\xi} =0.98.
\end{array}
\end{equation*}
Then, Theorem \ref{th1} is  applied.
On the other hand, for step $h=0.02$, Figures \ref{fig1}, \ref{fig2} and \ref{fig3} illustrate trajectories of the states $x_{1}(t), x_{2}(t), y_{1}(t)$ and $y_{2}(t)$, whilst, Figures \ref{fig4} and \ref{fig5} depict  trajectories of the error states $e_{1}(t), e_{2}(t), \bar{e}_{1}(t)$ and $\bar{e}_{2}(t)$.  The convergence  of solutions of system (\ref{examp1}) to the equilibrium  in $b-$Mittag-Leffler manner is shown by Figures \ref{fig1}, \ref{fig2} and \ref{fig3}. Besides, Figures \ref{fig4} and \ref{fig5} describe the convergence  of solutions of  system  (\ref{examp3}) to the zero state.
\end{exm}

\begin{exm}
For the unbounded case, we assume
\begin{equation*}
 g_{q}(x)= h_{q}(x) =  \bar{g}_{p}(x) =\bar{h}_{p}(x) = argsh(x),\; p,q= 1, 2, \\
\end{equation*}
and  we have
\begin{equation*}
\begin{array}{c}
c^{*} < \frac{1}{2 B^{*} (1+U)}= 0.0002,\quad \Omega  < \frac{1}{4 \pi}=  0.4.
\end{array}
\end{equation*}
Therefore,  the conditions stated in Theorem \ref{th2} are met. Moreover, for step $h=0.02$, trajectories of the states $x_{1}(t), x_{2}(t), y_{1}(t)$ and $y_{2}(t)$ are illustrated by  Figures \ref{fig6}, \ref{fig7} and \ref{fig8}, and Figures \ref{fig9} and \ref{fig10} describe trajectories of $e_{1}(t),e_{2}(t), \bar{e}_{1}(t)$ and $\bar{e}_{2}(t)$ for various data and values of $b$, respectively. The convergence  of solutions of systems (\ref{examp1}) and (\ref{examp2}) in $b-$Mittag Leffler type is depicted  by Figures \ref{fig6}, \ref{fig7}, \ref{fig8}, \ref{fig9} and \ref{fig10}.

\end{exm}
\newpage
\section{Conclusion}
We have considered a neural network system with some challenging features. Indeed, the system is of fractional order and furthermore it is of higher-order  in addition to the presence of  neutral delays. Both features are problematic. The difficulties caused by the fractional derivatives are overcome by the use of a fractional version of Halanay inequality. To get around the second obstacle, we have performed some appropriate manipulations and evaluations. The stability is shown to be of Mittag-Leffler type as is expected for fractional differential equations. The synchronization issue is obtained easily from our stability results through linear feedback controls.
\section*{Acknowledgments} The second author would like to thank King Fahd University of Petroleum and Minerals of its continuous support through project SB 181008.

\newpage
\begin{figure}[tbp]\centering
              \includegraphics[width=.75\textwidth]{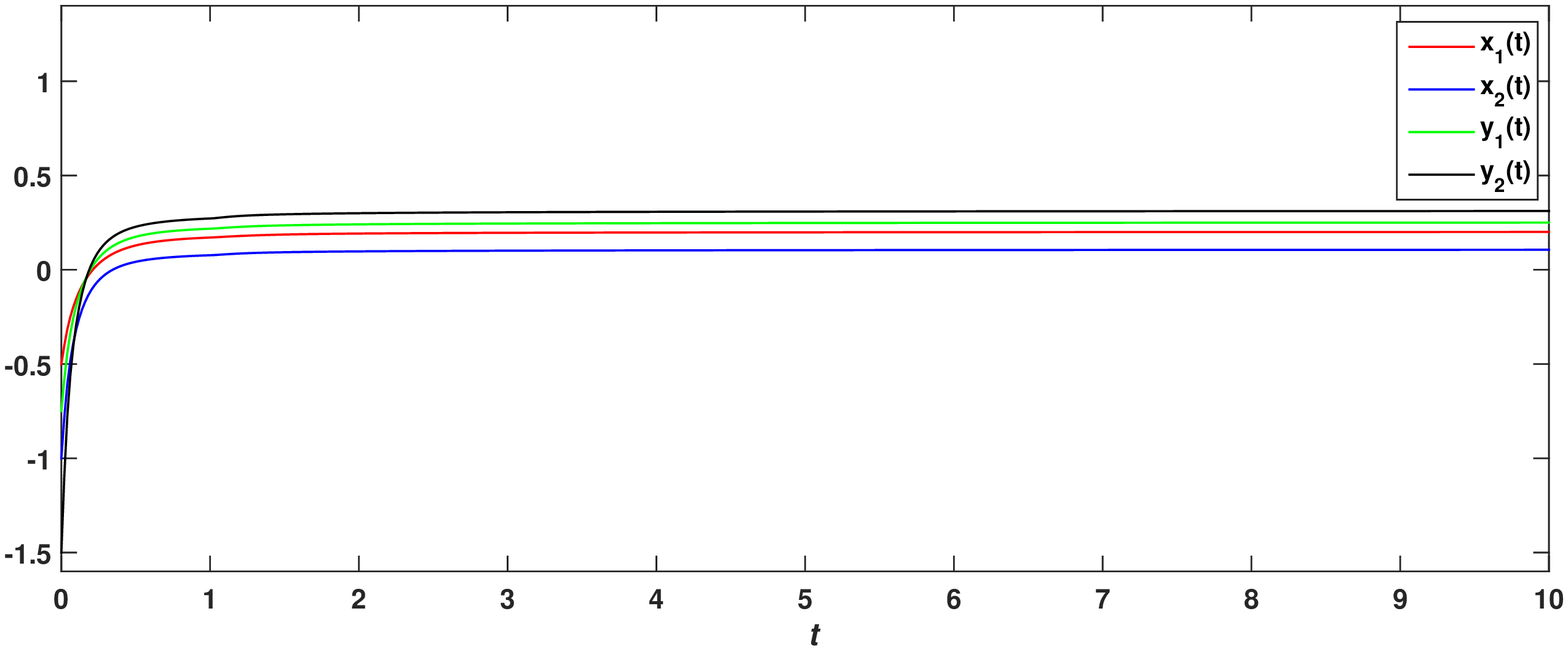}
\caption{Convergence  of the solutions  of system (\ref{examp1}) to the stationary state.}\label{fig1}
\end{figure}
\begin{figure}[h!]\centering
              \includegraphics[width=8cm, height=4cm]{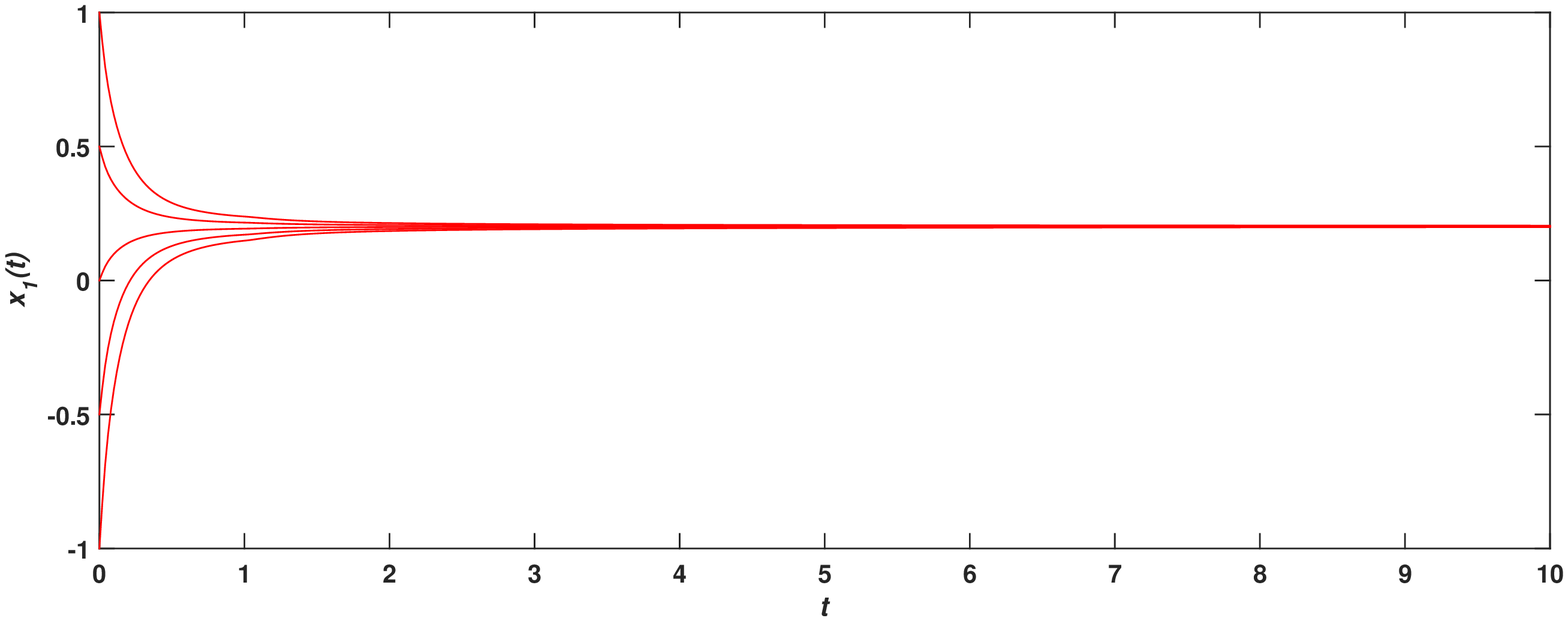}
        \qquad
           \includegraphics[width=8cm, height=4cm]{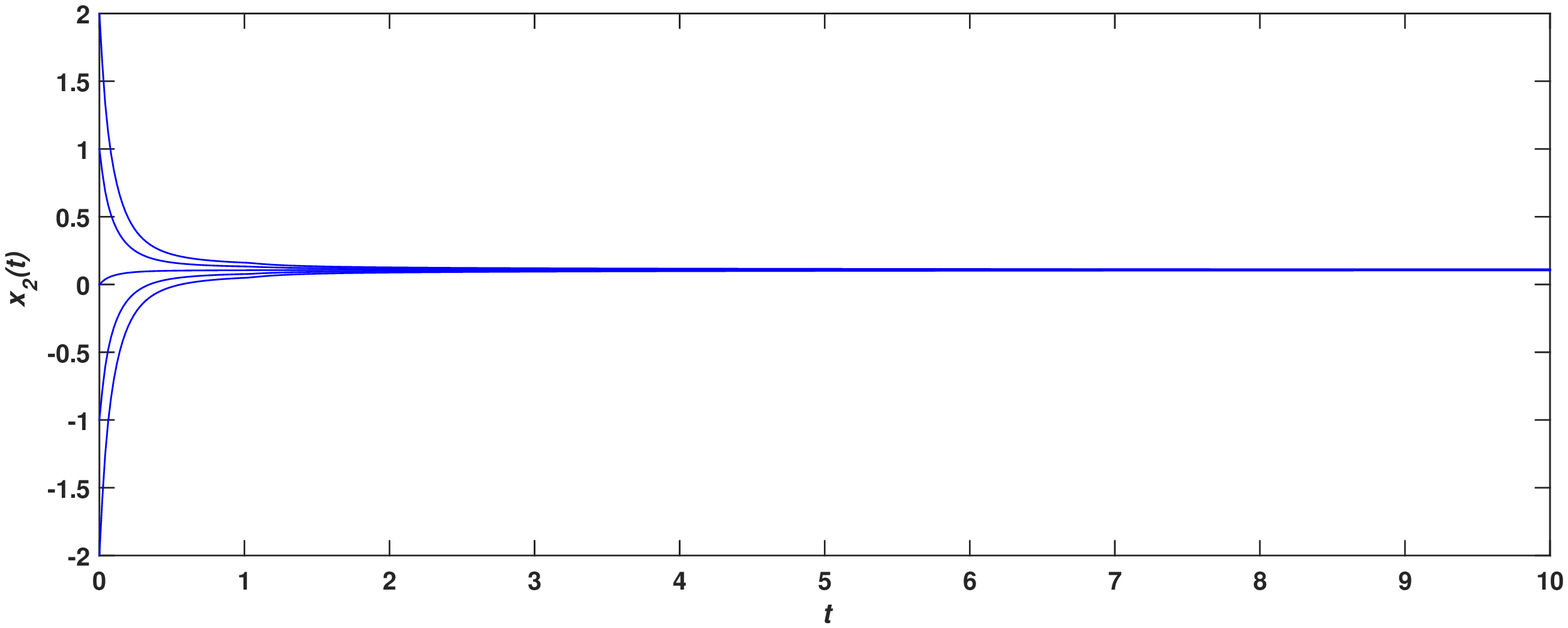}\\
              \includegraphics[width=8cm, height=4cm]{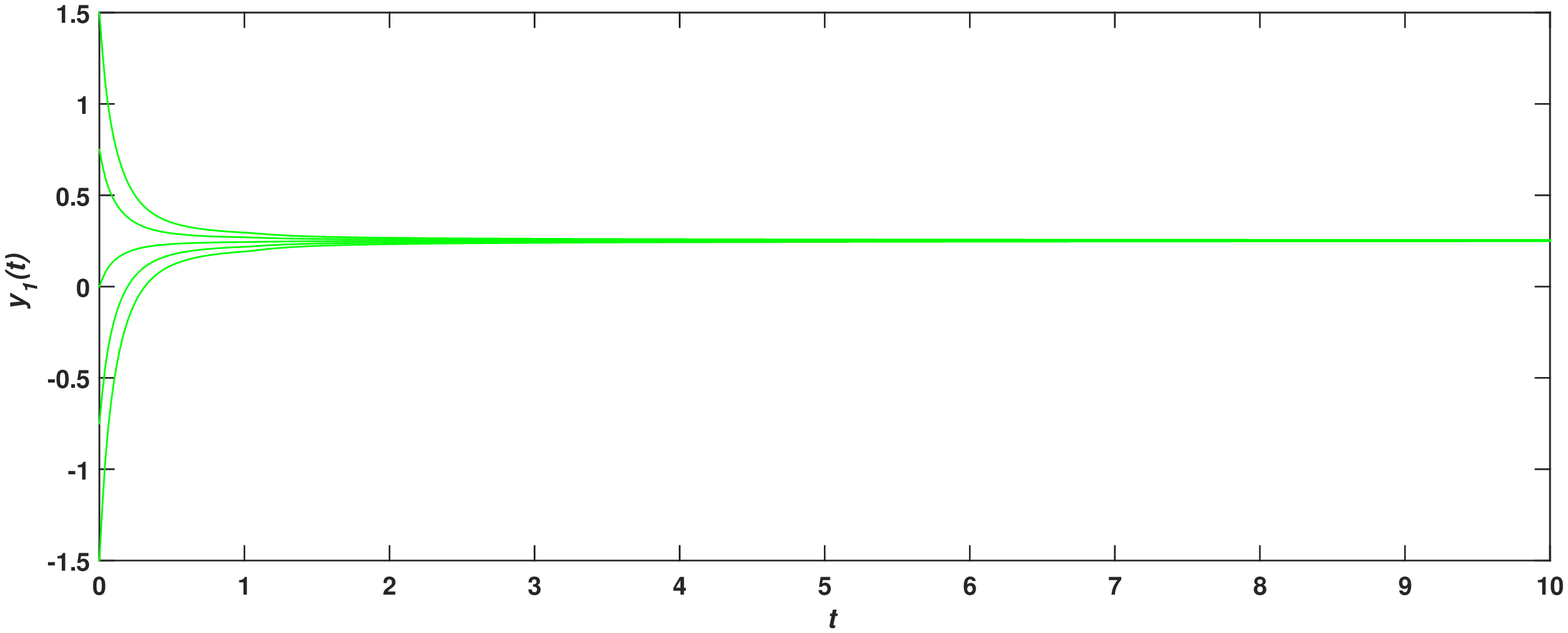}
        \qquad
           \includegraphics[width=8cm, height=4cm]{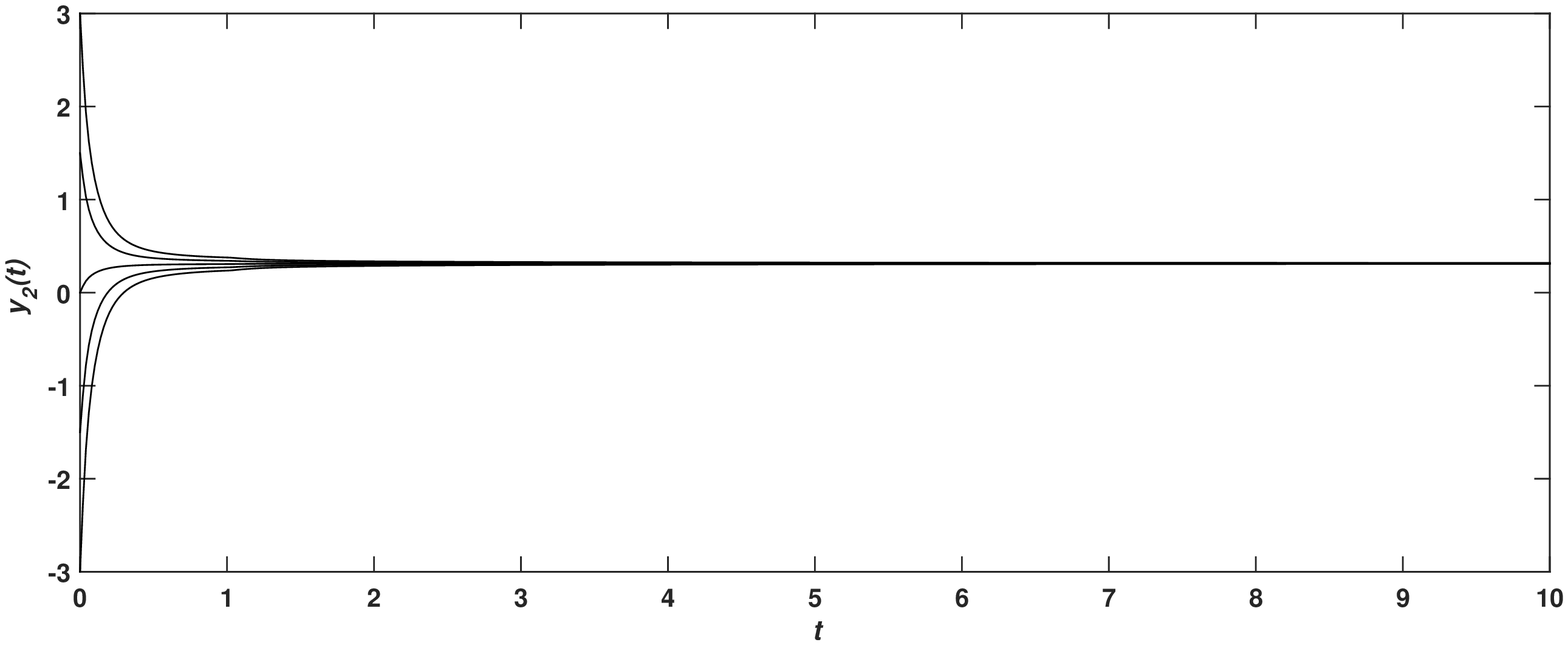}
\caption{Convergence  of   solutions of   system (\ref{examp1}) to the stationary state for various data.
}
\label{fig2}
\end{figure}
\newpage
\begin{figure} [h!]\centering
              \includegraphics[width=8cm, height=4cm]{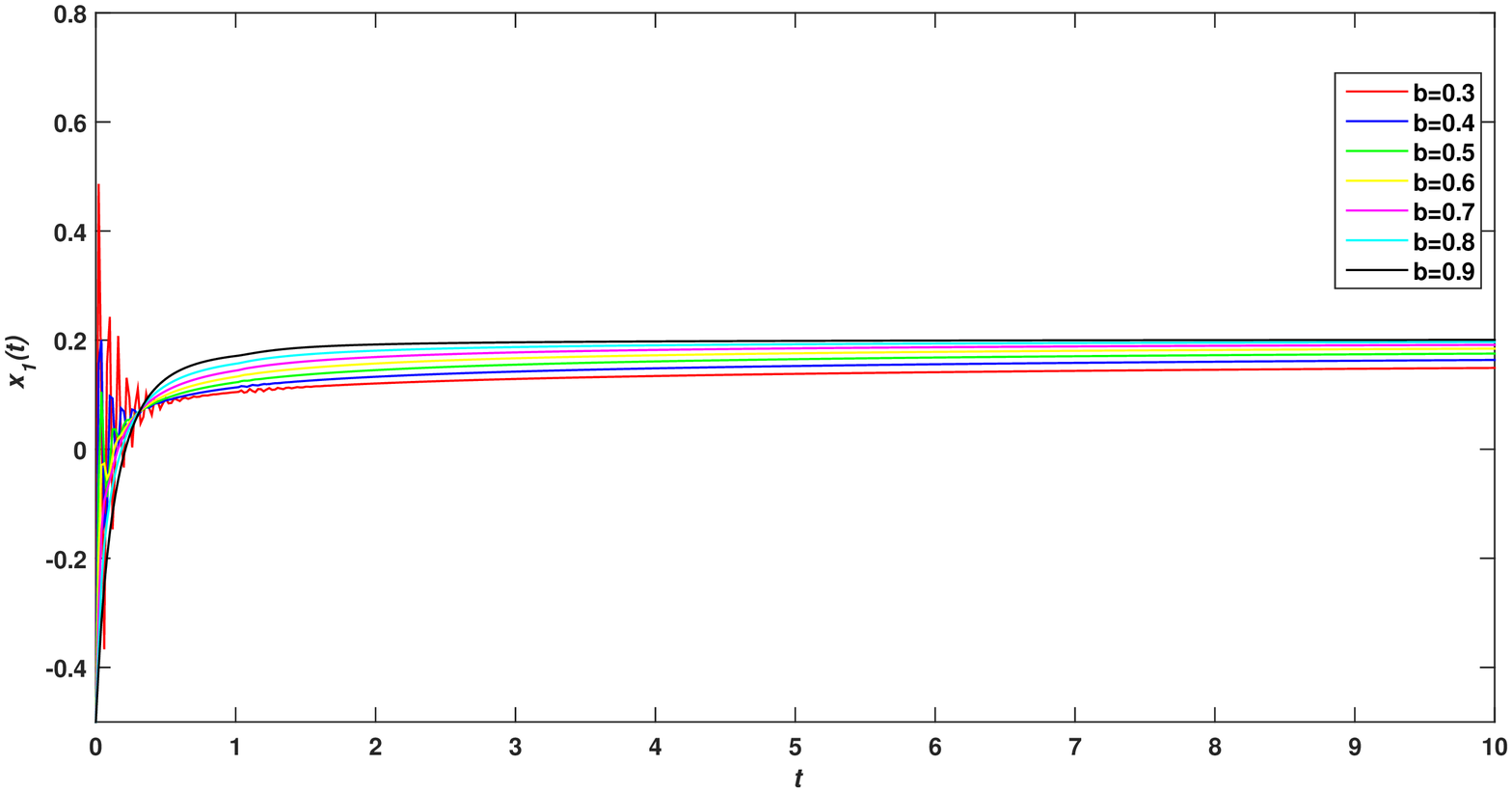}
        \qquad
           \includegraphics[width=8cm, height=4cm]{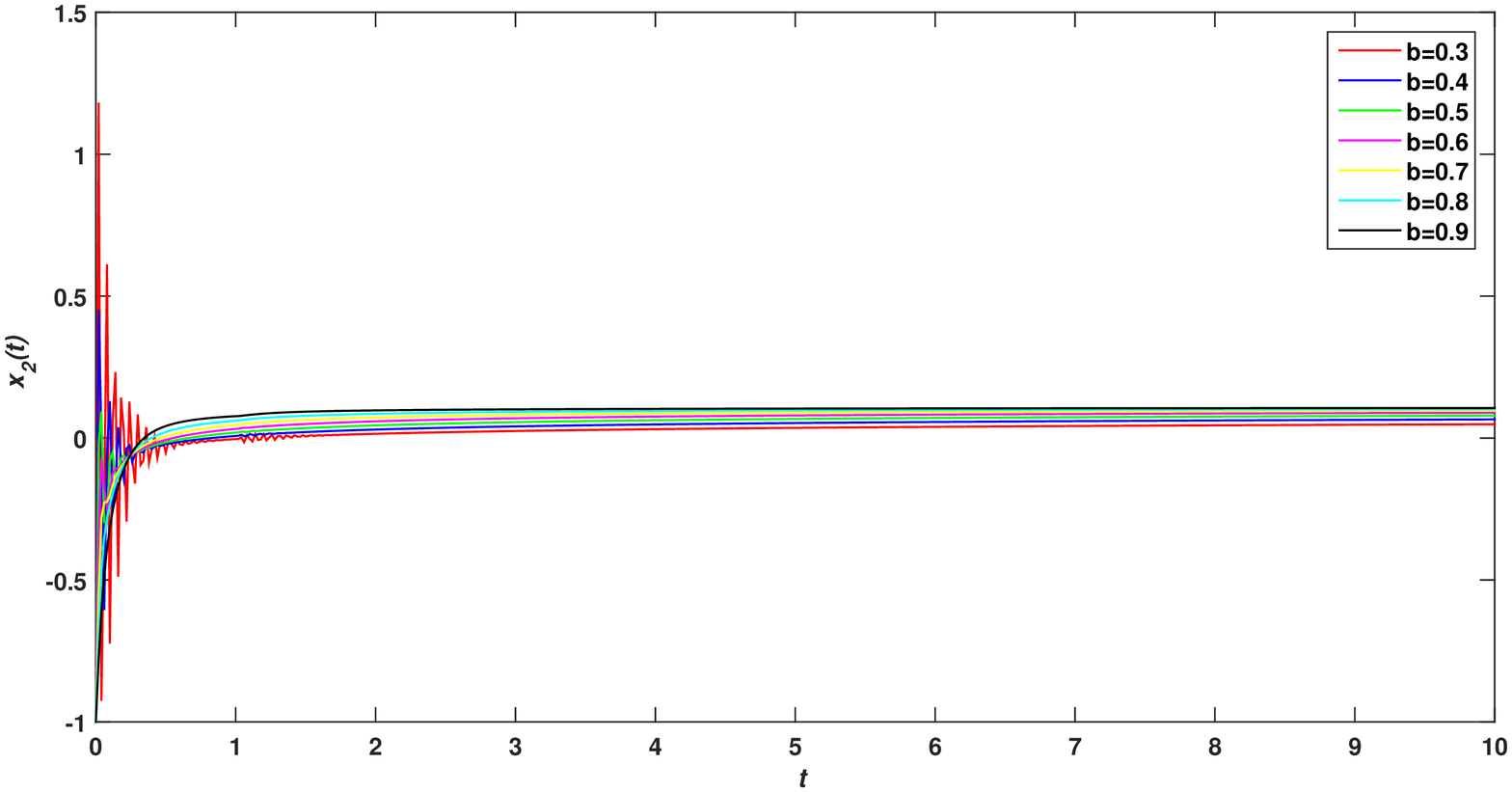}
\\
              \includegraphics[width=8cm, height=4cm]{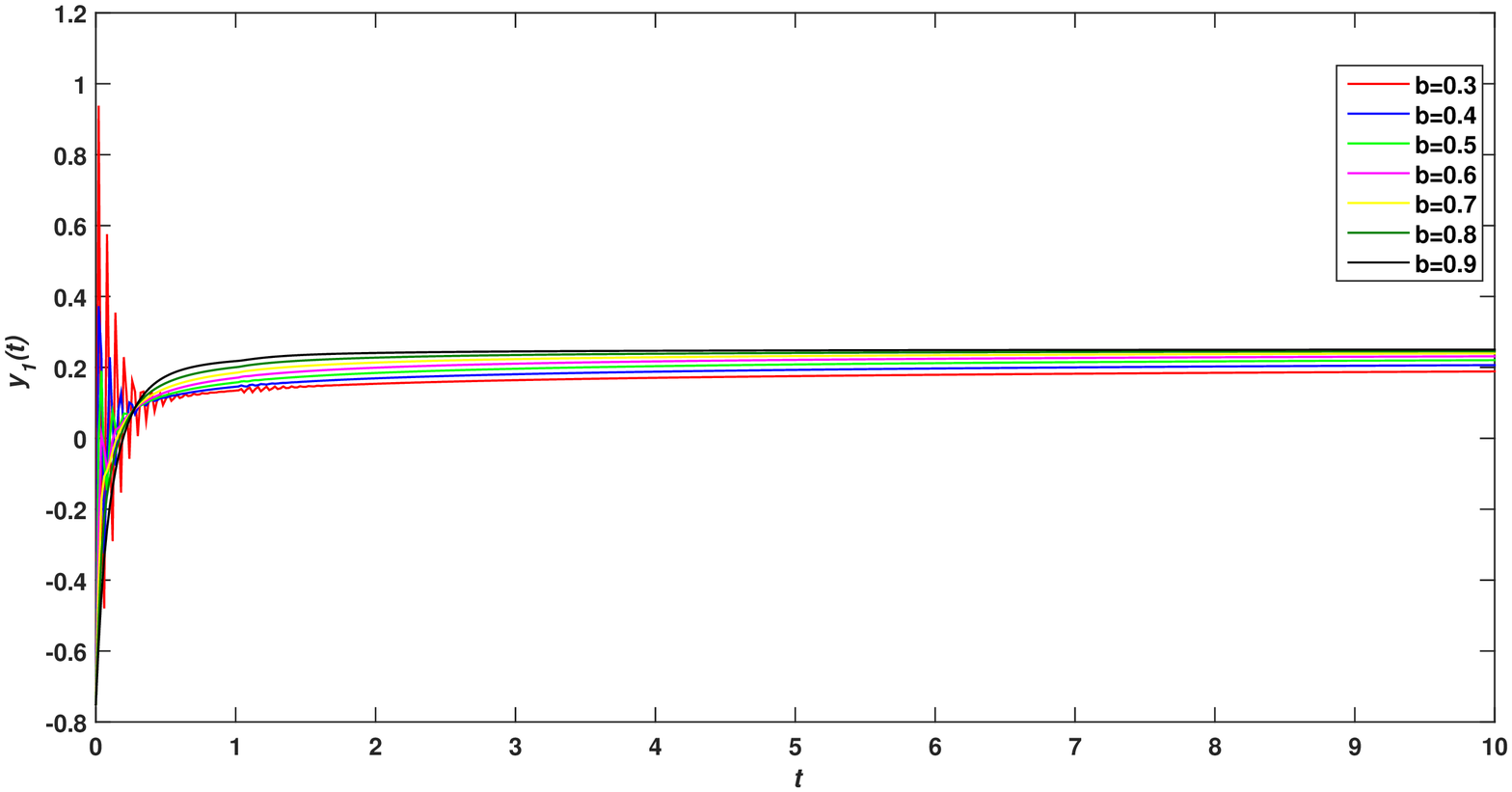}
        \qquad
           \includegraphics[width=8cm, height=4cm]{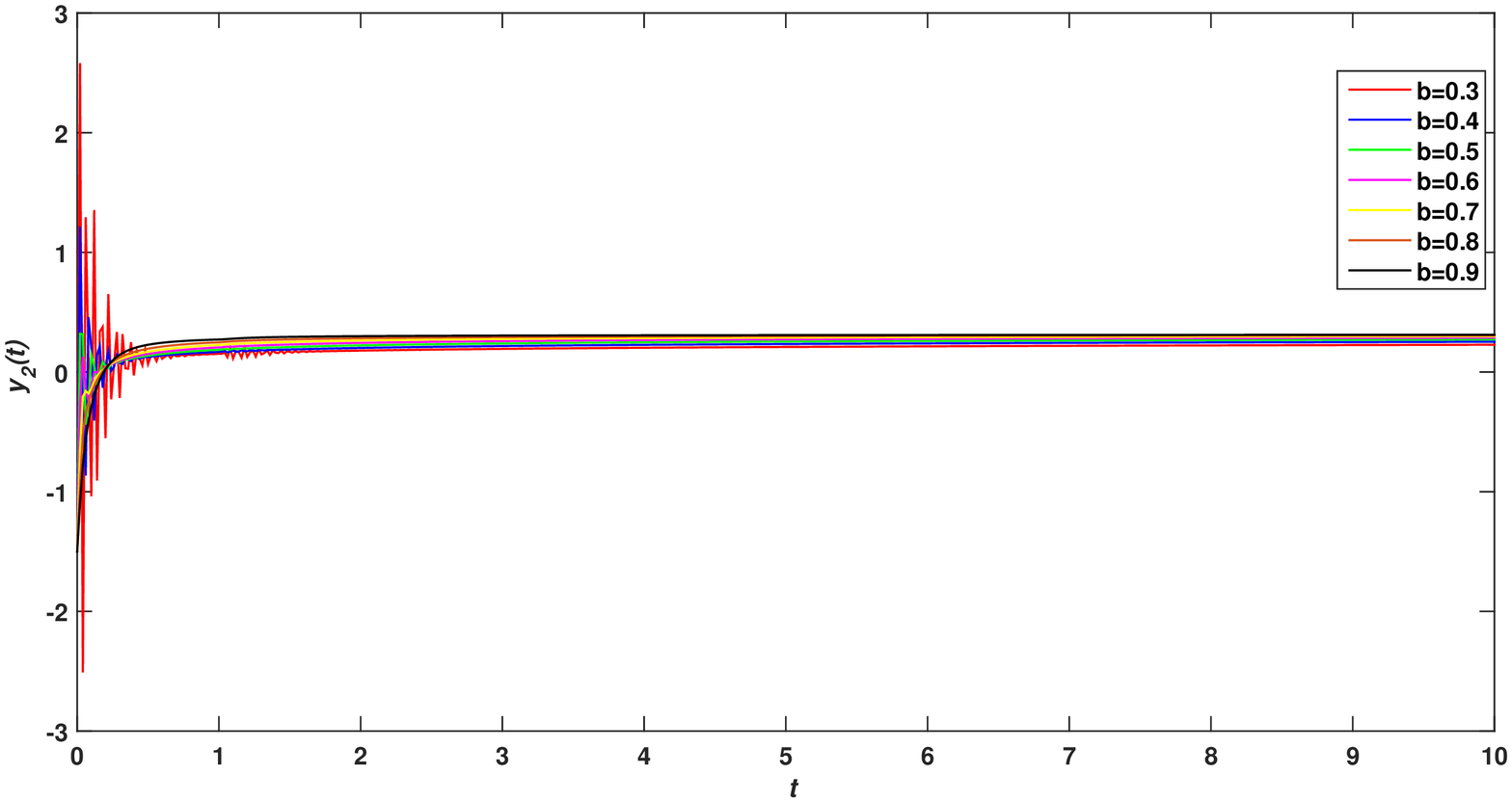}
\caption{Convergence of  solutions of  system (\ref{examp1}) to the stationary state  for different values of $b$.
}
\label{fig3}
\end{figure}

%
%
%
%
\begin{figure}[h!]\centering
              \includegraphics[width=8cm, height=4cm]{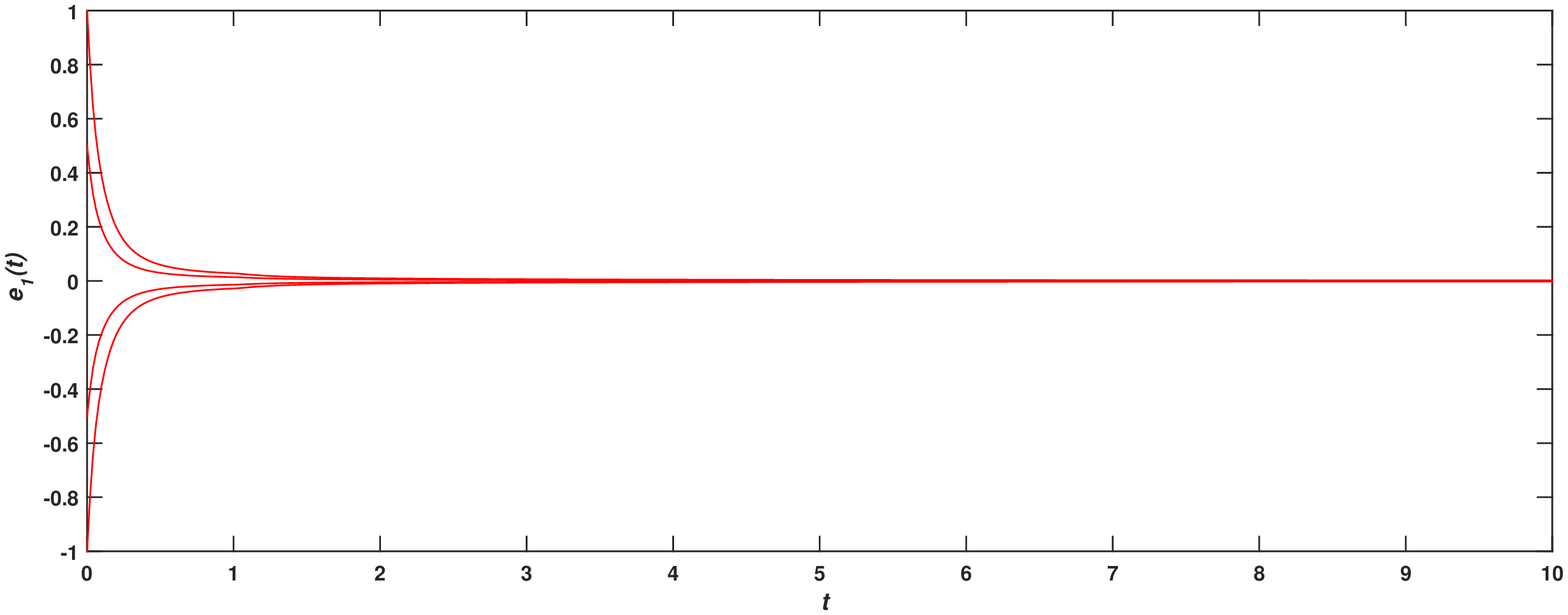}
        \qquad
           \includegraphics[width=8cm, height=4cm]{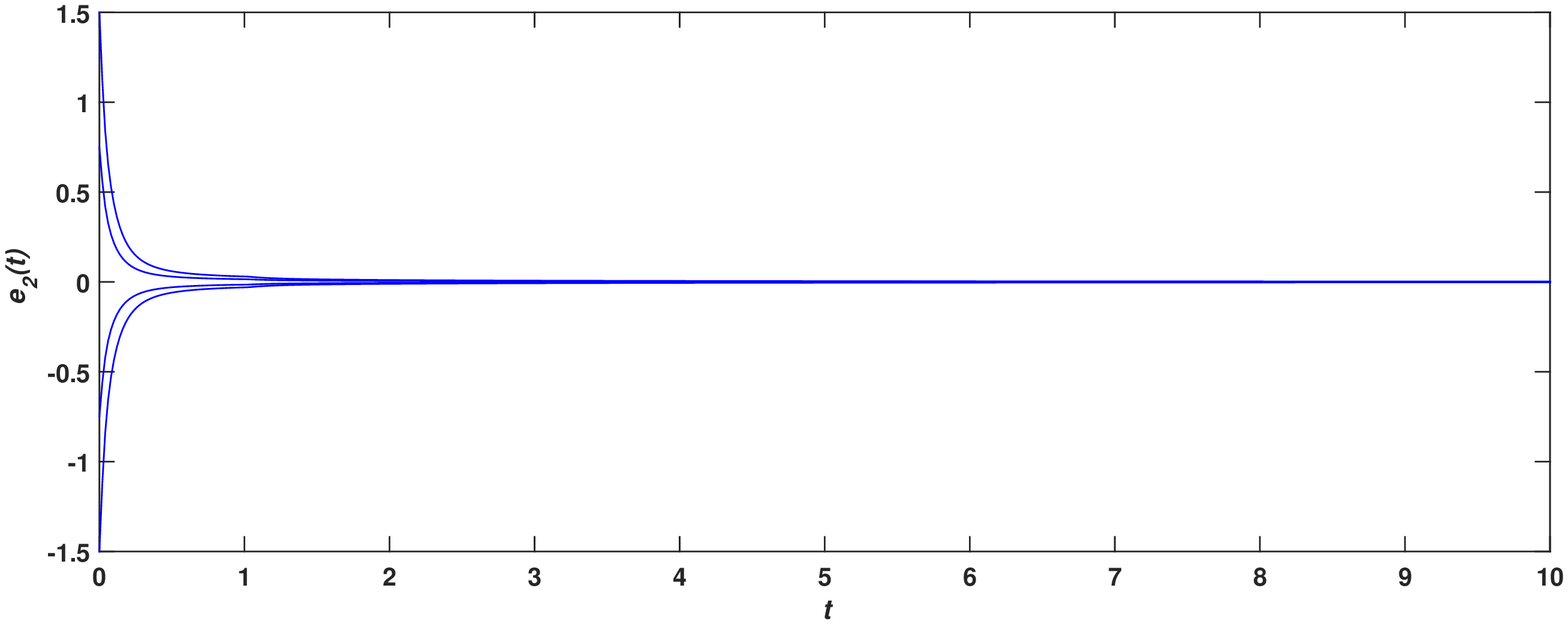}
\\
              \includegraphics[width=8cm, height=4cm]{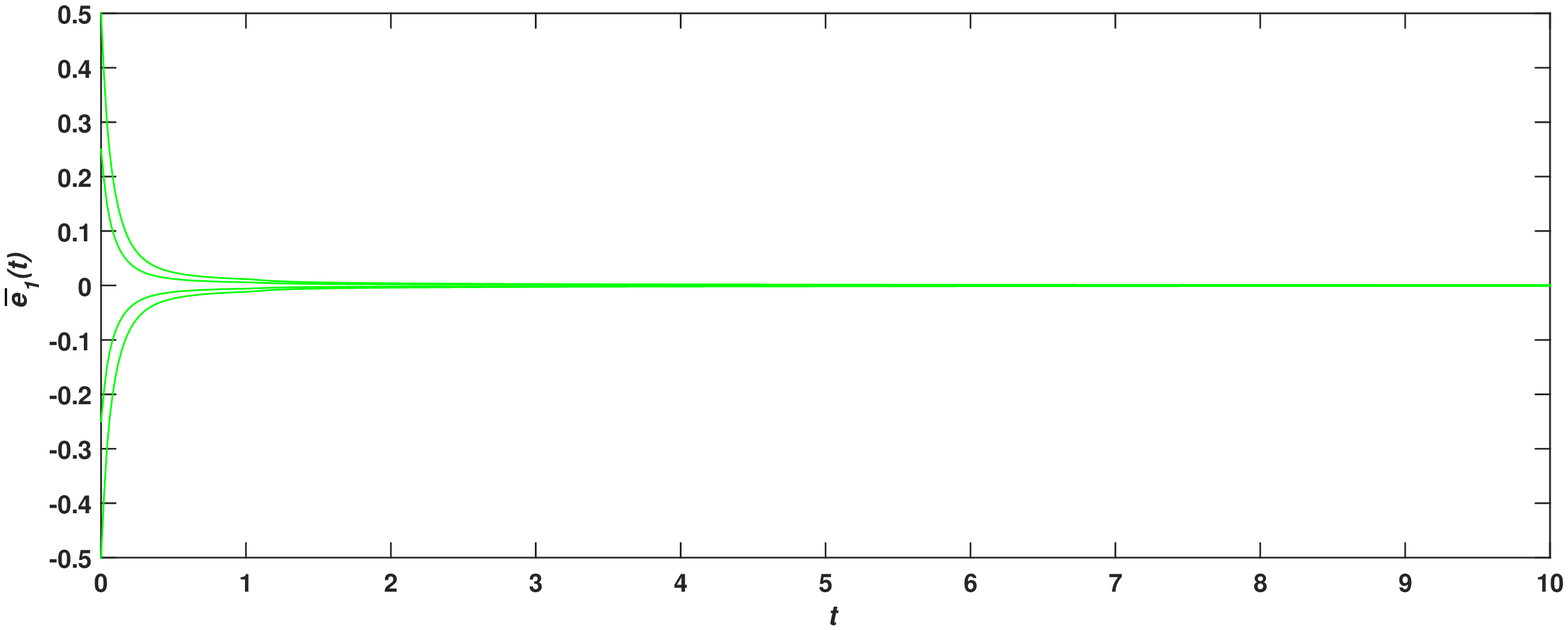}
        \qquad
           \includegraphics[width=8cm, height=4cm]{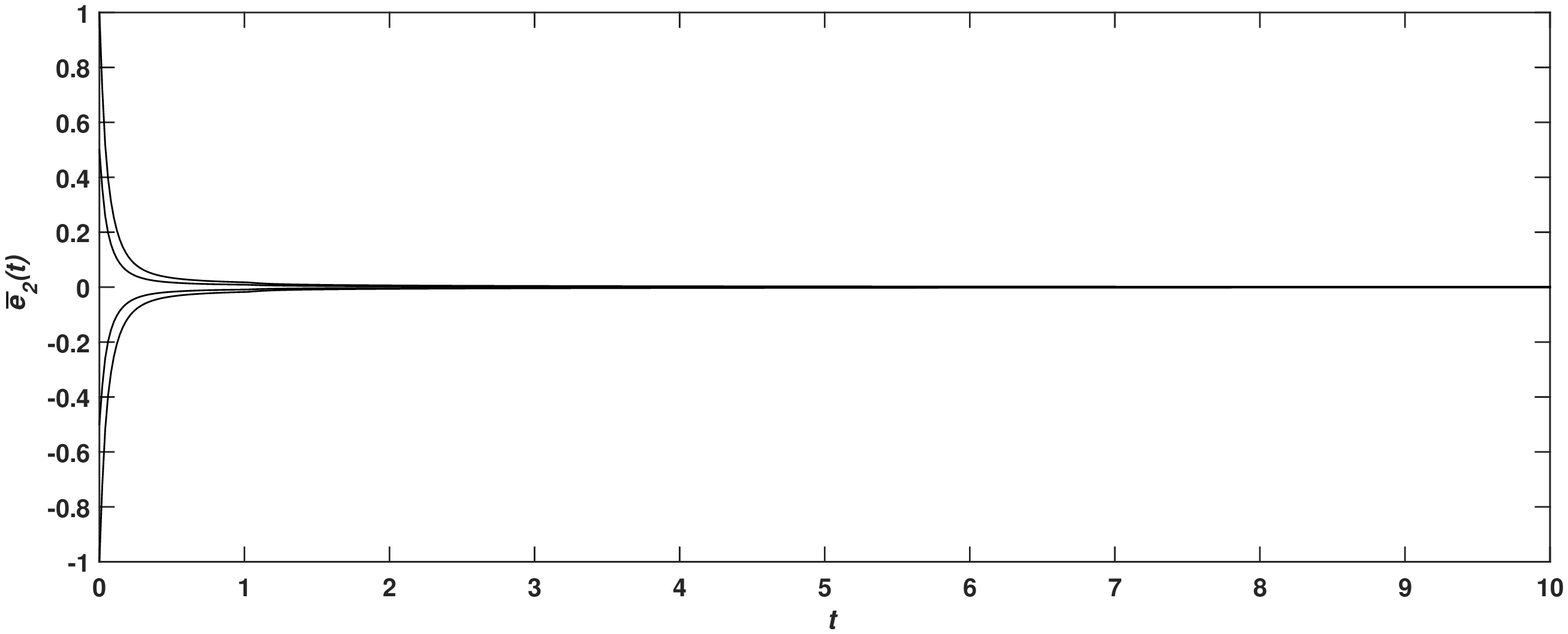}

\caption{Convergence  of  solutions of system (\ref{examp3}) to the stationary state  for various data.
}
\label{fig4}
\end{figure}
\begin{figure} [h!]\centering
              \includegraphics[width=8cm, height=4cm]{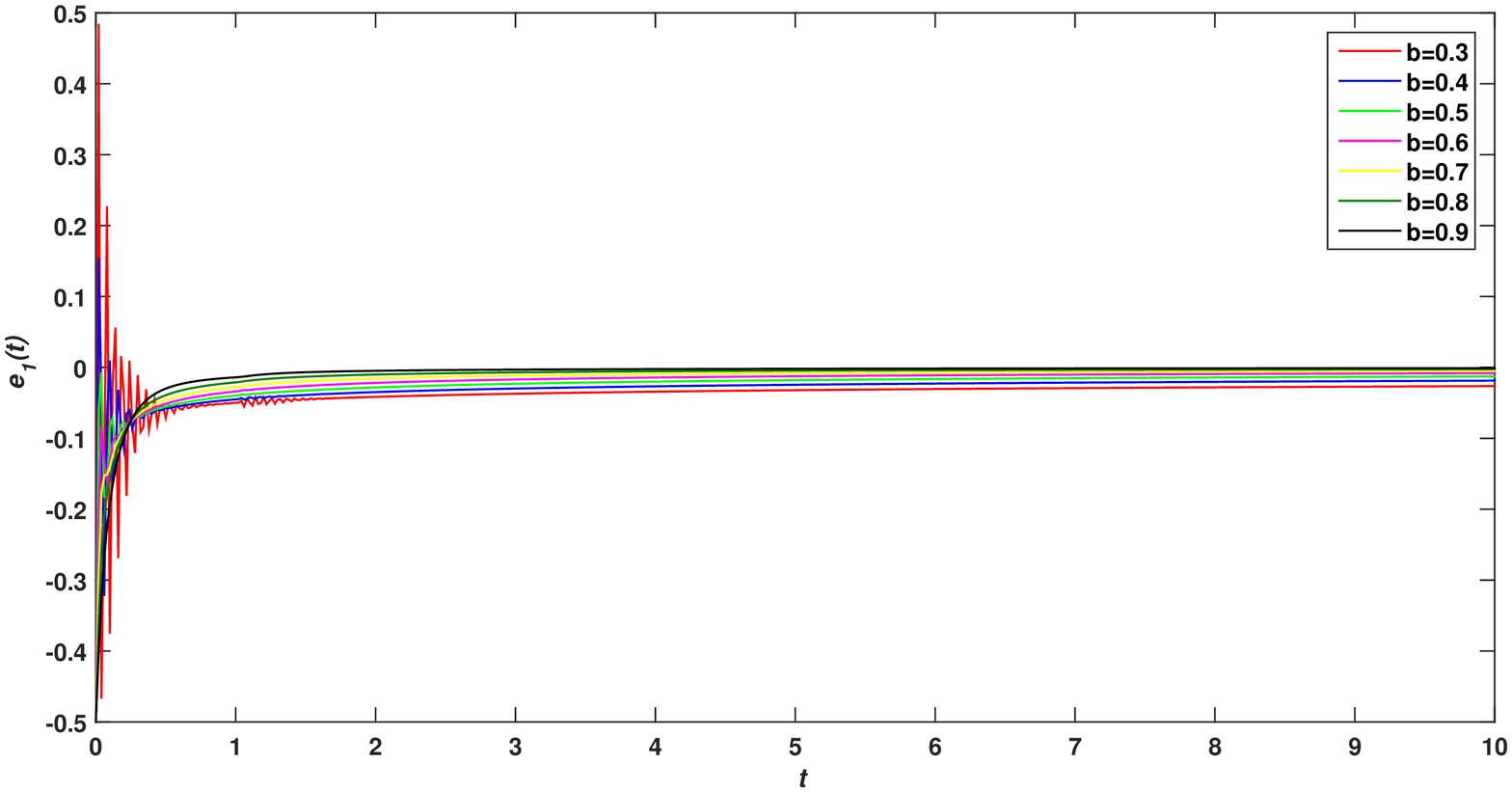}
        \qquad
           \includegraphics[width=8cm, height=4cm]{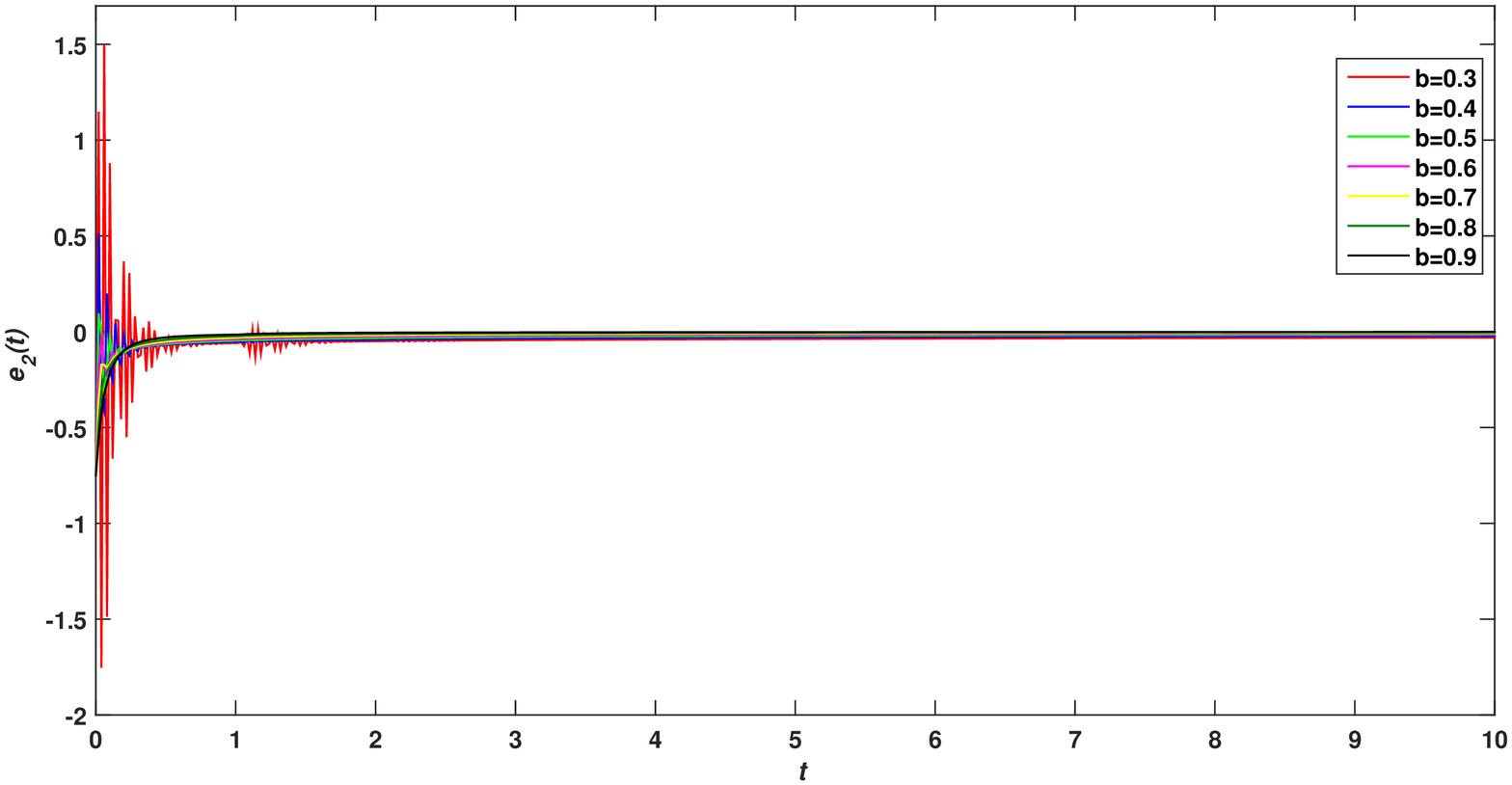}
           \\
              \includegraphics[width=8cm, height=4cm]{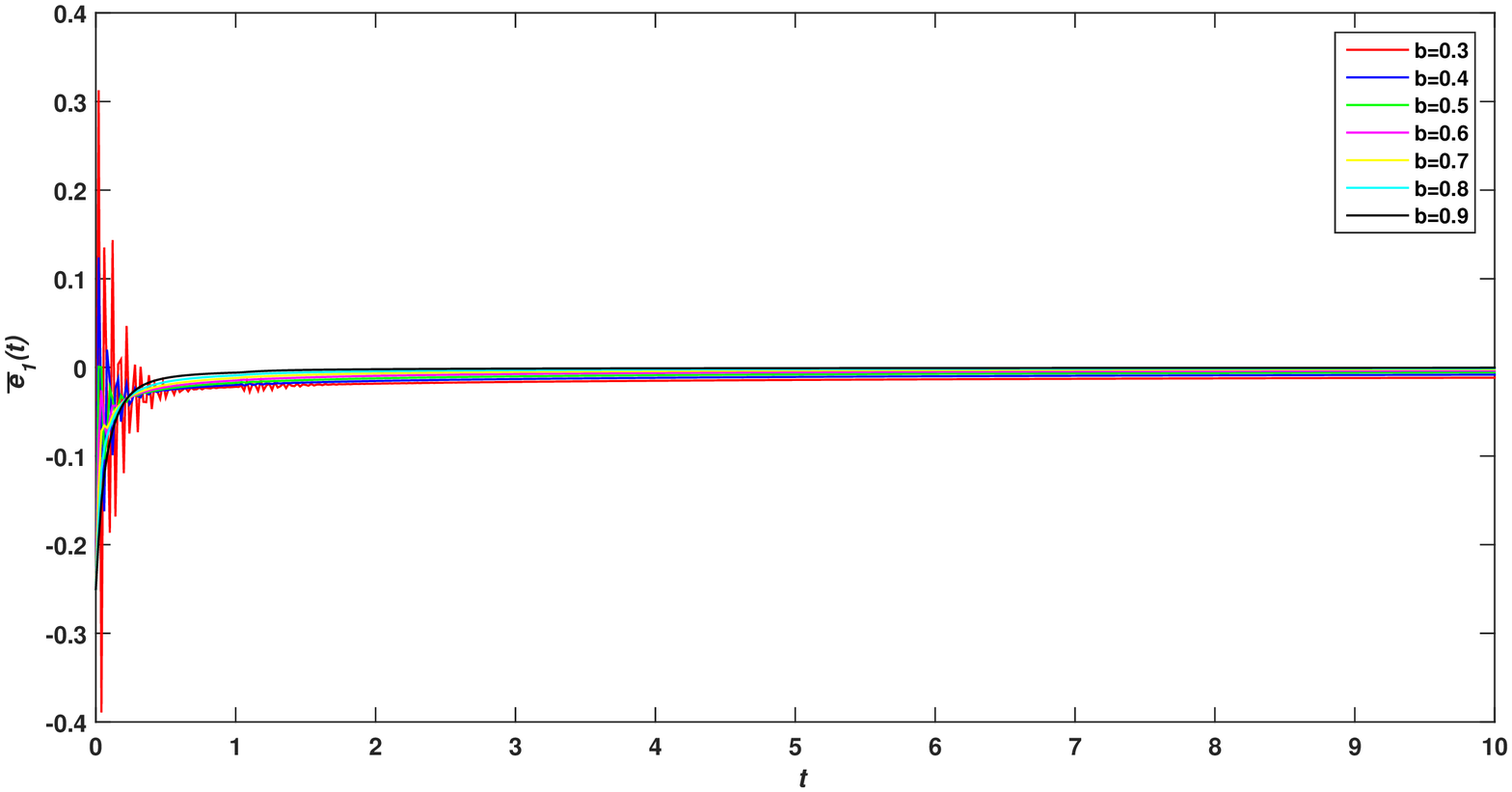}
        \qquad
           \includegraphics[width=8cm, height=4cm]{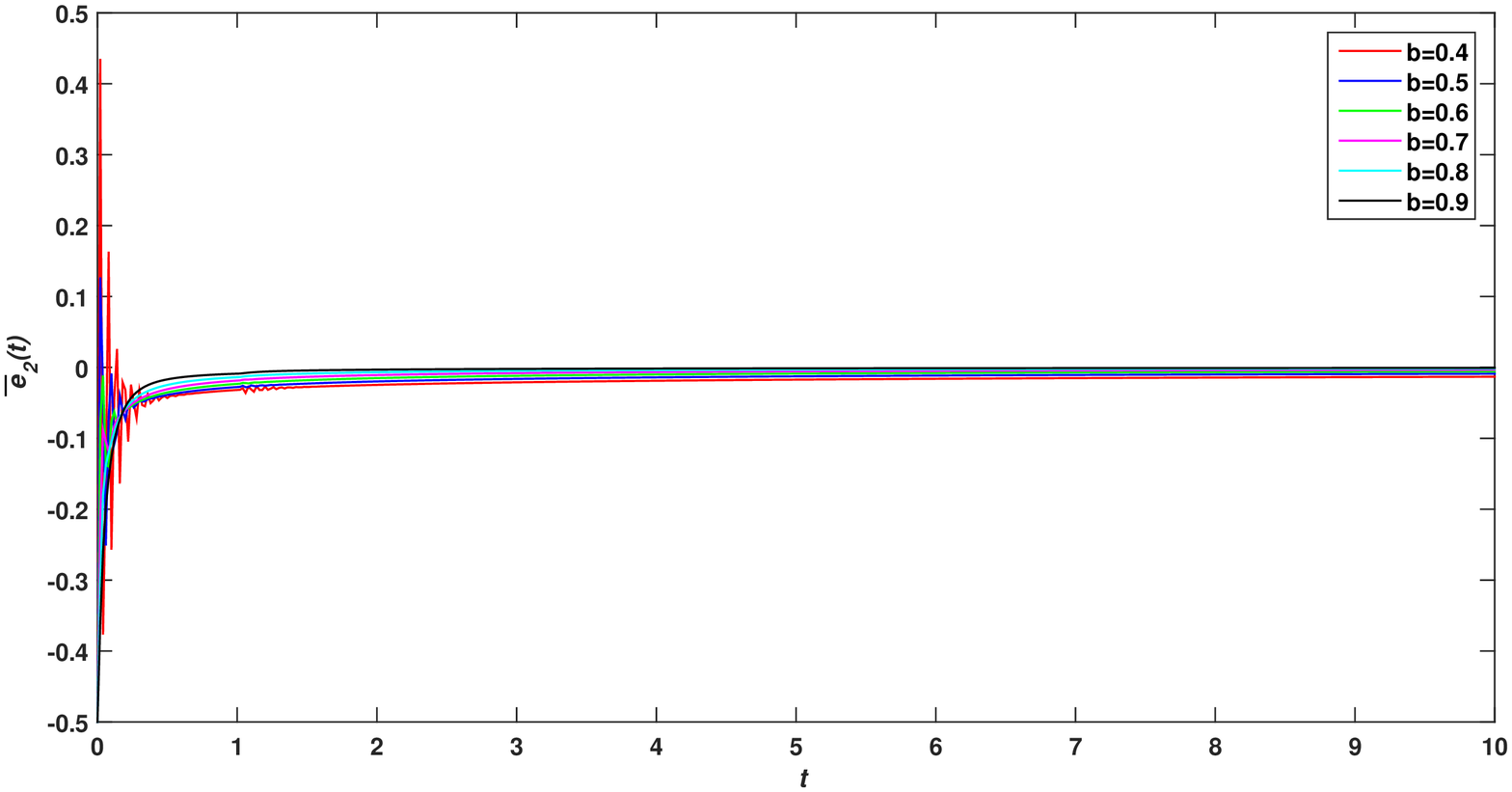}
\caption{Convergence  of  solutions of system (\ref{examp3})  to the stationary state  for different values of $b$.
}
\label{fig5}
\end{figure}
\begin{figure} [tbp]\centering
              \includegraphics[width=.75\textwidth]{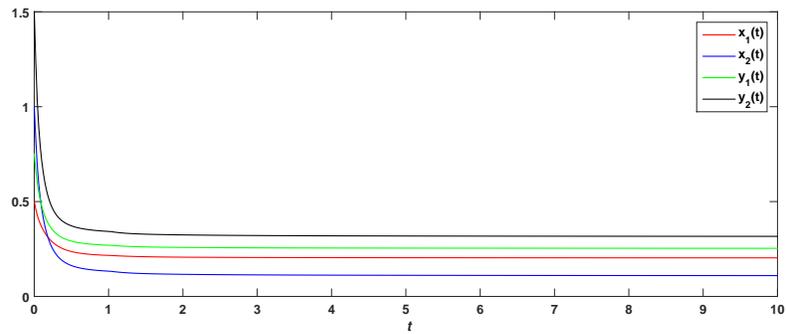}
\caption{Decay  of the solutions  of system (\ref{examp1}) to the stationary state.
}
\label{fig6}
\end{figure}
\begin{figure} [h!]\centering
              \includegraphics[width=8cm, height=4cm]{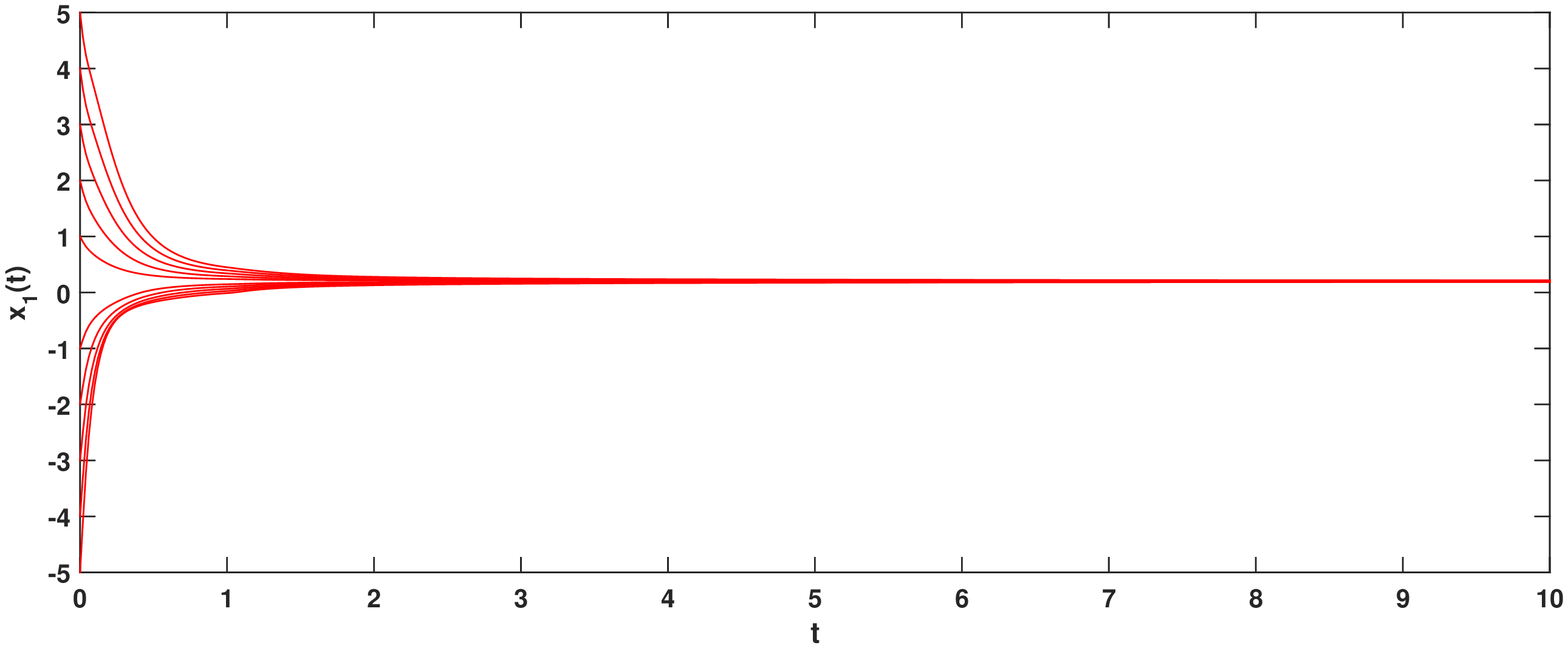}
        \qquad
           \includegraphics[width=8cm, height=4cm]{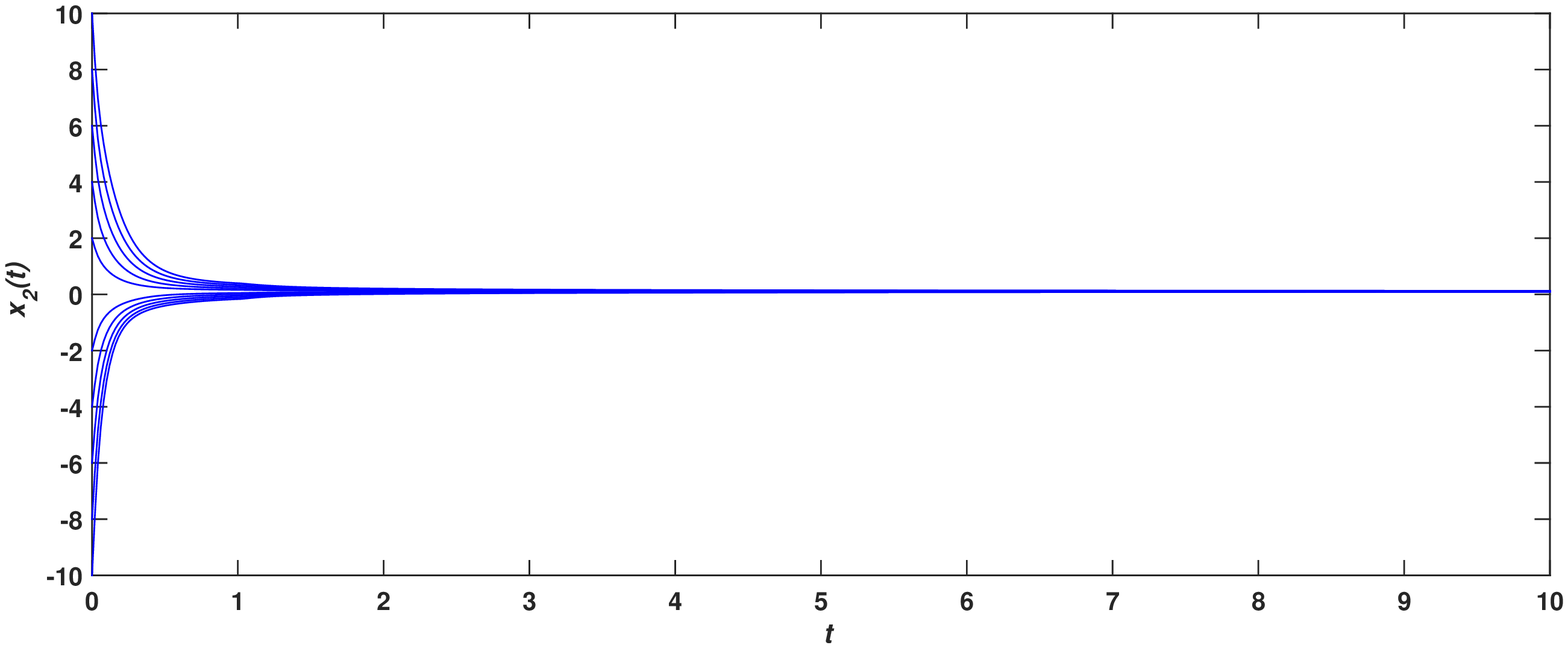}
\\
              \includegraphics[width=8cm, height=4cm]{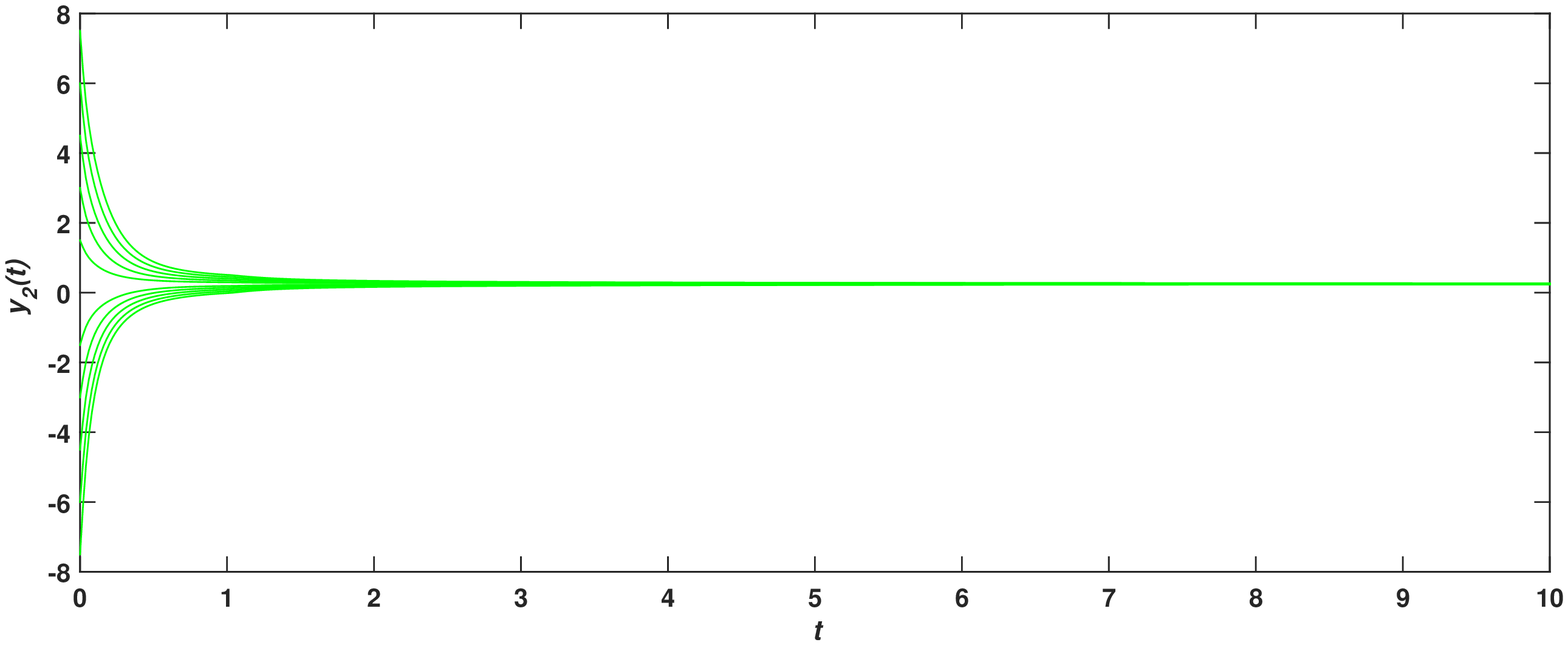}
        \qquad
           \includegraphics[width=8cm, height=4cm]{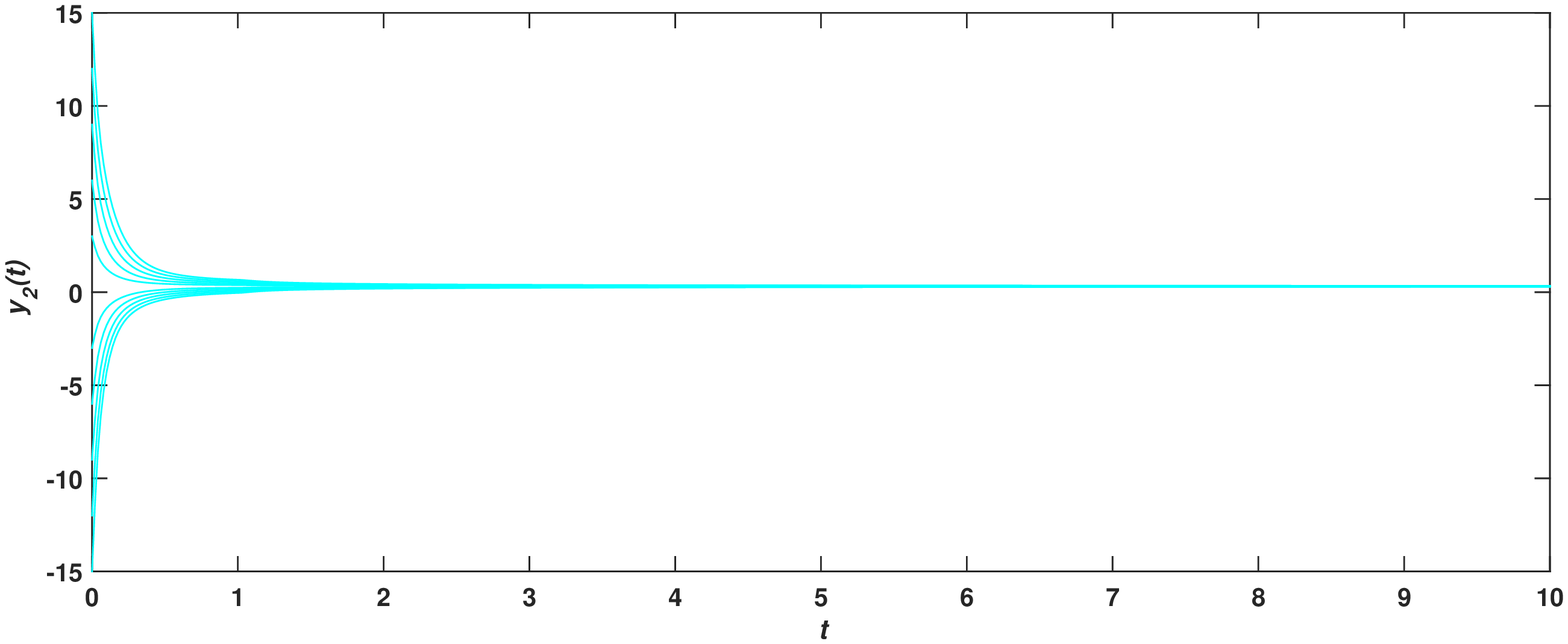}
\caption{Convergence  of   solutions  of  system (\ref{examp1}) to the stationary state  for various data.
}
\label{fig7}
\end{figure}
\newpage
\begin{figure} [h!]\centering
              \includegraphics[width=8cm, height=4cm]{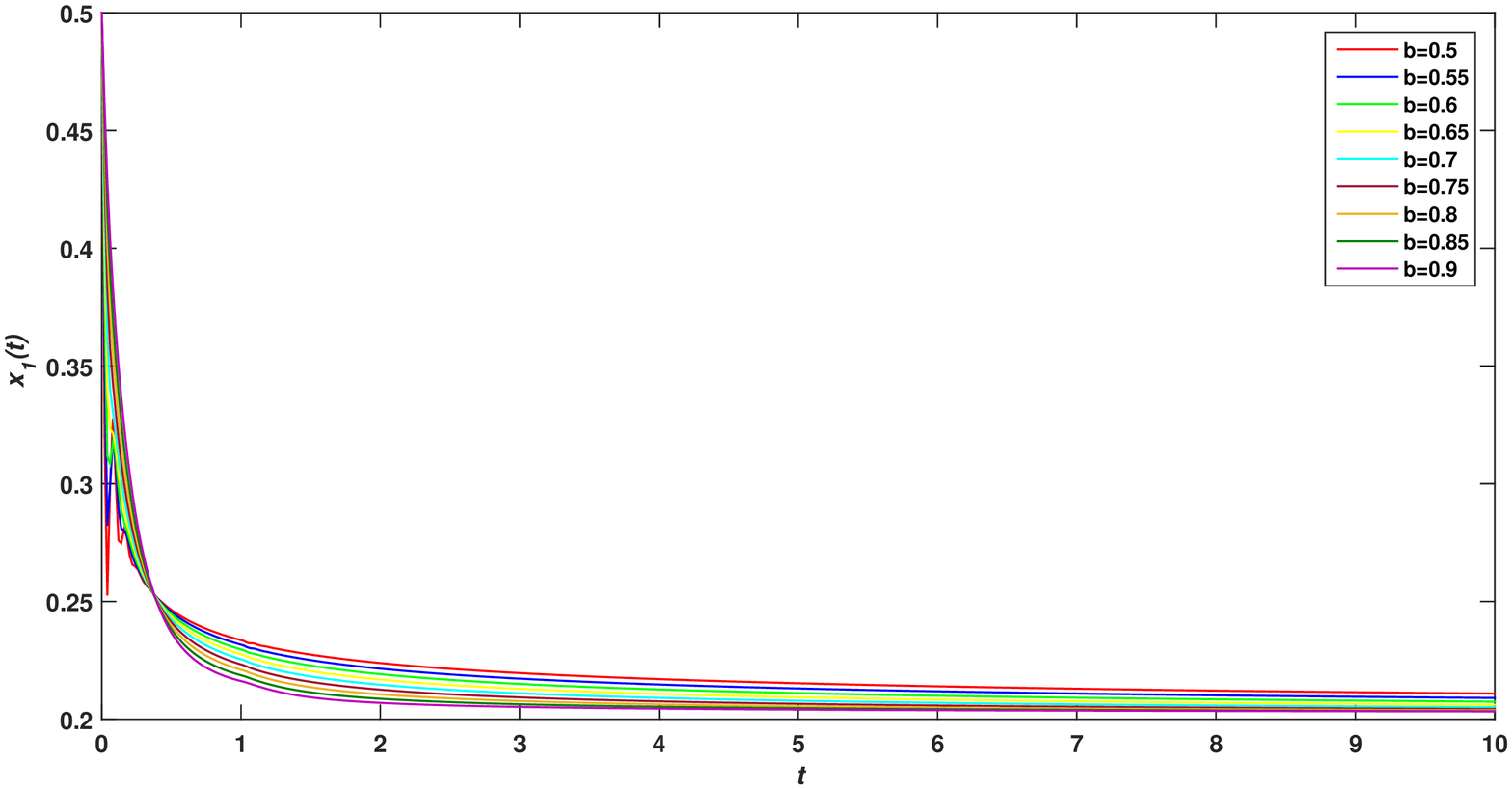}
        \qquad
           \includegraphics[width=8cm, height=4cm]{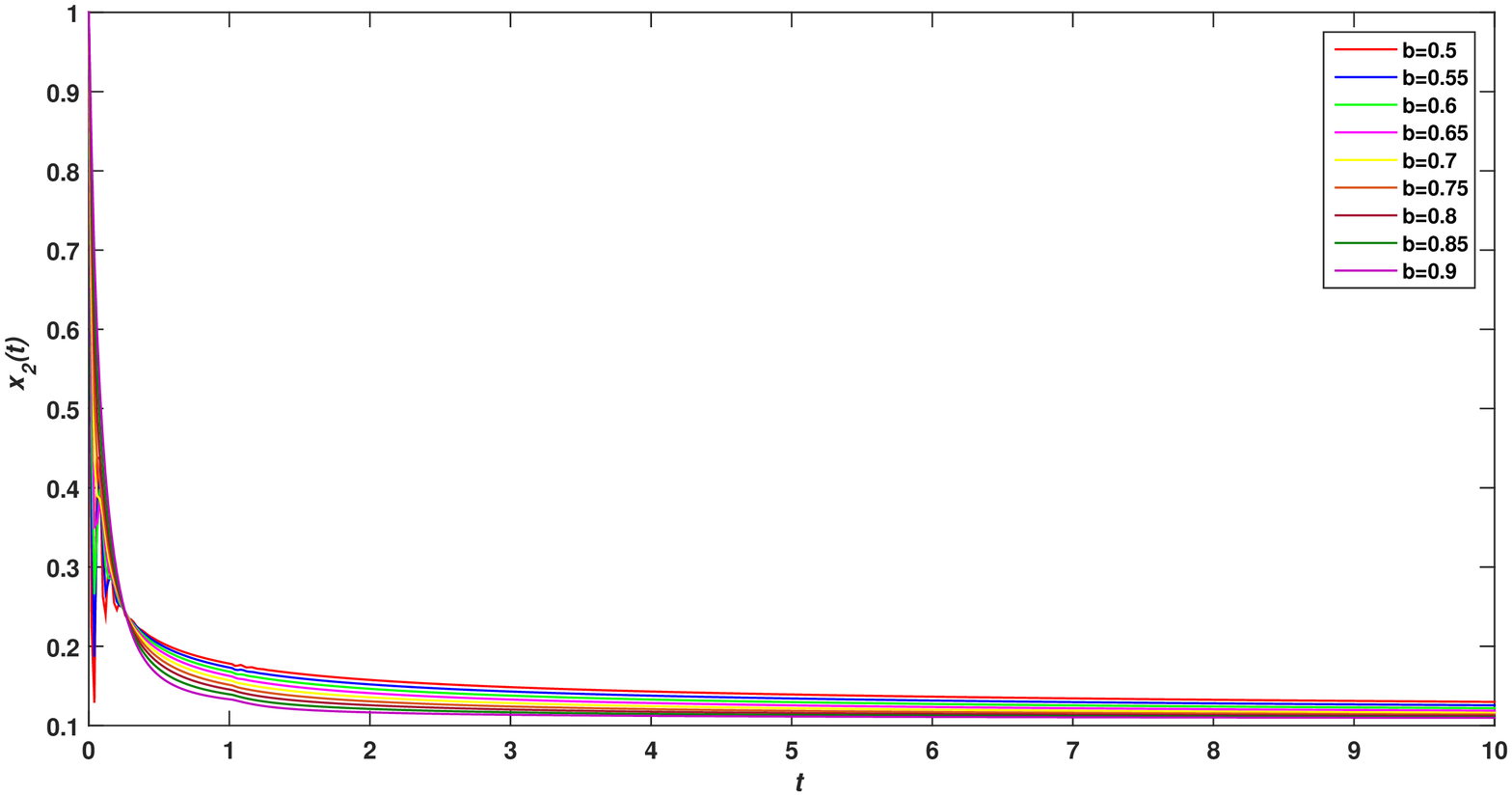}
\\
              \includegraphics[width=8cm, height=4cm]{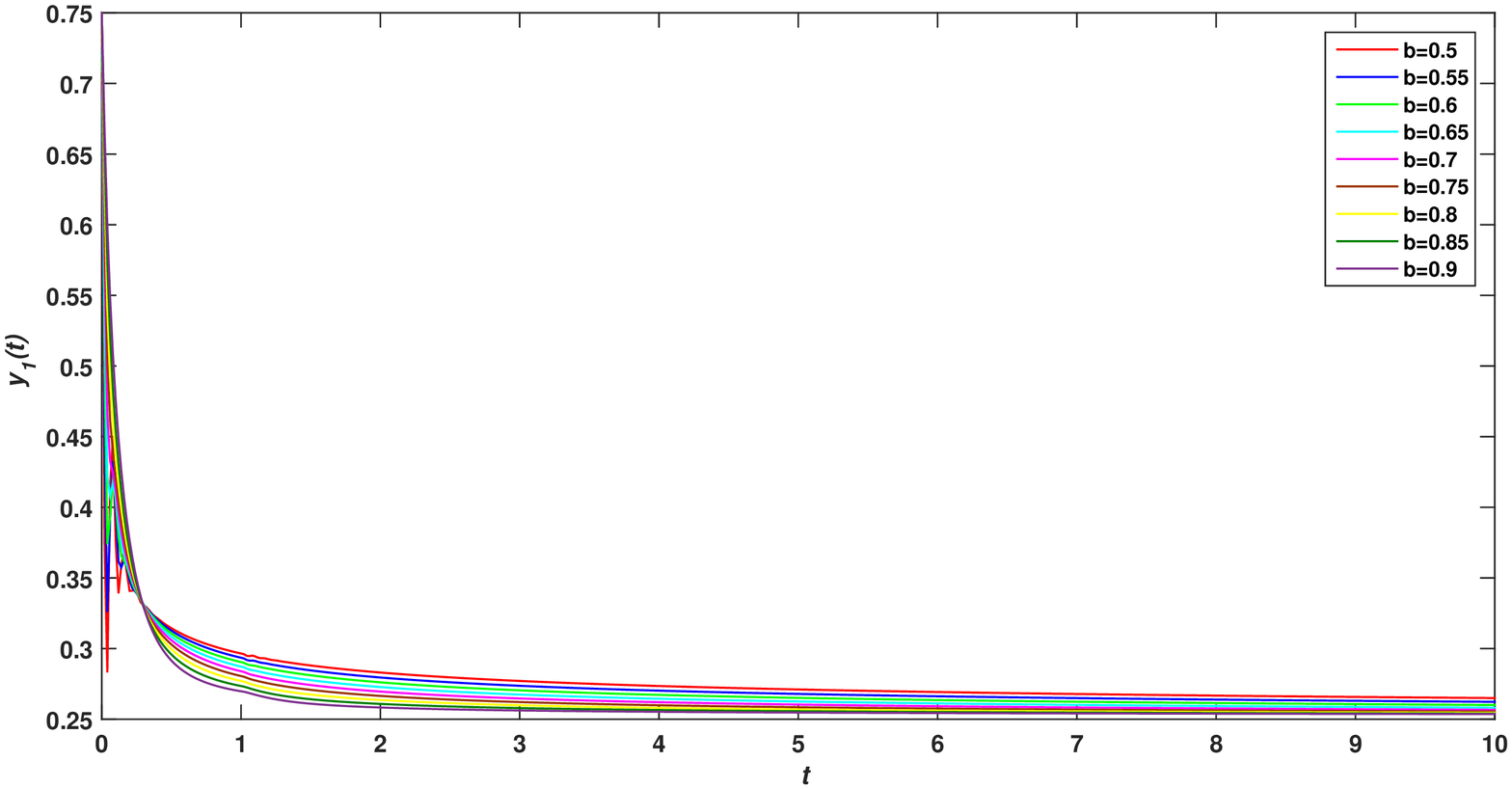}
        \qquad
           \includegraphics[width=8cm, height=4cm]{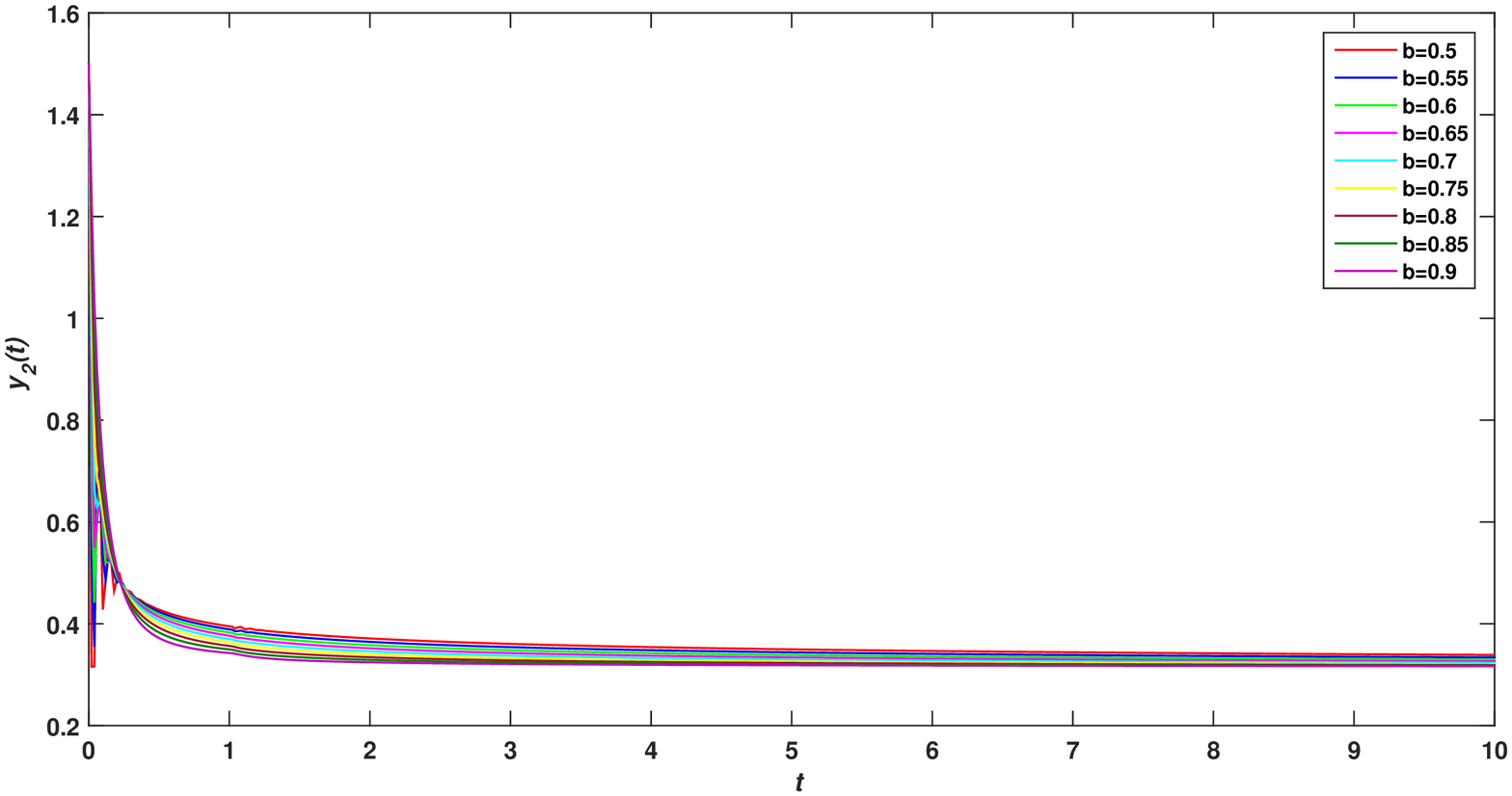}
\caption{Decay of  solutions of  system (\ref{examp1}) to the stationary state for different values of $b$.
}
\label{fig8}
\end{figure}
\begin{figure} [h!]\centering
              \includegraphics[width=8cm, height=4cm]{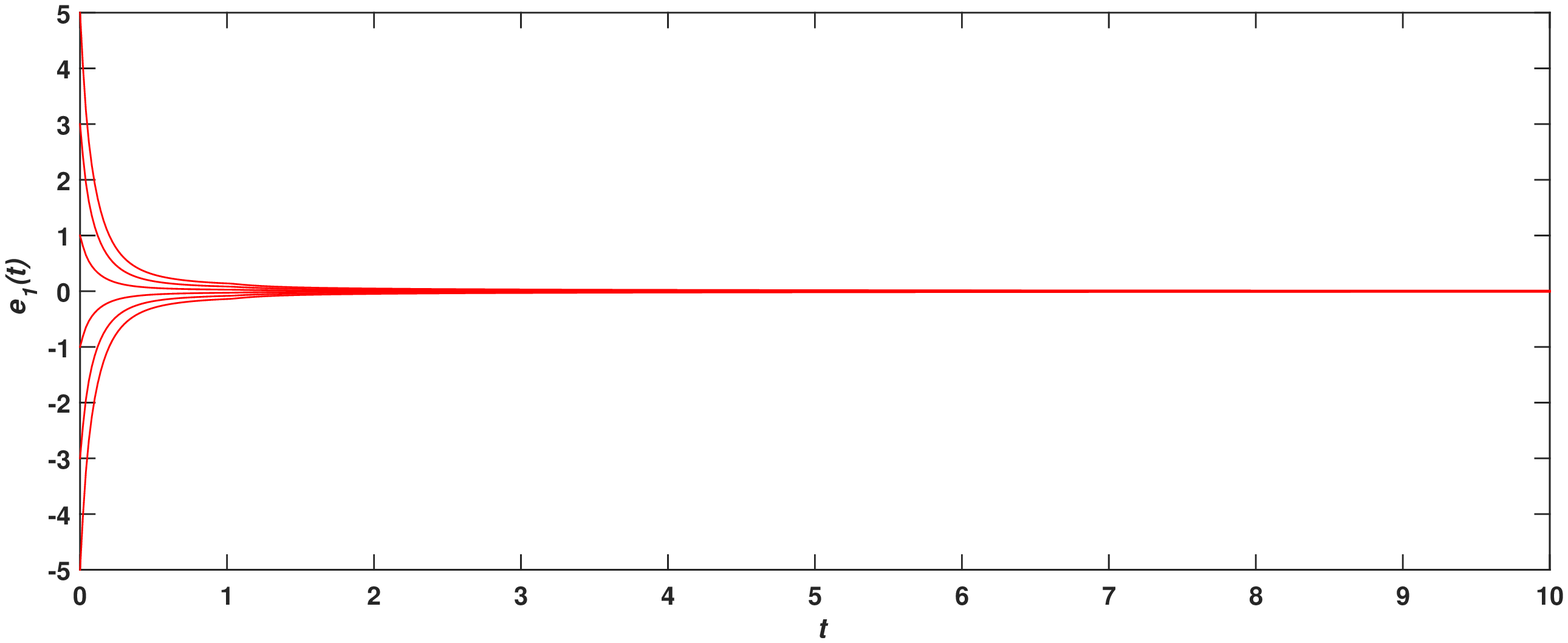}
        \qquad
           \includegraphics[width=8cm, height=4cm]{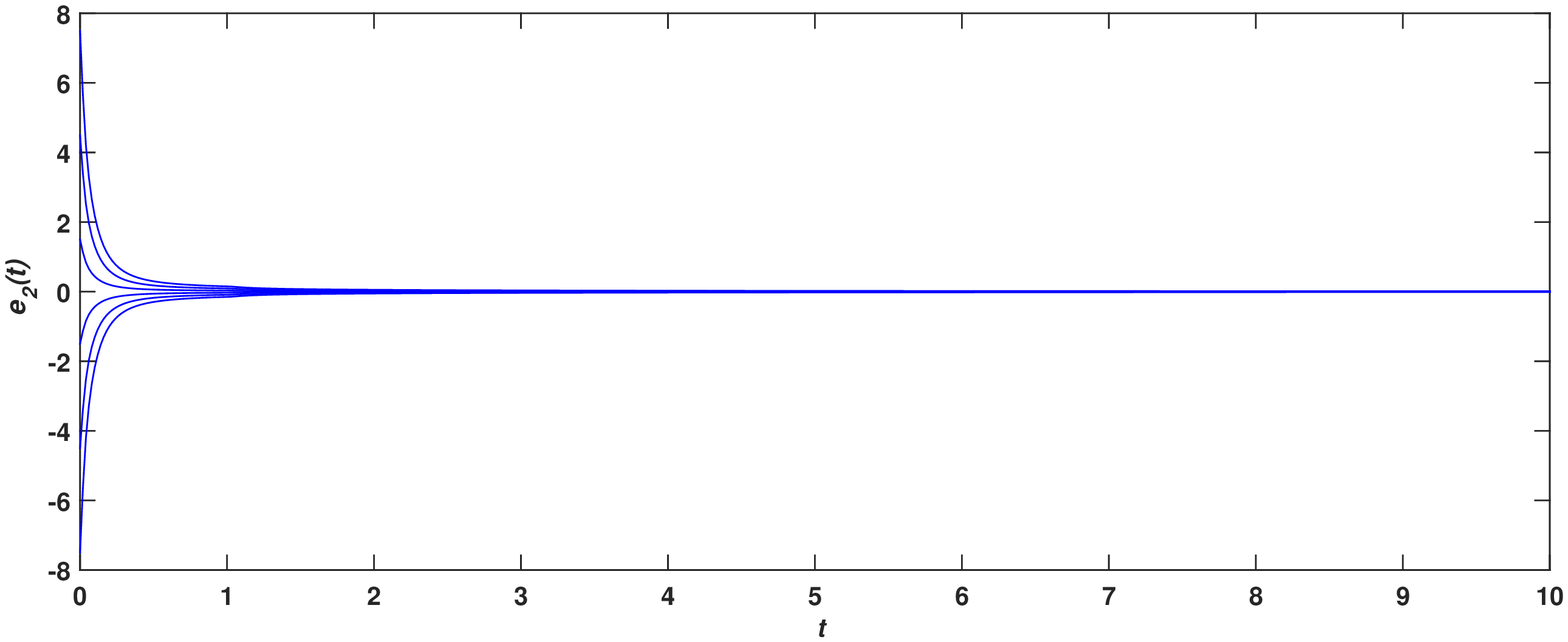}
\\
              \includegraphics[width=8cm, height=4cm]{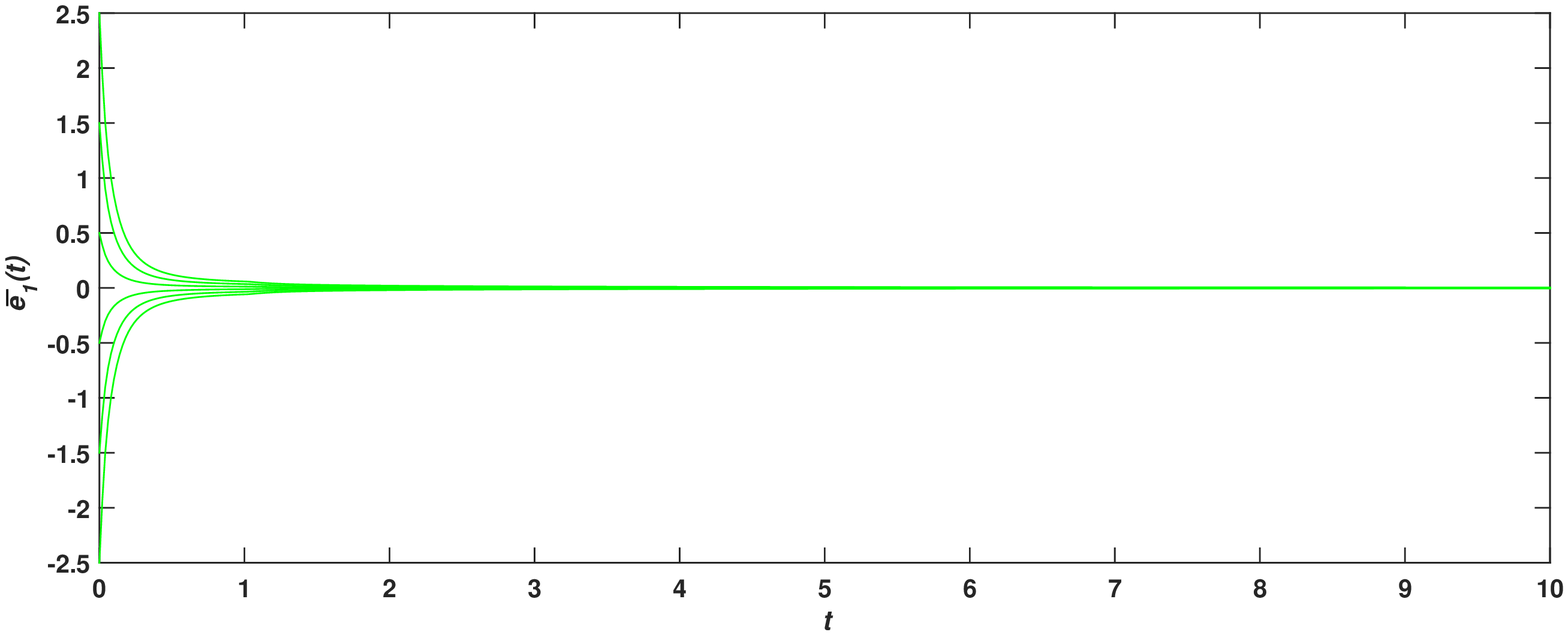}
        \qquad
           \includegraphics[width=8cm, height=4cm]{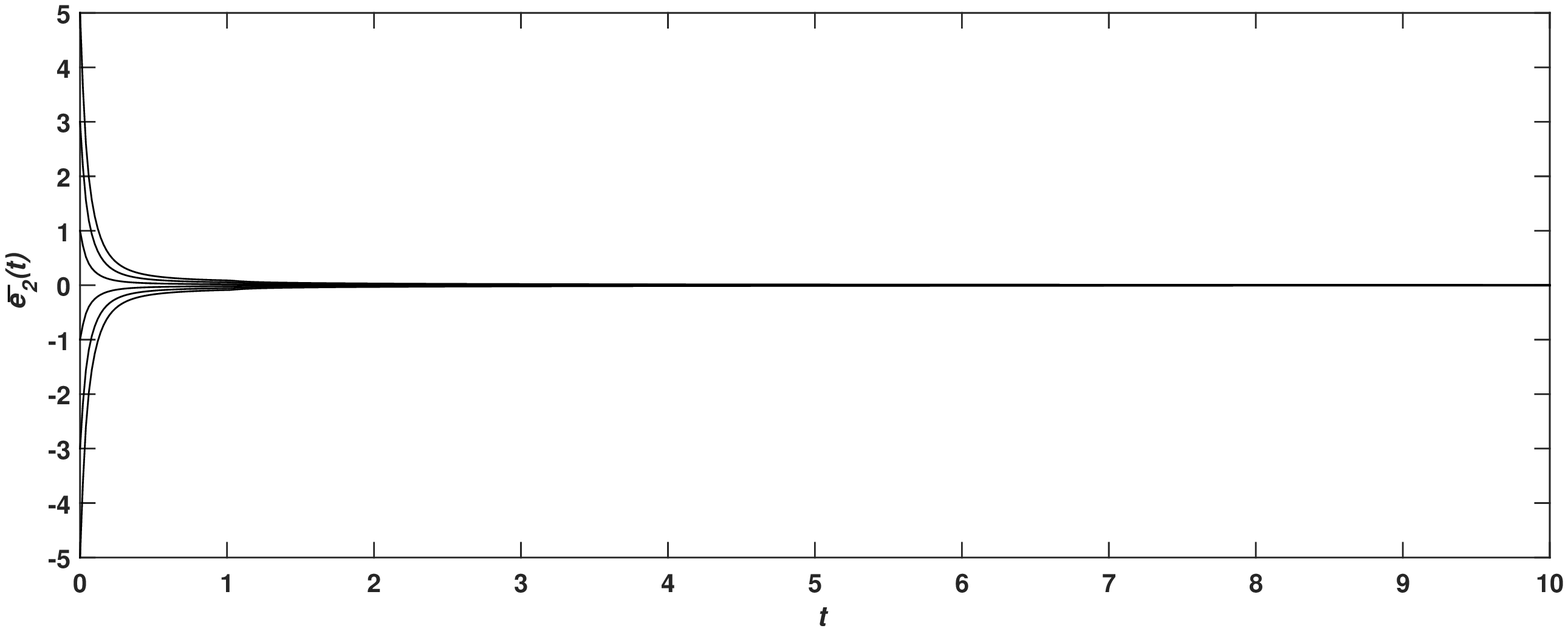}
\caption{Convergence  of  solutions of system (\ref{examp3}) to the stationary state  for various  data.
}
\label{fig9}
\end{figure}
\begin{figure} [h!]\centering
              \includegraphics[width=8cm, height=4cm]{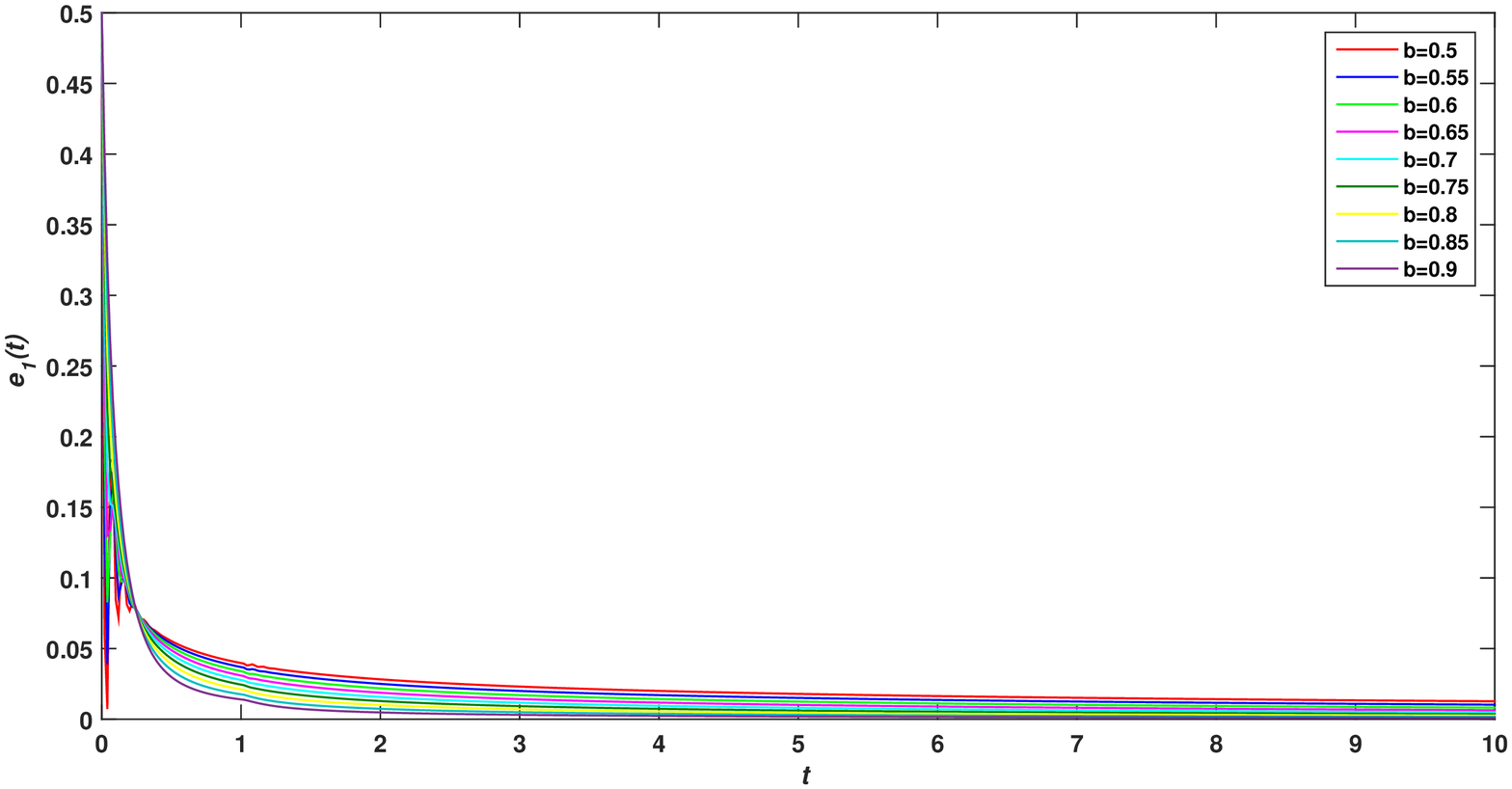}
        \qquad
           \includegraphics[width=8cm, height=4cm]{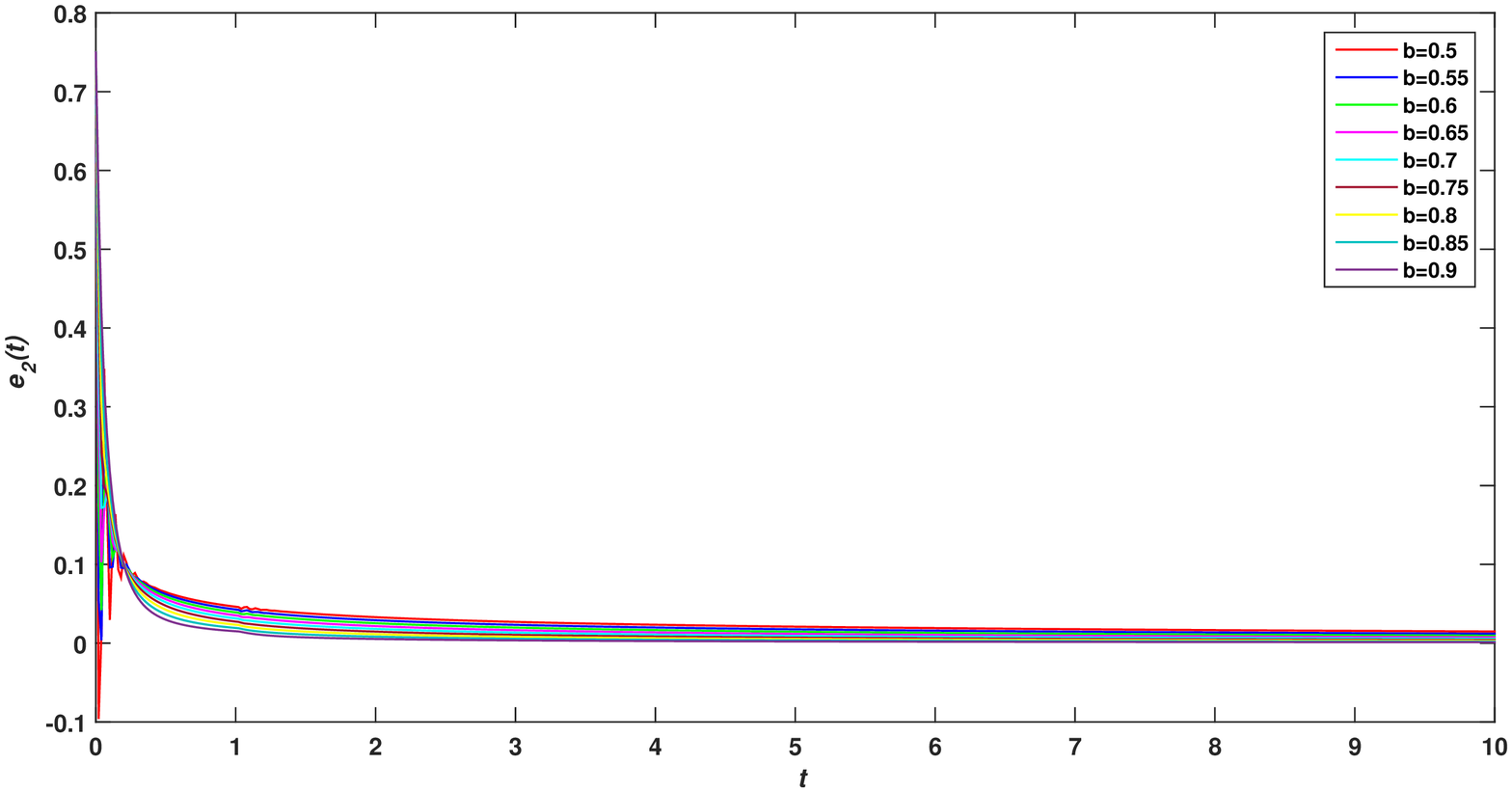}
\\
              \includegraphics[width=8cm, height=4cm]{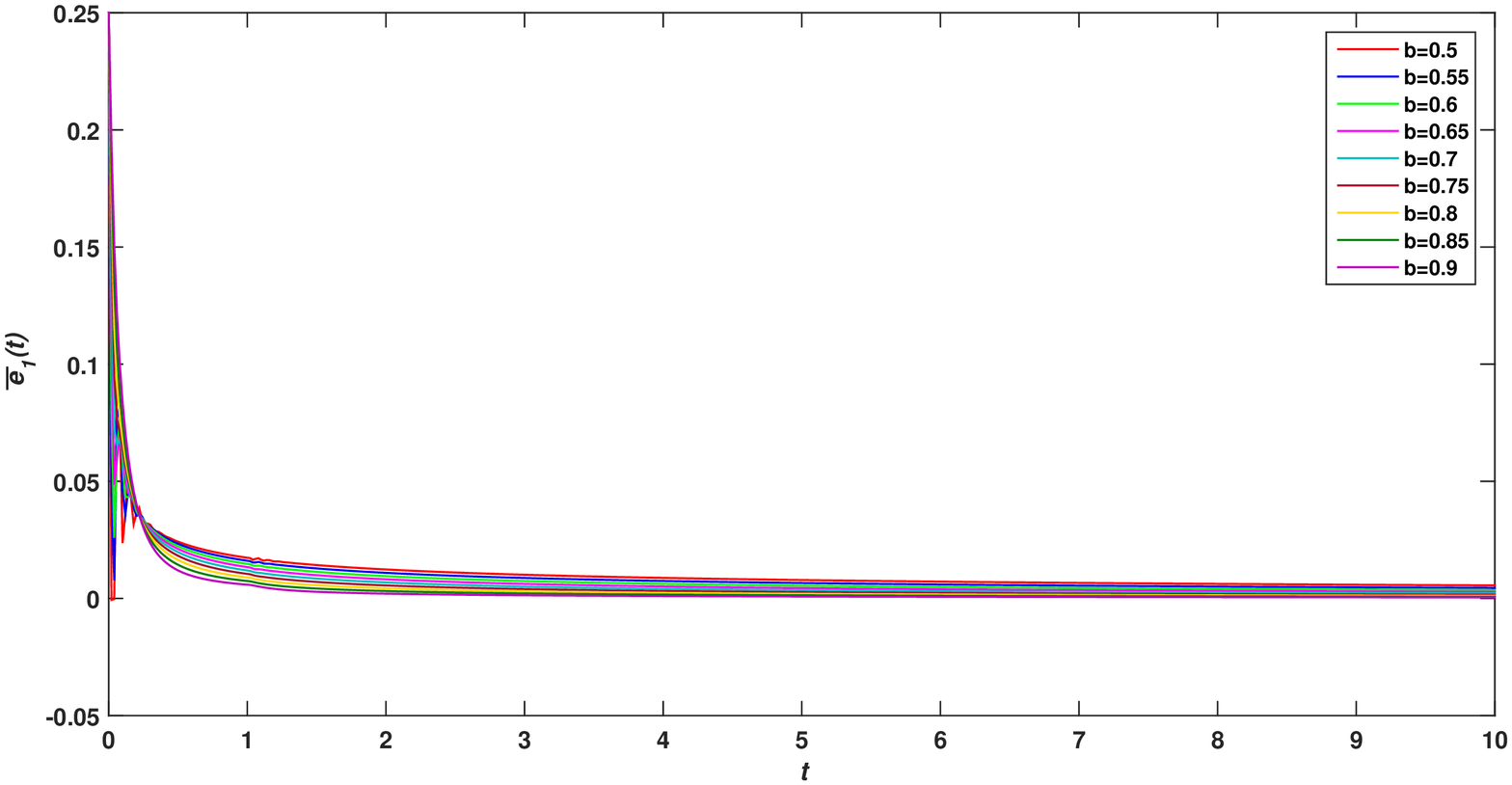}
        \qquad
           \includegraphics[width=8cm, height=4cm]{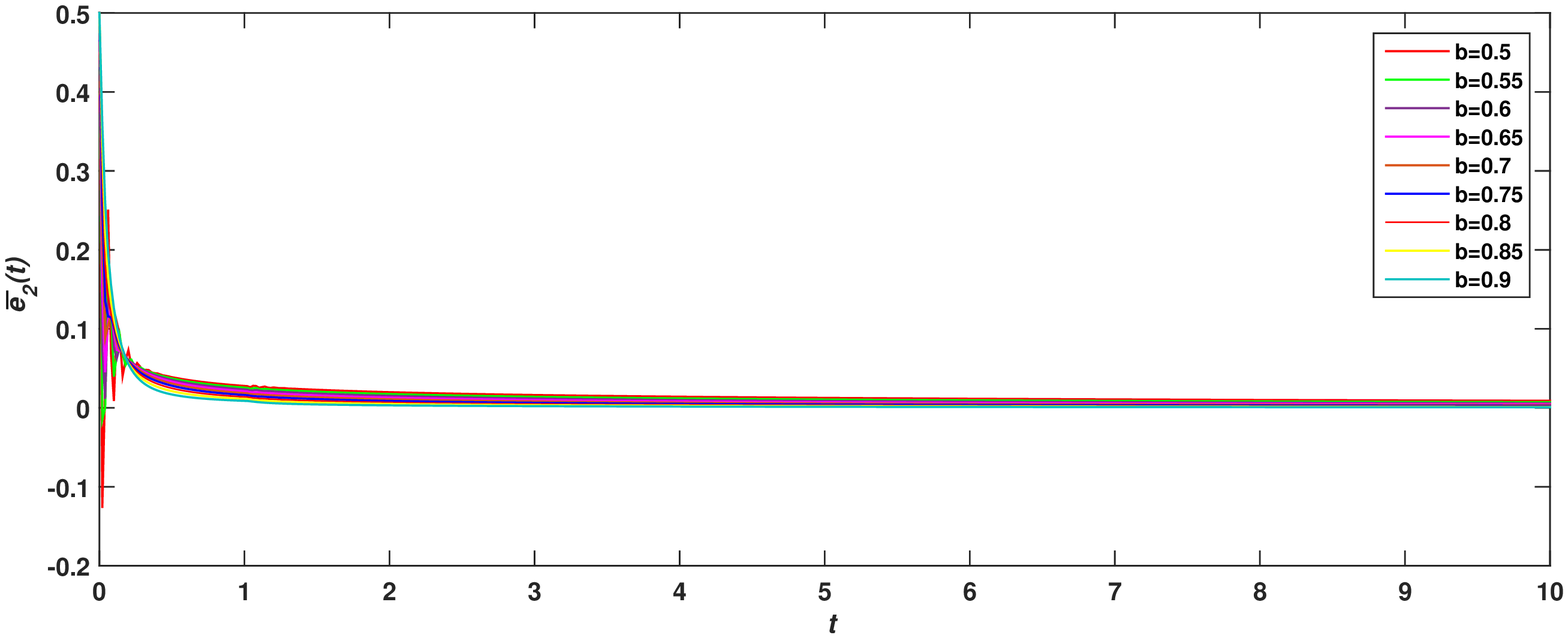}
\caption{Decay  of  solutions of system (\ref{examp3}) to the stationary state  for different values of $b$.
}
\label{fig10}
\end{figure}
\end{document}